\documentclass{amsart}

\usepackage{a4wide}
\usepackage{amscd,amsmath,amssymb}
\usepackage{epsf}
\usepackage{graphicx}
\usepackage{algpseudocode,algorithm}
\usepackage[arrow, matrix, curve]{xy}
\usepackage[english]{babel}

\newtheorem{theorem}{Theorem}[section]
\newtheorem{lemma}[theorem]{Lemma}

\newtheorem{proposition}[theorem]{Proposition}
\newtheorem{corollary}[theorem]{Corollary}

\theoremstyle{definition}
\newtheorem{definition}[theorem]{Definition}

\newtheorem{conjecture}[theorem]{Conjecture}
\newtheorem{question}[theorem]{Question}
\newtheorem{problem}[theorem]{Problem}
\theoremstyle{remark}
\newtheorem{remark}[theorem]{Remark}
\newtheorem{example}[theorem]{Example}

\numberwithin{equation}{section}

\newcommand {\D}  {{\mathcal D}}
\newcommand {\E}  {{\mathcal E}}

\newcommand {\NN}  {{\mathbb N}}

\newcommand {\ZZ}  {{\mathbb Z}}

\newcommand {\RR}  {{\mathbb R}}

\newcommand {\CC}  {{\mathbb C}}

\newcommand {\OR}  {{C}}

\newcommand{\abs}[1]{\left\vert#1\right\vert}

\newcommand{\spk}[1]{\left\langle #1\right\rangle}

\def\interior{{\rm int\,}}
\def\vectz #1 #2 {\left(
\begin{array}{c}
 #1 \\ #2
\end{array}\right)}
\def\vektor #1 #2 {\left(
\begin{array}{c}
 #1 \\ \vdots \\ \vdots \\ #2
\end{array}\right)}
\def\eps{\varepsilon}

\title[Shift Radix Systems]{Shift Radix Systems - A Survey}

\author[P.~Kirschenhofer]{Peter Kirschenhofer}
\author[J.~M.~Thuswaldner]{J\"org~M.~Thuswaldner}
\address{Chair of Mathematics and Statistics,
Montanuniversit\"at Leoben, Franz-Josef-Strasse 18, A-8700 Leoben,
AUSTRIA}
 \email{Peter.Kirschenhofer@unileoben.ac.at}
\email{Joerg.Thuswaldner@unileoben.ac.at}
\thanks{The authors are supported by the Franco-Austrian research project I-1136 ``Fractals and numeration'' granted by the French National Research Agency (ANR) and the Austrian Science Fund (FWF), and by the Doctoral Program W1230 ``Discrete Mathematics'' granted by the Austrian Science Fund (FWF)}

\date{\today}



\begin{document}

\begin{abstract}
Let $d\ge 1$ be an integer and ${\bf r}=(r_0,\dots,r_{d-1}) \in \RR^d$.
The {\em shift radix system} $\tau_\mathbf{r}:
\mathbb{Z}^d \to \mathbb{Z}^d$ is defined by
$$
\tau_{{\bf r}}({\bf z})=(z_1,\dots,z_{d-1},-\lfloor {\bf r} {\bf z}\rfloor)^t \qquad ({\bf z}=(z_0,\dots,z_{d-1})^t).
$$
$\tau_\mathbf{r}$ has the {\em finiteness
property} if each ${\bf z} \in \mathbb{Z}^d$ is eventually mapped to
${\bf 0}$ under iterations of $\tau_\mathbf{r}$.
In the present survey we summarize results on these nearly linear mappings. We discuss how these mappings are related to well-known numeration systems, to rotations with round-offs, and to a conjecture on periodic expansions w.r.t.\ Salem numbers. Moreover, we review the behavior of the orbits of points under iterations of $\tau_\mathbf{r}$ with special emphasis on ultimately periodic orbits and on the finiteness property. We also describe a geometric theory related to shift radix systems.\vspace{-0.6cm}
\end{abstract}
\maketitle

\tableofcontents

\section{Introduction}

{\it Shift Radix Systems} and their relation to beta-numeration seem to have appeared first in Hollander's PhD thesis~\cite{Hollander:96} from 1996. Already in 1994 Vivaldi~\cite{Vivaldi:94} studied similar mappings in order to investigate rotations with round-off errors.  In 2005 Akiyama {\it et al.}~\cite{Akiyama-Borbeli-Brunotte-Pethoe-Thuswaldner:05} introduced the notion of shift radix system formally and elaborated the connection of  these simple dynamical systems to several well-known notions of number systems such as beta-numeration and canonical number systems. We recall the definition of these objects (here for $y\in\mathbb{R}$ we denote by $\lfloor y \rfloor$ the largest $n\in\mathbb{Z}$ with $n\le y$; moreover, we set $\{y\} = y - \lfloor y \rfloor$).

\begin{definition}[Shift radix system]\label{def:SRS}
Let $d\ge 1$ be an integer and ${\bf r}=(r_0,\dots,r_{d-1}) \in \RR^d$. Then we define the {\em shift radix system} (SRS by short) to be the following mapping $\tau_{{\bf r}}\; : \; \ZZ^d \to \ZZ^d$: For ${\bf z}=(z_0,\dots,z_{d-1})^t \in \ZZ^d$ let
\begin{equation}\label{SRSdefinition}
\tau_{{\bf r}}({\bf z})=(z_1,\dots,z_{d-1},-\lfloor {\bf r} {\bf z}\rfloor)^t,
\end{equation}
where ${\bf r} {\bf z}=r_0z_0 +\dots + r_{d-1}z_{d-1}$.

If for each $ {\bf z} \in \ZZ^d$ there is $k\in\mathbb{N}$ such that the $k$-fold iterate of the application of $\tau_{\mathbf{r}}$ to $\mathbf{z}$ satisfies $\tau_{{\bf r}}^k({\bf z})={\bf 0}$ we say that  $\tau_\mathbf{r}$ has the {\em finiteness property}.
\end{definition}

It should be noticed that the definition of SRS differs in literature. Our definition agrees with the one in \cite{BSSST2011}, but the SRS in \cite{Akiyama-Borbeli-Brunotte-Pethoe-Thuswaldner:05} coincide with our SRS with finiteness property.

Equivalently to the above definition we might state that $\tau_{{\bf r}}({\bf z})=(z_1,\dots,z_{d-1},z_{d})^t,$ where $z_{d}$ is the unique integral solution of the linear inequality
\begin{equation}\label{linearinequalities}
0 \le r_0z_0+\cdots + r_{d-1}z_{d-1} + z_{d} < 1.
\end{equation}
Thus, the investigation of shift radix systems has natural relations to the study of {\it almost linear recurrences} and {\em linear Diophantine inequalities}.

Shift radix systems have many remarkable properties and admit relations to several seemingly unrelated objects studied independently in the past. In the present paper we will survey these properties and relations. In particular, we will emphasize on the following topics.

\begin{itemize}
\item For an algebraic integer $\beta > 1$ the {\em beta-transformation} $T_\beta
$ is conjugate to $\tau_\mathbf{r}$ for a parameter $\mathbf{r}$ that is defined in terms of $\beta$. Therefore, the well-known {\it beta-expansions} (R\'enyi \cite{Renyi:57}, Parry \cite{Parry:60}) have a certain finiteness property called {\em property (F)}  ({\it cf.} Frougny and Solomyak \cite{Frougny-Solomyak:92}) if and only if the related $\tau_\mathbf{r}$ is an SRS with finiteness property. Pisot numbers $\beta$ are of special importance in this context.
\item The {\em backward division mapping} used to define {\it canonical number systems} is conjugate to $\tau_\mathbf{r}$ for certain parameters $\mathbf{r}$. For this reason, characterizing all bases of canonical number systems is a special
case of describing all vectors $\mathbf{r}$ giving rise to SRS with finiteness property ({\it cf.} Akiyama {\it et al.} \cite{Akiyama-Borbeli-Brunotte-Pethoe-Thuswaldner:05}).
\item The {\em Schur-Cohn region} (see Schur~\cite{Schur:18}) is the set of all vectors $\mathbf{r}=(r_0,\ldots,r_{d-1}) \in \mathbb{R}^d$ that define a contractive polynomial $X^d + r_{d-1}X^{d-1} + \cdots + r_1X+r_0$. This region is intimately related to the set of all parameters
$\mathbf{r}$ for which each orbit of $\tau_\mathbf{r}$ is ultimately periodic.
\item Vivaldi~\cite{Vivaldi:94} started to investigate discretized rotations which are of interest in computer science because they can be performed in integer arithmetic. A fundamental problem is to decide whether their orbits are periodic. It turns out that discretized rotations are special cases of shift radix systems and their periodicity properties have close relations to the conjecture of Bertrand and Schmidt mentioned in the following item.
\item Bertrand~\cite{Bertrand:77} and Schmidt~\cite{Schmidt:80} studied beta-expansions w.r.t.\ a Salem number $\beta$. They conjectured that
each element of the number field $\mathbb{Q}(\beta)$ admits a periodic beta-expansion. As a Salem number has conjugates on the unit circle it can be shown that this conjecture can be reformulated in terms of the ultimate periodicity of $\tau_{\mathbf{r}}$, where $\mathbf{r}$ is a parameter whose companion matrix $R(\mathbf{r})$ (see \eqref{mata}) has non-real eigenvalues on the unit circle.
\item Shift radix systems admit a geometric theory. In particular, it is possible to define so-called {\it SRS tiles} (see Berth\'e {\it et al.}~\cite{BSSST2011}). In view of the conjugacies mentioned above these tiles contain Rauzy fractals~\cite{Akiyama:02,Rauzy:82} as well as self-affine {\em fundamental domains} of canonical number systems (see K\'atai and  K\H{o}rnyei~\cite{Katai-Koernyei:92}) as special cases. However, also new tiles with different (and seemingly new) geometric properties occur in this class. It is conjectured that SRS tiles always induce tilings of their representation spaces. This contains the {\em Pisot conjecture} (see {\it e.g.}\ Arnoux and Ito~\cite{Arnoux-Ito:01} or Baker~{\it et al.}~\cite{BBK:06}) in the setting of beta-expansions as a special case.
\item Akiyama~{\it et al.}~\cite{Akiyama-Frougny-Sakarovitch:07} study number systems with rational bases and establish relations of these number systems to Mahler's $\frac32$-problem ({\it cf.}~\cite{Mahler:68}). Also these number systems can be regarded as special cases of SRS (see Steiner and Thuswaldner~\cite{ST:11}) and there seem to be relations between the $\frac32$-problem and the length of SRS tiles associated with $\tau_{-2/3}$. These tiles also have relations to the Josephus problem (see again \cite{ST:11}).
\item In recent years variants of shift radix systems have been studied. Although their definition is very close to that of $\tau_\mathbf{r}$, some of them have different properties. For instance, the ``tiling properties'' of SRS tiles are not the same in these modified settings.
\end{itemize}

It is important to recognize that the mapping $\tau_\mathbf{r}$ is ``almost linear'' in the sense that it is the sum of a linear function and a small error term caused by the floor function $\lfloor\cdot\rfloor$ occurring in the definition. To make this more precise define the matrix
\begin{equation}\label{mata}
R({\bf r})= \left(
\begin{array}{ccccc}
  0 & 1 & 0   & \cdots  &0 \\
  \vdots & \ddots &\ddots  & \ddots & \vdots\\
  \vdots  &  & \ddots &\ddots& 0      \\
  0     & \cdots & \cdots & 0 &1 \\
  -r_0     & -r_1   & \cdots & \cdots & -r_{d-1}
\end{array}
\right) \qquad({\bf r}=(r_0, \ldots,r_{d-1}) \in \RR^d)
\end{equation}
and observe that its characteristic polynomial
\begin{equation}\label{chi}
\chi_{\bf r}(X)= X^d +r_{d-1} X^{d-1}+ \cdots + r_1 X + r_0
\end{equation}
 is also the characteristic polynomial of the linear recurrence $z_{n}+r_{d-1}z_{n-1}+ \cdots + r_0 z_{n-d}=0$. Thus \eqref{linearinequalities} implies that
\begin{equation}\label{linear}
\tau_\mathbf{r}(\mathbf{z}) =  R({\bf r})\mathbf{z} +
\mathbf{v}(\mathbf{z}),
\end{equation}
where $\mathbf{v}(\mathbf{z}) = (0,\ldots,0,\{\mathbf{rz}\})^t$ (in particular, $||\mathbf{v}(\mathbf{z})||_\infty<1$). Moreover, again using \eqref{linearinequalities} one easily derives that
\begin{equation}\label{almostlinear}
\hbox{either}\quad\tau_\mathbf{r}(\mathbf{z}_1+\mathbf{z}_2) = \tau_\mathbf{r}(\mathbf{z}_1) + \tau_\mathbf{r}(\mathbf{z}_2) \quad \hbox{or}\quad \tau_\mathbf{r}(\mathbf{z}_1+\mathbf{z}_2) =
\tau_\mathbf{r}(\mathbf{z}_1) - \tau_\mathbf{r}(-\mathbf{z}_2)
\end{equation}
holds for $\mathbf{z}_1,\mathbf{z}_2\in \mathbb{Z}^d$.

The following sets turn out to be of importance in the study of various aspects of shift radix systems.

\begin{definition}[$\D_d$ and $\D_d^{(0)}$]
For $d\in \mathbb{N}$ set
\begin{equation}\label{DdDd0}
\begin{split}
\D_{d}:= & \left\{{\bf r} \in \RR^{d}\; :\; \forall \mathbf{z} \in
\ZZ^{d} \, \exists k,l \in \NN:
     \tau_{\bf r}^{k}({\bf z})=\tau_{\bf r}^{k+l}({\bf z}) \right\}
 \quad\hbox{and}\\
\D_{d}^{(0)}:= & \left\{{\bf r} \in \RR^{d}\; :\; \tau_{\bf r}
\mbox{ is an SRS with finiteness property} \right\} .
\end{split}
\end{equation}
\end{definition}
Observe that $\D_{d}$ consists of all vectors $\mathbf{r}$ such that the iterates of $\tau_{\bf r}$ end up periodically for each starting vector $\mathbf{z}$.

In order to give the reader a first impression of these sets, we present in Figure~\ref{d20} images of (approximations of) the sets $\D_{2}$ and $\D_{2}^{(0)}$. As we will see in Section~\ref{sec:Dd0}, the set $\D_d^{(0)}$ can be constructed starting from $\D_d$ by ``cutting out'' polyhedra. Each of these polyhedra corresponds to a nontrivial periodic orbit. Using this fact, in Section~\ref{sec:algorithms} we shall provide algorithms for the description of $\D_d^{(0)}$. (Compare the more detailed comments on $\D_d$ and $\D_d^{(0)}$ in Sections~\ref{Dd} and~\ref{sec:Dd0}, respectively.)

\begin{figure}
\centering
\includegraphics[width=0.4\textwidth]{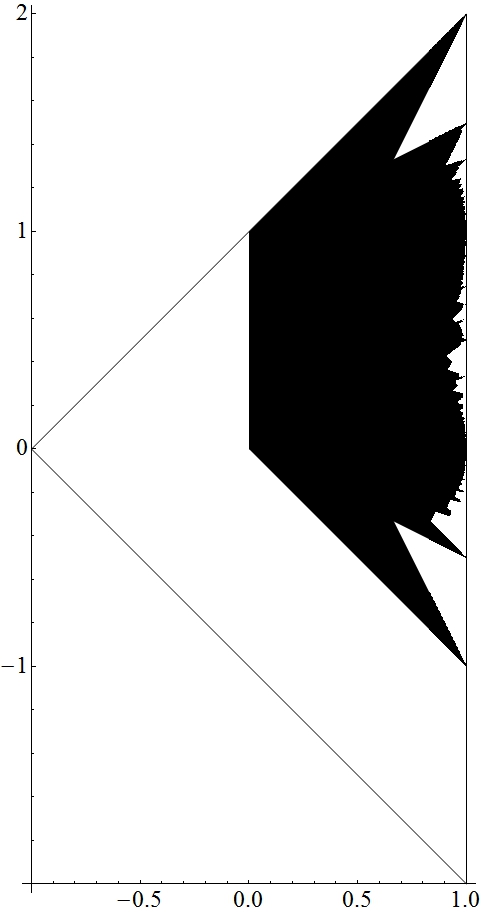}
 \caption{The large triangle is (up to its boundary) the set $\D_2$.  The black set is an approximation of $\D_2^{(0)}$ (see~\cite[Figure~1]{Akiyama-Borbeli-Brunotte-Pethoe-Thuswaldner:05}).}
\label{d20}
\end{figure}

\section{The relation between shift radix systems and numeration systems}

In the present section we discuss relations between SRS and beta-expansions as well as canonical number systems (see {\it e.g.} \cite{Akiyama-Borbeli-Brunotte-Pethoe-Thuswaldner:05,Hollander:96}). Moreover, we provide a notion of radix expansion for integer vectors which is defined in terms of SRS.

\subsection{Shift radix systems and beta-expansions}\label{sec:beta}
The notion of {\it beta-expansion} (introduced by R\'enyi \cite{Renyi:57} and Parry \cite{Parry:60}) is well-known in number theory.

\begin{definition}[Beta-expansion] Let $\beta>1$ be a non-integral real number and define the set
of ``digits" to be
$\mathcal{A}=\{0,1,\ldots,\lfloor\beta\rfloor\}.$
Then
each $\gamma\in[0,\infty)$ can be represented
uniquely by
\begin{equation}\label{greedyexp}
\gamma= \sum_{i=m}^\infty  \frac{a_i}{\beta ^i}
\end{equation}
with $m\in\mathbb{Z}$ and $a_i\in \mathcal{A}$ chosen in a way that
\begin{equation}\label{greedycondition}
0\leq \gamma-\sum_{i=m}^n \frac{a_{i}}{\beta^{i}} < \frac{1}{\beta^{n}}
\end{equation}
for all $n\ge m$. Observe that this means that the representation in \eqref{greedyexp} is the {\it greedy expansion} of $\gamma$ with respect to $\beta$.
\end{definition}

For $\gamma\in [0,1)$ we can use the {\it beta-transform}
\begin{equation}\label{betatransform}
T_{\beta} (\gamma) = \beta \gamma -\lfloor \beta \gamma \rfloor,
\end{equation}
to establish this greedy expansion, namely we have
$$ 
a_i = \lfloor \beta T^{i-1}_\beta (\gamma) \rfloor
$$ 
({\it cf.} Renyi \cite{Renyi:57}). This holds no longer for $\gamma = 1$, where the beta-transform yields a  representation (whose digit string is often denoted by $d(1,\beta)$) different from the greedy algorithm (see \cite{Parry:60}).

In the investigation of beta-expansions two classes of algebraic numbers, {\em Pisot} and {\em Salem numbers} play an important role. For convenience, we recall their definitions.
\begin{itemize}
\item An algebraic integer $\alpha > 1$ is called a {\em Pisot number} if all its algebraic conjugates have absolute value less than 1.
\item An algebraic integer $\alpha > 1$ is called a {\em Salem number} if all its algebraic conjugates have absolute value less than or equal to 1 with at least one of them lying on the unit circle.
\end{itemize}

Bertrand \cite{Bertrand:77} and Schmidt \cite{Schmidt:80} provided relations between periodic beta-expansions and  Pisot as well as Salem numbers (see Section~\ref{sec:Salem} for details). Frougny and Solomyak \cite{Frougny-Solomyak:92} investigated the problem to characterize base numbers $\beta$ which give rise to {\it finite beta-expansions} for large classes of numbers. Denoting the set of positive reals having finite greedy expansion with respect to $\beta$  by ${\rm Fin}(\beta)$, we say that $\beta>1$ has property (F) if
\begin{equation}\label{PropertyF}
{\rm Fin}(\beta) = \mathbb{Z}[1/\beta]\cap [0,\infty).
\end{equation}
 As it is shown in \cite{Frougny-Solomyak:92} property (F) can hold only for Pisot numbers $\beta$, on the other hand, not all Pisot numbers have property (F).

Characterizing all Pisot numbers with property (F) has turned out to be a very difficult problem: many partial results have been established, {\em e.g.} by Frougny and Solomyak~\cite{Frougny-Solomyak:92} (quadratic Pisot numbers), or Akiyama~\cite{Akiyama:00} (cubic Pisot units). The following theorem is basically due to Hollander~\cite{Hollander:96} (except for the notion of SRS) and establishes the immediate relation of the problem in consideration with shift radix systems (recall that $\D_{d}^{(0)}$ is defined in \eqref{DdDd0}). 

\begin{theorem}\label{srsbeta}
Let $\beta>1$ be an algebraic integer with minimal polynomial
$X^{d+1}+b_dX^{d}+\cdots + b_1X+b_0$. Set
\begin{equation}\label{ers}
r_j := -\left( b_j \beta^{-1} + b_{j-1}\beta^{-2} + \cdots + b_0
\beta^{-j-1} \right), \quad 0\le j\le d-1.
\end{equation}
Then $\beta$ has property (F) if and only if $(r_0,\ldots, r_{d-1}) \in \D_{d}^{(0)}$.
\end{theorem}

\begin{remark}
Observe that $r_0,\ldots,r_{d-1}$ in \eqref{ers} can also be defined in terms of the identity
$$
X^{d+1}+b_d X^{d}  + \dots +b_{1} X+ b_0 =(X-\beta)(X^{d}+r_{d-1}X^{d-1} +\dots + r_0).
$$
\end{remark}
It turns out that Theorem \ref{srsbeta} is an immediate consequence of the following more general observation ({\it cf.}\ Berth\'e {\it et al.}~\cite{BSSST2011}).

\begin{proposition}\label{prop:betanumformula}
Under the assumptions of Theorem \ref{srsbeta} and denoting $\mathbf{r}=(r_0,\ldots,r_{d-1})$ we have
\begin{equation} \label{eq:conj}
\{\mathbf{r}\tau_\mathbf{r}(\mathbf{z})\}=T_\beta(\{\mathbf{r}\mathbf{z}\}) \quad \mbox{for all } \mathbf{z}\in\mathbb{Z}^d.
\end{equation}
In particular, the restriction of $T_\beta$ to $\mathbb{Z}[\beta]\cap[0,1)$ is conjugate to $\tau_\mathbf{r},$ {\em i.e.}, denoting $\Phi_{\bf r}:\mathbf{z}\mapsto\{\mathbf{r}\mathbf{z}\}$ we have the following commutative diagram.
$$
\CD
\ZZ^d       @> \tau_{{\bf r}} >>\ZZ^d\\
@V\Phi_{\bf r} VV                               @VV\Phi_{\bf r} V\\
\mathbb{Z}[\beta]\cap [0,1) @> T_{\beta} >>   \mathbb{Z}[\beta]\cap [0,1)
\endCD
$$
\end{proposition}

\begin{proof}
Let the notations be as in Theorem~\ref{srsbeta}. Let $\mathbf{z}=(z_0,\ldots,z_{d-1})^t\in\mathbb{Z}^d,$  $z_d=-\lfloor\mathbf{r}\mathbf{z}\rfloor$, and $\mathbf{b}=(b_0,\ldots,b_d).$
Then we have, with the $(d+1)\times(d+1)$ companion matrix $R(\mathbf{b})$ 
defined analogously as $R({\bf  r})$ in~\eqref{mata} (note that the vector $\mathbf{b}$ has $d+1$ entries),
\begin{align}
\{(r_0,\ldots,r_{d-1},1)R(\mathbf{b})(z_0,\ldots,z_d)^t\}
&=  \{(-b_0, r_0-b_1,\ldots,r_{d-1}-b_d)(z_0,\ldots,z_d)^t\} \nonumber\\
&= \{  -b_0 z_0 + (r_0-b_1)z_1 +\cdots + (r_{d-1} - b_d) z_d \} \nonumber\\
&= \{  r_0 z_1 + \cdots +r_{d-1} z_d \} \label{RBfirst} \\
&= \{ (r_0,\ldots, r_{d-1})(z_1,\ldots, z_d)^t \} \nonumber\\
&= \{\mathbf{r}\tau_\mathbf{r}(\mathbf{z})\}.\nonumber
\end{align}
In the third identity we used that $b_0,\ldots,b_d,z_0,\ldots,z_d$ are integers.
Observing that $(r_0,\ldots,r_{d-1},1)$ is a left eigenvector of the
matrix $R(\mathbf{b})$ with eigenvalue $\beta$ we conclude that
\begin{align}
\{(r_0,\ldots,r_{d-1},1)R(\mathbf{b})(z_0,\ldots,z_d)^t\} 
&= \{\beta (r_0,\ldots,r_{d-1},1) (z_0,\ldots,z_d)^t\} \nonumber\\ 
&= \{\beta ({\bf rz} +z_d)\} \nonumber\\ 
&= \{\beta ({\bf rz} - \lfloor {\mathbf{rz}} \rfloor)\} \label{RBsecond}\\
&= \{\beta\{\mathbf{r}\mathbf{z}\}\} \nonumber\\
&= T_\beta(\{\mathbf{r}\mathbf{z}\}).\nonumber
\end{align}
Combining \eqref{RBfirst} and \eqref{RBsecond} yields \eqref{eq:conj}.

Since the minimal polynomial of $\beta$ is irreducible,  $\{r_0,\ldots,r_{d-1},1\}$ is a basis of $\mathbb{Z}[\beta]$.
Therefore  the map $$\Phi_{\bf r}:\mathbf{z}\mapsto\{\mathbf{r}\mathbf{z}\}$$ is a bijective map from
$\mathbb{Z}^d$ to $\mathbb{Z}[\beta]\cap[0,1)$.
This proves the conjugacy between $T_\beta$ on $\mathbb{Z}[\beta]\cap[0,1)$ and $\tau_\mathbf{r}.$
\end{proof}

Theorem~\ref{srsbeta} is now an easy consequence of this conjugacy:

\begin{proof}[Proof of Theorem \ref{srsbeta}]
Let $\gamma\in\mathbb{Z}[1/\beta]\cap[0,\infty)$. Then, obviously,  $\gamma\beta^k\in \mathbb{Z}[\beta]\cap[0,\infty)$ for a suitable integer exponent $k$, and the beta-expansions of $\gamma$ and $\gamma\beta^k$ have the same digit string. Thus $\beta$ admits property (F) if and only if every element of $\mathbb{Z}[\beta]\cap[0,\infty)$ has finite beta-expansion. The greedy condition \eqref{greedycondition}
now shows that it even suffices to guarantee finite beta-expansions for every element of $\mathbb{Z}[\beta]\cap[0,1)$. Thus the conjugacy in Proposition~\ref{prop:betanumformula} implies the result.
\end{proof}

\begin{example}[Golden mean and Tribonacci]\label{ex:fibo}
First we illustrate Proposition~\ref{prop:betanumformula} for $\beta$ equal to the golden mean $\varphi=\frac{1+\sqrt{5}}{2}$ which is a root of the polynomial $X^2 - X - 1 = (X-\varphi)(X+r_0)$ with $r_0=\frac1\varphi=\frac{-1+\sqrt{5}}{2}$. By Proposition~\ref{prop:betanumformula} we therefore get that $T_\varphi$ is conjugate to $\tau_{1/\varphi}$. As $\varphi$ has property (F) (see \cite{Frougny-Solomyak:92}), we conclude that $\frac1\varphi \in \D_1^{(0)}$. Let us confirm the conjugacy for a concrete example. Indeed, starting with $3$ we get $\tau_{1/\varphi}(3)=-\lfloor\frac{3}{\varphi}\rfloor=-1$. The mapping $\Phi_{1/\varphi}$ for these values is easily calculated by $\Phi_{1/\varphi}(3)=\{\frac3\varphi\}=\{3\varphi-3\}=3\varphi-4$ and $\Phi_{1/\varphi}(-1)=-\varphi+2$. As $T_\varphi(3\varphi-4)=\{3\varphi^2-4\varphi\}=\{-\varphi+3\}=-\varphi+2$ the conjugacy is checked for this instance.

The root $\beta > 1$ of the polynomial $X^3-X^2-X-1$ is often called {\em Tribonacci number}. In this case Proposition~\ref{prop:betanumformula} yields that  $r_0=\frac1\beta$ and  $r_1=\frac1\beta + \frac1{\beta^2}$. Thus $T_\beta$ is conjugate to $\tau_{(1/\beta, 1/\beta+1/\beta^2)}$. Property (F) holds also in this case.
\end{example}

\subsection{Shift radix systems and canonical number systems}
It was already observed  in 1960 by Knuth~\cite{Knuth:60} and  in 1965 Penney~\cite{Penney:65} that $\alpha=-1+\sqrt{-1}$ can be used as a base for a number system in the Gaussian integers. Indeed, each non-zero $\gamma\in\mathbb{Z}[\sqrt{-1}]$ has a unique representation of the shape
$\gamma=c_0+c_1\alpha+\cdots+c_h\alpha^h$ with $c_i\in\{0,1\}$ $(0\le i< h)$, $c_h=1$ and $h\in\mathbb{N}$. This simple observation has been the starting point for several generalizations of the classical $q$-ary number systems to algebraic number fields, see for instance \cite{Gilbert:81,Katai-Kovacs:80,Katai-Kovacs:81,Katai-Szabo:76,Kovacs-Pethoe:91}.

The following more general notion has proved to be useful in this context.

\begin{definition}[Canonical number system, see {Peth\H{o}~\cite{Pethoe:91}}]\label{def:CNS}
Let
\[
P(X) = p_dX^d + p_{d-1}X^{d-1}+\cdots+p_1X+ p_0 \in \mathbb{Z}[X], \quad
p_0\ge 2, \quad p_d\neq 0; \quad \mathcal{N}=\{0,1,\ldots,p_0-1\}
\]
and $\mathcal{R}:=\mathbb{Z}[X]/P(X)\mathbb{Z}[X]$ and let $x$ be the image of $X$ under the canonical epimorphism from ${\ZZ}[X]$ to $\mathcal{R}$. If every $B\in \mathcal{R}, B \neq 0$, can be represented uniquely as
$$
B=b_0 + b_1 x + \cdots + b_{\ell} x^{\ell}
$$
with $b_0,\ldots,b_{\ell} \in \mathcal{N}, b_{\ell} \neq 0$, the system $(P,\mathcal{N})$ is called a {\it canonical number system} ({\it CNS} for short); $P$ is called its {\it base} or {\it CNS polynomial}, $\mathcal{N}$ is called the set of {\it digits}.
\end{definition}

Using these notions the problem arises, whether it is possible to characterize CNS polynomials by algebraic conditions on their coefficients and roots. First of all, it is easy to see that a CNS polynomial has to be expanding (see~\cite{Pethoe:91}). Further characterization results could be gained {\it e.g.} by Brunotte~\cite{Brunotte:01}, who gave a characterization of all quadratic monic CNS polynomials. For irreducible CNS polynomials of general degree, Kov\'acs~\cite{Kovacs:81a} proved that a polynomial $P$ given as in Definition~\ref{def:CNS}
is CNS if $p_0\ge 2$ and $p_0\ge p_1\ge \cdots \ge p_{d-1}>0$. In \cite{Akiyama-Pethoe:02,Scheicher-Thuswaldner:03} characterization results under the condition $p_0 > |p_1|+ \cdots + |p_{d-1}|$ were shown, \cite{Burcsi-Kovacs:08} treats polynomials with small $p_0$. A general characterization of CNS polynomials is not known and seems to be hard to obtain.

It has turned out that in fact there is again a close connection to the problem of determining shift radix systems with finiteness property. The corresponding result due to Akiyama {\em et al.}~\cite{Akiyama-Borbeli-Brunotte-Pethoe-Thuswaldner:05} and Berth\'e {\em et al.}~\cite{BSSST2011} is given in the following theorem  (recall the definition of $\D_{d}^{(0)}$ in \eqref{DdDd0}).

\begin{theorem}\label{srscns}
Let $P(X) := p_dX^d+p_{d-1}X^{d-1}+\cdots+p_1X+p_0 \in \ZZ[X]$. Then $P$ is a CNS-polynomial if and only if $\mathbf{r}:=\left(\frac{p_d}{p_0},\frac{p_{d-1}}{p_0},\ldots, \frac{p_1}{p_0}\right) \in \D_{d}^{(0)}$.
\end{theorem}

By a similar reasoning as in the proof of Theorem \ref{srsbeta} we will derive the result from a more general one, this time establishing a conjugacy between $\tau_\mathbf{r}$ and the restriction of the following {\it backward division mapping} $D_P:\mathcal{R} \rightarrow \mathcal{R}$ (with $\mathcal{R}:=\mathbb{Z}[X]/P(X)\mathbb{Z}[X]$ as above) to a well-suited finitely generated $\ZZ$-submodule of $\mathcal{R}$ (compare \cite{BSSST2011}).

\begin{definition}[Backward division mapping] \label{lem:bdm}
The {\em backward division mapping} $D_P:\mathcal{R} \rightarrow \mathcal{R}$  for $B=\sum_{i=0}^\ell b_i x^i$, $b_i\in\mathbb{Z},$ is defined by
\[
D_P(B) = \sum_{i=0}^{\ell-1} b_{i+1}x^i - \sum_{i=0}^{d-1} q p_{i+1}x^i,\quad
q=\left\lfloor{\frac{b_0}{p_0}}\right\rfloor.
\]
\end{definition}

Observe that $D_P(B)$ does not depend on the representative of the equivalence class of~$B$, that
\begin{equation} \label{DP}
B=(b_0-q p_0)+x D_P(B),
\end{equation}
and that $c_0=b_0-q p_0$ is the unique element in $\mathcal{N}$ with $B-c_0 \in x\mathcal{R}$ (compare \cite{Akiyama-Borbeli-Brunotte-Pethoe-Thuswaldner:05} for the case of monic $P$ and \cite{Scheicher-Surer-Thuswaldner-vdWoestijne:08} for the general case). Iterating the application of $D_P$ yields
\begin{equation}\label{DPn}
B=\sum_{n=0}^{m-1} c_n x^n + x^m D_P^m(B)
\end{equation}
with $c_n=D_P^n(B)-x D_P^{n+1}(B)\in\mathcal{N}.$
If we write formally
\begin{equation}\label{cnsrepinfinity}
B = \sum_{n=0}^\infty c_n x^n
\end{equation}
then it follows from the reasoning above that this representation is unique in having the property that
\begin{equation}\label{PNrep}
B -\sum_{i=0}^{m-1} c_n x^n \in x^m\mathcal{R},\; c_n\in\mathcal{N},\,
\text { for all } m\in\mathbb{N}.
\end{equation}
Expansion \eqref{cnsrepinfinity} is called the $(P,\mathcal{N})${\it -representation} of $B\in \mathcal{R}$.

In order to relate $D_P$ to an SRS, it is appropriate to use the so-called {\it Brunotte module} \cite{Scheicher-Surer-Thuswaldner-vdWoestijne:08}.

\begin{definition}[Brunotte module]
The \emph{Brunotte basis modulo $P = p_dX^d+p_{d-1}X^{d-1}+\cdots+p_1X+p_0$} is the set  $\{w_0,\ldots,w_{d-1}\}$ with
\begin{equation} \label{eq:Brunottebasis}
w_0=p_d, \quad w_1=p_dx+p_{d-1}, \quad w_2=p_dx^2+p_{d-1}x+p_{d-2},\dots,\quad w_{d-1}=p_dx^{d-1}+\dots+p_1.
\end{equation}
The \emph{Brunotte module} $\Lambda_P$ is the $\mathbb{Z}$-submodule of $\mathcal{R}$ generated by the Brunotte basis.
We furthermore denote the representation mapping with respect to the Brunotte basis by
\[
\Psi_P: \,\Lambda_P \to \mathbb{Z}^d, \quad  B = \sum_{k=0}^{d-1} z_kw_k
\mapsto (z_0,\ldots,z_{d-1})^t.
\]
\end{definition}

We call P {\it monic}, if $p_d = \pm 1.$ Note that in this instance $\Lambda_P$ is isomorphic to $\mathcal{R}$, otherwise $\mathcal{R}$ is not finitely generated.

Now we can formulate the announced conjugancy between backward division and the SRS transform (compare \cite{BSSST2011}).
\begin{proposition}\label{p:CNSconjugacy}
Let $P(X)=p_d X^d+p_{d-1}X^{d-1}+\cdots+p_1X+p_0\in\mathbb{Z}[X]$, $p_0 \geq 2$, $p_d \neq 0$, $\mathbf{r}=\big(\frac{p_d}{p_0},\ldots,\frac{p_1}{p_0}\big)$. Then we have
\begin{equation} \label{eq:CNSconj}
D_P\Psi_P^{-1}(\mathbf{z}) = \Psi_P^{-1}\tau_\mathbf{r}(\mathbf{z}) \quad\mbox{for all }\mathbf{z}\in\mathbb{Z}^d.
\end{equation}
In particular, the restriction of $D_P$ to $\Lambda_P$ is conjugate to $\tau_\mathbf{r}$ according to the diagram
$$
\CD
\ZZ^d       @> \tau_{{\bf r}} >>\ZZ^d\\
@V\Psi_{\bf P}^{-1} VV        @VV\Psi_{\bf P}^{-1} V\\
\Lambda_P @> D_P >>   \Lambda_P
\endCD
.$$
\end{proposition}

\begin{proof}
It follows immediately from the definitions that on $\Lambda_P$ we have
\begin{equation} \label{eq:TA}
D_P\bigg(\sum_{k=0}^{d-1} z_kw_k\bigg) = \sum_{k=0}^{d-2} z_{k+1}w_k - \left\lfloor \frac{z_0p_d+\cdots+z_{d-1}p_1}{p_0} \right\rfloor w_{d-1},
\end{equation}
 which implies (\ref{eq:CNSconj}). Since $\Psi_P:\ \Lambda_P \to \mathbb{Z}^d$ is bijective the proof is complete.
\end{proof}

For monic $P$ Proposition \ref{p:CNSconjugacy} establishes a conjugacy between $D_P$ on the full set $\mathcal{R}$ and $\tau_{{\bf r}}$.

\begin{proof}[Proof of Theorem~\ref{srscns}.]

Observing Proposition \ref{p:CNSconjugacy} the  theorem follows from the fact,  that it is sufficient to establish the finiteness of the $(P,\mathcal{N})$-representations of all $B\in\Lambda_P$ in order to check whether $(P,\mathcal{N})$ is a CNS (compare \cite{Scheicher-Surer-Thuswaldner-vdWoestijne:08}).
\end{proof}

\begin{example}[The $\frac32$-number system]\label{ex32}
Considering $P(X)=-2X+3$ and $\mathcal{R}=\mathbb{Z}[X]/P(X)\mathbb{Z}[X]$ we get $\mathcal{R}\cong\mathbb{Z}[\frac32]=\mathbb{Z}[\frac12]$. Thus in this case we can identify the image of $X$ under the natural epimorphism $\mathbb{Z}[X]\to \mathcal{R}$ with $x=\frac32$ and the backward division mapping yields representations of the form $B=b_0+b_1\frac32 + b_2 (\frac32)^2+ \cdots$ for $B\in\mathbb{Z}[\frac12]$.  For $B\in \mathbb{Z}$ this was discussed (apart from a leading factor $\frac12$) under the notation {\em $\frac32$-number system} in~\cite{Akiyama-Frougny-Sakarovitch:07}. Namely, each $B\in\mathbb{N}$ can be represented as $B=\frac12 \sum_{n=0}^{\ell(n)} b_n(\frac32)^n$ with ``digits'' $b_n\in\{0,1,2\}$.  The language of possible digit strings turns out to be very complicated.
If we restrict ourselves to the Brunotte module $\Lambda_{P}$ (which is equal to $2\mathbf{Z}$ in this case) Proposition~\ref{p:CNSconjugacy} implies that the backward division mapping $D_{-2X+3}$ is conjugate to the SRS $\tau_{-2/3}$. As $-1$ doesn't have a finite $(-2X+3,\{0,1,2\})$-representation, we conclude that $-\frac23 \not \in \D_1^{(0)}$.

We mention that the $\frac32$-number system was used in~\cite{Akiyama-Frougny-Sakarovitch:07} to established irregularities in the distribution of certain generalizations of Mahler's $\frac32$-problem ({\it cf.}~\cite{Mahler:68}).
\end{example}

\begin{example}[Knuth's Example]\label{exKnuth}
Consider $P(X)=X^2 + 2X + 2$. As this polynomial is monic with root $\alpha=-1+\sqrt{-1}$, we obtain $\mathcal{R}\cong \mathbb{Z}[\sqrt{-1}]\cong\Lambda_P$. In this case $(X^2+2X+2,\{0,1\})$ is a CNS (see Knuth~\cite{Knuth:60}) that allows to represent each $\gamma\in\mathbb{Z}[\sqrt{-1}]$ in the form $\gamma=b_0+b_1\alpha + \cdots + b_\ell \alpha^\ell$ with digits $b_j\in \{0,1\}$. According to Proposition~\ref{p:CNSconjugacy} the backward division mapping $D_P$ is conjugate to the SRS mapping $\tau_{(\frac12,1)}$. Therefore, $(\frac12,1)\in\D_2^{(0)}$.
\end{example}

\subsection{Digit expansions based on shift radix systems}
In the final part of this section we consider a notion of representation for vectors $\mathbf{z} \in \mathbb{Z}^d$ based on the SRS-transformation $\tau_\mathbf{r}$ (compare \cite{BSSST2011}).

\begin{definition}[SRS-representation]\label{def:SRSexp}
Let $\mathbf{r} \in \mathbb{R}^d$.
The \emph{SRS-representation} of $\mathbf{z} \in \mathbb{Z}^d$ with respect to $\mathbf{r}$ is the sequence $(v_1,v_2,v_3,\ldots)$, where $v_k=\big\{\mathbf{r}\tau_\mathbf{r}^{k-1}(\mathbf{z})\big\}$ for all $k\ge1$.
\end{definition}

Observe that vectors $\mathbf{r}\in\mathcal{D}_d^{(0)}$ give rise to finite SRS-representations, and vectors $\mathbf{r}\in\mathcal{D}_d$ to ultimately periodic SRS-representations.

Let $(v_1,v_2,v_3,\ldots)$ denote the SRS-representation of $\mathbf{z} \in \mathbb{Z}^d$ with respect to $\mathbf{r}$. The following lemma shows that there is a radix expansion of integer vectors where the companion matrix $R(\mathbf{r})$ acts like a base and the vectors $(0,\ldots,0,v_j)^t$ act like the digits (see \cite[Equation~(4.2)]{Akiyama-Borbeli-Brunotte-Pethoe-Thuswaldner:05}). This justifies the name {\em shift radix system}.

\begin{lemma}\label{lum}
Let $\mathbf{r} \in \mathbb{R}^d$ and $(v_1,v_2,\ldots)$ be the SRS-representation of $\mathbf{z} \in \mathbb{Z}^d$ with respect to~$\mathbf{r}$.
Then we have
\begin{equation} \label{eq:tauk}
R(\mathbf{r})^n \mathbf{z} = \tau_\mathbf{r}^n(\mathbf{z}) - \sum_{k=1}^n R(\mathbf{r})^{n-k} (0,\ldots,0,v_k)^t.
\end{equation}
\end{lemma}

\begin{proof}
Starting from \eqref{linear} and using induction we immediately get for
the $n$-th iterate of $\tau_{\mathbf r}$ that
\begin{equation}\label{tauiterate}
\tau^n_{\bf r}({\bf z})=R({\bf r})^n {\bf {z}} + \sum_{k=1}^n R({\bf r})^{n-k} {\bf v}_k
\end{equation}
with vectors ${\bf v}_k=(0,\ldots,0, \{\mathbf{r} \tau_\mathbf{r}^{k-1}(\mathbf{z}) \})^t$.
\end{proof}

There is a direct relation between the digits of a given beta-expansion and the digits of the associated SRS-representation of $\mathbf{z} \in \mathbb{Z}^d$ ({\it cf.}~\cite{BSSST2011}; a related result for CNS is contained in \cite[Lemma~5.5]{BSSST2011}).

\begin{proposition} \label{cor:betanumformula}
Let $\beta$ and $\mathbf{r}$ be defined as in Theorem~\ref{srsbeta}. Let  $(v_1,v_2,v_3,\ldots)$ be the SRS-representation of $\mathbf{z} \in \mathbb{Z}^d$ and $\{\mathbf{r}\mathbf{z}\}=\sum_{n=1}^\infty a_n \beta^{-n}$  be the beta-expansion of $v_1=\{\mathbf{r}\mathbf{z}\}$. Then we have
\[
v_n=T_\beta^{n-1}(\{\mathbf{r}\mathbf{z}\}) \quad \mbox{and} \quad a_n=\beta v_n-v_{n+1} \quad \mbox{for all }n\ge1.
\]
\end{proposition}

\begin{proof}
By Definition~\ref{def:SRSexp} and (\ref{eq:conj}), we obtain that $v_n=\{\mathbf{r}\tau_\mathbf{r}^{n-1}(\mathbf{z})\}=T_\beta^{n-1}(\{\mathbf{r}\mathbf{z}\})$, which yields the first assertion. Using this equation and the definition of the beta-expansion, we obtain
\[
a_n = \big\lfloor \beta T_\beta^{n-1}(\{\mathbf{r}\mathbf{z}\}) \big\rfloor = \beta T_\beta^{n-1}(\{\mathbf{r}\mathbf{z}\}) - \big\{\beta T_\beta^{n-1}(\{\mathbf{r}\mathbf{z}\})\big\} = \beta v_n - T_\beta^n(\{\mathbf{r}\mathbf{z}\}) = \beta v_n-v_{n+1}. \hfill\qedhere
\]
\end{proof}

\begin{example}[Golden mean, continued]
Again we deal with $\beta=\varphi$, the golden mean, and consider the digits of $3\varphi-4 = \Phi_{1/\varphi}(3)$. Using \eqref{greedycondition} one easily computes the beta expansion $3\varphi-4 = \frac1\beta + \frac1{\beta^3}$.  Using the notation of Proposition~\ref{cor:betanumformula} we have that $a_1=a_3=1$ and $a_i=0$ otherwise. On the other hand we have $\tau_{1/\varphi}(3)= -\lfloor \frac3\varphi\rfloor = -1$, $\tau_{1/\varphi}(-1)= -\lfloor- \frac{1}\varphi\rfloor = 1$, and $\tau_{1/\varphi}(1)= -\lfloor \frac1\varphi\rfloor = 0$. Thus, for the SRS-representation $3=(v_1,v_2,\ldots)$ we get $v_1=\{\frac3\varphi\}=\{3\varphi-3\}=3\varphi-4$, $v_2=\{-\frac{1}\varphi\}=\{-\varphi+1\}=-\varphi+2$,  $v_3=\{\frac1\varphi\}=\{\varphi-1\}=\varphi-1$, and $v_i=0$ for $i \ge 4$. It is now easy to verify the formulas in Proposition~\ref{cor:betanumformula}. For instance, $\varphi v_1 - v_2 = \varphi(3\varphi-4)-(-\varphi+2)=3\varphi^2 - 3\varphi - 2 = 1 = a_1$.
\end{example}

\section{Shift radix systems with periodic orbits: the sets $\D_{d}$ and the Schur-Cohn region} \label{Dd}

\subsection{The sets $\D_d$ and their relations to the Schur-Cohn region}
In this section we focus on results on the sets $\D_{d}$  defined in \eqref{DdDd0}. To this aim it is helpful to consider the {\em Schur-Cohn} region (compare~\cite{Schur:18})
\[
\E_d:=\{\mathbf{r} \in \RR^d\; : \; \varrho(R(\mathbf{r})) < 1\}.
\]
Here $\varrho(A)$ denotes the spectral radius of the matrix $A$. The following important relation holds between the sets  $\E_d$ and $\D_d$ ({\it cf.}~\cite{Akiyama-Borbeli-Brunotte-Pethoe-Thuswaldner:05}) .

\begin{proposition}\label{EdDdEd}
For $d\in\mathbb{N}$ we have
$$
\E_d \subset \D_d \subset \overline{\E_{d}}.
$$
\end{proposition}

\begin{proof}
We first deal with the proof of the assertion
$\E_d \subset \D_d$.
Let us assume now $0 <\varrho(R({\mathbf r}))<1$ (the instance $\mathbf{r}=\mathbf{0}$ is trivial) and choose a number ${\tilde\varrho}\in(\varrho(R({\bf r})),1)$. According {\it e.g.} to \cite[formula (3.2)]{Lagarias-Wang:96a} it is possible to construct a norm $|| \cdot ||_{\tilde\varrho}$ on $\ZZ^d$ with the property that
\begin{equation}\label{norminequ}
||R({\mathbf r}){\bf x}||_{\tilde\varrho}\leq
{\tilde\varrho}
||{\bf x}||_{\tilde\varrho}.
\end{equation}
Using \eqref{tauiterate} it follows that
\begin{equation*}
\Vert \tau_{{\bf r}}^n({\bf z})\Vert _{\tilde\varrho} \leq {\tilde\varrho} ^n \Vert {\bf z} \Vert _{\tilde\varrho}+ c\sum_{k=1}^n {\tilde\varrho}^{n-k}\leq {\tilde\varrho} ^n \Vert {\bf z} \Vert _{\tilde\varrho}+ \frac{c}{1
-{\tilde\varrho}},
\end{equation*}
where $c=\sup\{||(0,\ldots,0,\varepsilon)^t||_{\tilde\varrho} \,:\, \varepsilon\in[0,1)\}$ is a finite positive constant. Hence, for $n$ large enough,
\begin{equation}\label{pluseins}
\Vert \tau_\mathbf{r}^n({\bf z})\Vert_{\tilde\varrho} \leq  \frac{c}{1-{\tilde\varrho}}+ 1.
\end{equation}
Since the set of all iterates $\tau_\mathbf{r}^n({\bf z}), n\in \NN$, is bounded in $\ZZ^d$ it has to be finite and, hence, the sequence $(\tau_\mathbf{r}^n({\bf z}))_ {n\in \NN}$ has to be ultimately periodic. Therefore we have $\mathbf{r}\in  \D_d$.

We now switch to the assertion $\D_d \subset \overline{\E_{d}}$.
Let us assume $\tau_\mathbf{r}\in \D_d$ and, by contrary, that there exists an eigenvalue $\lambda$ of $R(\bf{r})$ with $|\lambda| > 1.$ Let $\mathbf{u}^t$ be a left eigenvector of $R(\bf{r})$ belonging to $\lambda$. Multiplying \eqref{tauiterate} by $\mathbf{u}^t$ from the left we find that
\begin{equation}\label{utauiterate}
|\mathbf{u}^t\tau^n_{\bf r}({\bf z})|=|\lambda^n \mathbf{u}^t{\bf z} + \sum_{k=1}^n \lambda^{n-k} \mathbf{u}^t{\bf v}_k|
\end{equation}
for any ${\bf z}\in \ZZ^d.$ Since $||{\bf v}_k||_\infty<1$ there is an absolute constant, say $c_1$, such that $|\mathbf{u}^t{\bf v}_k|\leq c_1$ for any $k$. Choosing $\mathbf{z}\in \ZZ^d$ such that $|\mathbf{u}^t{\bf z}|> \frac {c_1+1}{|\lambda|-1}$ it follows from \eqref{utauiterate} that
$|\mathbf{u}^t\tau^n_{\bf r}({\bf z})| \geq c_2|\lambda|^n$
with some positive constant $c_2$. Therefore the sequence $(\tau_\mathbf{r}^n({\bf z}))_ {n\in \NN}$ cannot be bounded, which contradicts the assumption that $\mathbf{r}\in \D_d$.
\end{proof}

Using the last proposition and the fact that the spectral radius of a real monic polynomial is a continuous function in the coefficients of the polynomial it can easily be shown that the boundary of $\D_d$ can be characterized as follows (compare \cite{Akiyama-Borbeli-Brunotte-Pethoe-Thuswaldner:05} for details).
\begin{corollary}\label{Ddboundary}
For $d\in\mathbb{N}$ we have that
\begin{equation*}
\partial \D_{d}:=  \left\{{\bf r} \in \RR^{d}\; :\; \varrho(R({\mathbf r}))=1 \right\}.
\end{equation*}
\end{corollary}

\subsection{The Schur-Cohn region and its boundary} \label{sec:Ed}
We want to give some further properties of $\E_d$. Since the coefficients of a polynomial depend continuously on its roots it follows that  $\E_d = \interior(\overline{\E_d})$. Moreover, one can prove that $\E_d$ is simply connected  ({\it cf.} \cite{Fam-Meditch:78}). By a result of Schur~\cite{Schur:18} the sets $\E_d$ can be described by  determinants.

\begin{proposition}[{{\it cf}.~Schur~\cite{Schur:18}}]\label{stregion}
For $0 \leq k < d$ let
\[\delta_k(r_0,\ldots,r_{d-1})=
\left(\begin{array}{cccccccc}
1       & 0      & \cdots & 0      & r_0    & \cdots & \cdots & r_k    \\
r_{d-1} & \ddots & \ddots & \vdots & 0      & \ddots &        & \vdots \\
\vdots  & \ddots & \ddots & 0      & \vdots & \ddots & \ddots & \vdots \\
r_{d-k-1} & \cdots & r_{d-1}& 1      & 0      & \cdots & 0      & r_0    \\
r_0     & 0      & \cdots & 0      & 1      & r_{d-1}& \cdots & r_{d-k-1}\\
\vdots  & \ddots & \ddots & \vdots & 0      & \ddots & \ddots & \vdots \\
\vdots  &        & \ddots & 0      & \vdots & \ddots & \ddots & r_{d-1}\\
r_k     & \cdots & \cdots & r_0    & 0      & \cdots & 0      & 1
\end{array}\right)\in \RR^{2(k+1) \times 2(k+1)}.\]
Then the sets $\E_d$ are given by
\begin{equation}\label{charaEd}
\E_d=\left\{(r_0,\ldots,r_{d-1})\in \RR^d \; :\; \forall k \in
\{0,\ldots,d-1\}: \;
\det\left(\delta_k(r_0,\ldots,r_{d-1})\right)>0\right\}.
\end{equation}
\end{proposition}

\begin{example}
Evaluating the determinants for $d\in\{1,2,3\}$ yields (see also \cite{Fam-Meditch:78})
\begin{equation}\label{1902082}\begin{split}
\E_1= & \{x \in \RR \; : \; \abs{x}<1\}, \\
\E_2= & \{(x,y) \in \RR^2\; :\; \abs{x}<1, \abs{y}<x+1\}, \\
\E_3= & \{(x,y,z) \in \RR^3\; :\; \abs{x}<1, \abs{y-xz}<1-x^2,
\abs{x+z}<y+1\}.
\end{split}\end{equation}
Thus $\E_2$ is the (open) triangle in Figure 1. $\E_3$ is the solid depicted in Figure \ref{E3}.
\begin{figure}
\centering
\includegraphics[width=0.6\textwidth]{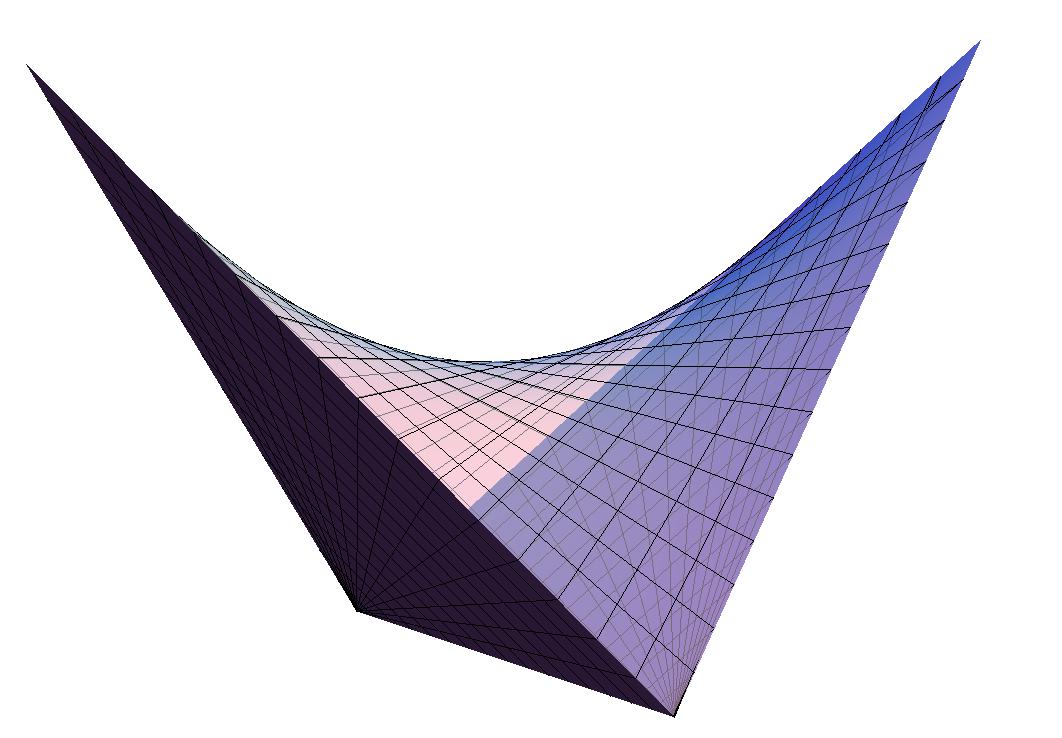}
\vskip -0.5cm
\caption{The set $\E_3$}
\label{E3}
\end{figure}
\end{example}
The boundary of $\E_d$ consists of all parameters $\mathbf{r}$ for which $R(\mathbf{r})$ has an eigenvalue of modulus~$1$. Thus $\partial \E_d$ naturally decomposes into the three hypersurfaces
\begin{align*}
E_d^{(1)} &:= \{\mathbf{r} \in \partial \E_d \;:\; R(\mathbf{r}) \hbox{ has $1$ as an eigenvalue}\}, \\
E_d^{(-1)} &:=  \{\mathbf{r} \in \partial \E_d \;:\; R(\mathbf{r}) \hbox{ has $-1$ as an eigenvalue}\}, \\
E_d^{(\mathbb{C})} &:= 
\{\mathbf{r} \in \partial \E_d \;:\; R(\mathbf{r}) \hbox{ has  a non-real eigenvalue of modulus 1}\},
\end{align*}
{\it i.e.},
\begin{equation}\label{Edboundary}
\partial\E_d=E_d^{(1)} \cup E_d^{(-1)} \cup E_d^{(\mathbb{C})}.
\end{equation}
These sets can be determined using $\E_{d-1}$ and $\E_{d-2}$. To state the corresponding result, we introduce the following terminology. Define for vectors ${\bf r}=(r_0,\ldots,r_{p-1})$, ${\bf s}=(s_0,\ldots,s_{q-1})$ of arbitrary dimension  $p,q\in\NN$ the binary operation $\odot$ by
\begin{equation}\label{eq:odot}
\chi_{{\bf r} \odot {\bf s}}=\chi_{\bf r}\cdot \chi_{\bf s},
\end{equation}
where ``$\cdot$'' means multiplication of polynomials. For $D \subset \RR^p$ and $E \subset \RR^q$ let $D \odot E:=\{{\bf r} \odot {\bf s} : \, {\bf r} \in D, {\bf s} \in E\}$. Then, due to results of Fam and Meditch~\cite{Fam-Meditch:78} (see also~\cite{Kirschenhofer-Pethoe-Surer-Thuswaldner:10}), the following theorem holds.

\begin{theorem}
For $d \geq 3$ we have
\begin{equation}\begin{split}
E_d^{(1)} = & (1) \odot \overline{\E_{d-1}}, \\
E_d^{(-1)} = & (-1) \odot \overline{\E_{d-1}}, \\
E_d^{(\mathbb{C})} = & \left\{(1,y) \;:\; y \in
(-2,2)\right\}\odot \overline{\E_{d-2}}.
\end{split}\end{equation}
Moreover, $E_d^{(1)}$ and $E_d^{(-1)}$ are subsets of hyperplanes while $E_d^{(\mathbb{C})}$ is a hypersurface  in $\mathbb{R}^d$
\end{theorem}

In order to characterize $\D_d$, there remains the problem to describe $\D_d \setminus \E_d$, which is a subset of $\partial \D_d=\partial \E_d$. The problem is relatively simple for $d=1$, where it is an easy exercise to prove that $\D_1=[-1,1]$. For dimensions $d\ge 2$ the situation is different and will be discussed in the following two sections.

\section{The boundary of $\D_2$ and discretized rotations in $\mathbb{R}^2$}\label{sec:D2}

In this section we consider the behavior of the orbits of $\tau_\mathbf{r}$ for $\mathbf{r} \in \partial \mathcal{D}_2$. In particular we are interested in whether these orbits are ultimately periodic or not. To this matter we subdivide the isosceles triangle $\partial \mathcal{D}_2$ into four pieces (instead of three as in \eqref{1902082}), in particular, we split $E_2^{(1)}$ in two parts as follows.
\begin{align*}
E_{2-}^{(1)}& =  \{(x,-x-1) \in \mathbb{R}^2 \;:\; -1 \le x \le 0 \}, \\
E_{2+}^{(1)}& =  \{(x,-x-1) \in \mathbb{R}^2 \;:\; 0 < x \le 1 \}, \\
E_2^{(-1)} &= \{(x,x+1) \in \mathbb{R}^2 \;:\; -1 \le x \le 1 \}, \\
E_2^{(\mathbb{C})}& = \{(1, y) \in \mathbb{R}^2 \;:\; -2 < y < 2 \}.
\end{align*}
It turns out that the behavior of the orbits is much more complicated for $\mathbf{r}\in E_2^{(\mathbb{C})}$ than it is for the remaining cases. This is due to the fact that the matrix $R(\mathbf{r})$ has one eigenvalue that is equal to $-1$ for $\mathbf{r} \in E_2^{(-1)} $, one
eigenvalue that is equal to $1$ for $\mathbf{r} \in E_2^{(1)}$, but a pair of complex conjugate eigenvalues on the unit circle for $\mathbf{r} \in E_2^{(\mathbb{C})}$. Figure~\ref{fig:PartialD2} surveys the known results on the behavior of the orbits of $\tau_{\mathbf{r}}$ for $\mathbf{r}\in\partial \mathcal{D}_2$.

\begin{figure}
\includegraphics[height=7cm]{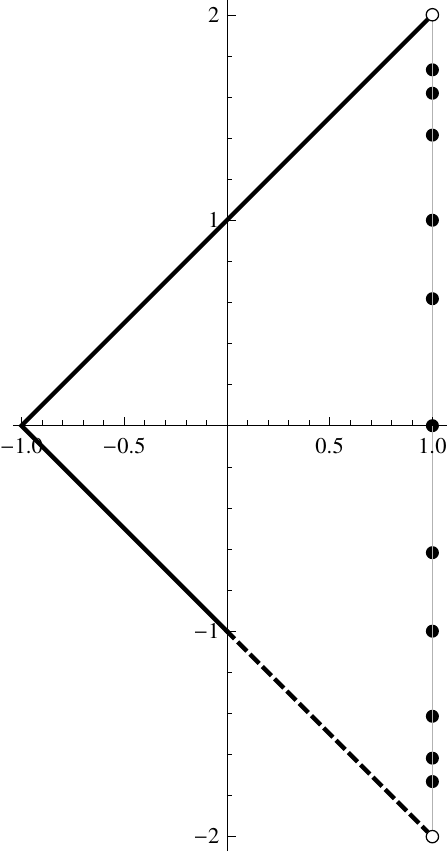}
\caption{An image of the isosceles triangle $\partial \mathcal{D}_2$. The black lines belong to $\mathcal{D}_2$, the dashed line doesn't belong to $\mathcal{D}_2$. For the grey line $E_2^{(\mathbb{C})}$ it is mostly not known whether it belongs to $\mathcal{D}_2$ or not. Only the 11 black points in $E_2^{(\mathbb{C})}$ could be shown to belong to $\mathcal{D}_2$ so far (compare~\cite[Figure~1]{Kirschenhofer-Pethoe-Surer-Thuswaldner:10}). For the two points $(1,2)$ and $(1,-2)$ it is easy to see that they do not belong to $\mathcal{D}_2$ by solving a linear recurrence relation.
 \label{fig:PartialD2}}
\end{figure}

\subsection{The case of real roots}

We start with the easier cases that have been treated in \cite[Section~2]{Akiyama-Brunotte-Pethoe-Thuswaldner:06}. In this paper the following result is proved.

\begin{proposition}[{\cite[Theorem~2.1]{Akiyama-Brunotte-Pethoe-Thuswaldner:06}}]\label{D2easyboundary}
If $\mathbf{r} \in (E_2^{(-1)}  \cup E_{2-}^{(1)})\setminus\{(1,2)\}$ then $\mathbf{r} \in \mathcal{D}_2\setminus \mathcal{D}_2^{(0)}$, \emph{i.e.}, each orbit of $\tau_\mathbf{r}$ is ultimately periodic, but not all orbits end in $\mathbf{0}$.

If $\mathbf{r} \in  E_{2+}^{(1)} \cup \{(1,2)\}$ then $\mathbf{r} \not \in  \mathcal{D}_2$, \emph{i.e.}, there exist orbits of $\tau_\mathbf{r}$ that are not ultimately periodic.
\end{proposition}

\begin{proof}[Sketch of the proof]
We subdivide the proof in four parts.
\begin{itemize}
\item[(i)]$\mathbf{r}=(x,x+1) \in E_{2}^{(-1)}$ with $x \le 0$. The cases $x\in \{-1,0\}$ are trivial, so we can assume that $-1<x<0$.
It is easy to see by direct calculation that $\tau_{\mathbf{r}}^2((-n,n)^t)=(-n,n)^t$ holds for each $n\in\mathbb{N}$. This implies
that $\mathbf{r}\not\in \mathcal{D}_2^{(0)}$ holds in this case. To show that $\mathbf{r}\in \mathcal{D}_2$
one proves that $||\tau_\mathbf{r}(\mathbf{z})||_\infty \le ||\mathbf{z}||_\infty$. This is accomplished by distinguishing four
cases according to the signs of the coordinates of the vector $\mathbf{z}$ and examining $\tau_\mathbf{r}(\mathbf{z})$ in each of these cases.

\item[(ii)] $\mathbf{r}=(x,-x-1) \in E_{2-}^{(1)}$. This is treated in the same way as the previous case; here $\tau_{\mathbf{r}}((n,n)^t)=(n,n)^t$  holds for
each $n\in\mathbb{N}$.

\item[(iii)] $\mathbf{r}=(x,x+1) \in E_{2}^{(-1)}$ with $x > 0$. Here again $\tau_{\mathbf{r}}^2((-n,n)^t)=(-n,n)^t$ implies that $\mathbf{r}\not\in \mathcal{D}_2^{(0)}$. To prove that
$\mathbf{r}\in\mathcal{D}_2$ for $x < 1$ is a bit more complicated. First ultimate periodicity is shown for starting vectors contained in the set $M=\{(-n,m)^t\,:\, m\ge n\ge 0\}$. This is done by an easy induction argument on the quantity $m-n$. After that one shows that each $\mathbf{z} \in \mathbb{Z}^2\setminus M$ hits $M$ after finitely many applications of $\tau_\mathbf{r}$. Proving this is done by distinguishing several cases.
If $x=1$ the fact that $\mathbf{r}\not\in\mathcal{D}_2$ follows by solving a linear recurrence relation.

\item[(iv)] $\mathbf{r}=(x,-x-1) \in E_{2+}^{(1)}$. If $n > m > 0$ then $\tau_\mathbf{r}((m,n)^t) = (n,p)^t$ with $p >n$ follows from the definition of $\tau_{\mathbf{r}}$. Thus
$||\tau_{\mathbf{r}}^k(1,2)||_\infty \to \infty$ for $k\to \infty$ implying that the orbit of $(1,2)^t$ is not ultimately periodic. \qedhere
\end{itemize}
\end{proof}

For $\mathbf{r} \in E_2^{(\mathbb{C})}$ we can only exclude that $\mathbf{r} \in \mathcal{D}_2^{(0)}$. Indeed, in this case $\mathbf{r}=(1,y)$ with $|y| < 2$. This implies that $\tau_\mathbf{r}^{-1}((0,0)^t)=\{(0,0)^t\}$.  In other words, in this case each orbit starting in a non-zero element does not end up at $(0,0)^t$. Combining this with Proposition~\ref{D2easyboundary} we obtain that $\mathcal{D}_2^{(0)} \cap \partial \mathcal{D}_2 = \emptyset$. By Proposition~\ref{EdDdEd} this is equivalent to the following result.

\begin{corollary}[{\cite[Corollary~2.5]{Akiyama-Brunotte-Pethoe-Thuswaldner:06}}]\label{D2boundarycorollary}
$$
\D_2^{(0)} \subset \E_2.
$$
\end{corollary}

\subsection{Complex roots and discretized rotations}
We now turn our attention to periodicity results for parameters $\mathbf{r}\in E_2^{(\mathbb{C})}$, {\it i.e.}, fr $\mathbf{r}=(1,\lambda)$ with $|\lambda| < 2$. From the definition of $\tau_\mathbf{r}$ it follows that $E_2^{(\mathbb{C})} \subset \mathcal{D}_2$ is equivalent to the following conjecture.

\begin{conjecture}[{see {\em e.g.} \cite{Akiyama-Brunotte-Pethoe-Thuswaldner:06,Bruin-Lambert-Poggiaspalla-Vaienti:03,Lowensteinetal:97,Vivaldi:94}}]\label{Vivaldi-SRS-Conjecture}
For each $\lambda \in \mathbb{R}$ satisfying $|\lambda| < 2$ the sequence $(a_n)_{n\in \mathbb{N}}$ defined by
\begin{equation}\label{eq:vivaldiconjecture}
0 \le a_{n-1} + \lambda a_n + a_{n+1} < 1
\end{equation}
is ultimately periodic for each pair of starting values $(a_0,a_1)\in\mathbb{Z}^2$.
\end{conjecture}

\begin{remark}\label{rem:vivaldi}
Several authors (in particular Franco Vivaldi and his co-workers) study the slightly different mapping
$\Phi_\lambda: \mathbb{Z}^2 \to \mathbb{Z}^2$, $(z_0,z_1)^t \mapsto (\lfloor \lambda z_0 \rfloor - z_1, z_0)^t$.
Their results --- which we state using $\tau_{(1,-\lambda)}$ --- carry over to our setting by obvious changes of the proofs.
\end{remark}

To shed some light on Conjecture~\ref{Vivaldi-SRS-Conjecture}, we emphasize that $\tau_{(1,\lambda)}$ can be regarded as a \emph{discretized rotation}. Indeed, the inequalities in \eqref{eq:vivaldiconjecture} imply that
\[
\begin{pmatrix}
a_n \\
a_{n+1}
\end{pmatrix} =
\begin{pmatrix}
0&1\\-1&-\lambda
\end{pmatrix}
\begin{pmatrix}
a_{n-1} \\
a_{n}
\end{pmatrix}+
\begin{pmatrix}
0 \\
\{\lambda a_n\}
\end{pmatrix},
\]
and writing $\lambda=-2\cos(\pi \theta)$ we get that the eigenvalues of the involved matrix are given by $\exp({\pm i\pi\theta})$. Thus $\tau_{(1,\lambda)}$ is a rotation followed by a ``round-off''. As in computer arithmetic round-offs naturally occur due to the limited accuracy of floating point arithmetic, it is important to worry about the effect of such round-offs. It was this application that Vivaldi and his co-authors had in mind when they started to study the dynamics of the mapping $\tau_{(1,\lambda)}$ in the 1990s (see \cite{Lowensteinetal:97,Lowenstein-Vivaldi:98,Vivaldi:94}). Of special interest is the case of rational rotation angles $\theta=p/q$, as rotations by these angles have periodic orbits with period $q$. In these cases, for the discretization one gets that $||\tau_{(1,\lambda)}^q(\mathbf{z}) - \mathbf{z}||_\infty$ is uniformly small for all $\mathbf{z}\in\mathbb{Z}^2$. The easiest non-trivial cases (if $\lambda\in\mathbb{Z}$ everything is trivial) occur for $q=5$. For instance, consider $\theta=\frac25$ which gives $\lambda=\frac{1-\sqrt{5}}2=-\frac1\varphi$, where $\varphi=\frac{1+\sqrt{5}}2$ is the golden ratio. Although the behavior of the orbits of $\tau_{(1,-1/\varphi)}$ looks rather involved and there is no upper bound on their period, Lowenstein {\it et al.}~\cite{Lowensteinetal:97} succeeded in proving that nevertheless each orbit of $\tau_{(1,-1/\varphi)}$ is periodic. This confirms Conjecture~\ref{Vivaldi-SRS-Conjecture} in the case $\lambda=-\frac1\varphi$. In their proof, they consider a dynamical system on a subset of the torus, which is conjugate to $\tau_{(1,-1/\varphi)}$ (see Section~\ref{sec:domainexchange} for more details). This system is then embedded in a higher dimensional space where the dynamics becomes periodic. This proves that $\tau_{(1,-1/\varphi)}$ is \emph{quasi-periodic} which is finally used in order to derive the result. The case $\theta=\frac45$ corresponds to $\tau_{(1,\varphi)}$ and is treated in detail in the next subsection.

\subsection{A parameter associated with the golden ratio}
Akiyama {\it et al.}~\cite{Akiyama-Brunotte-Pethoe-Steiner:06} came up with a very simple and beautiful proof for the periodicity of $\tau_{(1,\varphi)}$. In particular, in their argument they combine the fact that $||\tau_{(1,\varphi)}^5(\mathbf{z}) - \mathbf{z}||_\infty$ is small (see the orbits in Figure~\ref{fig:goldenorbit})
 \begin{figure}
\includegraphics[height=6cm]{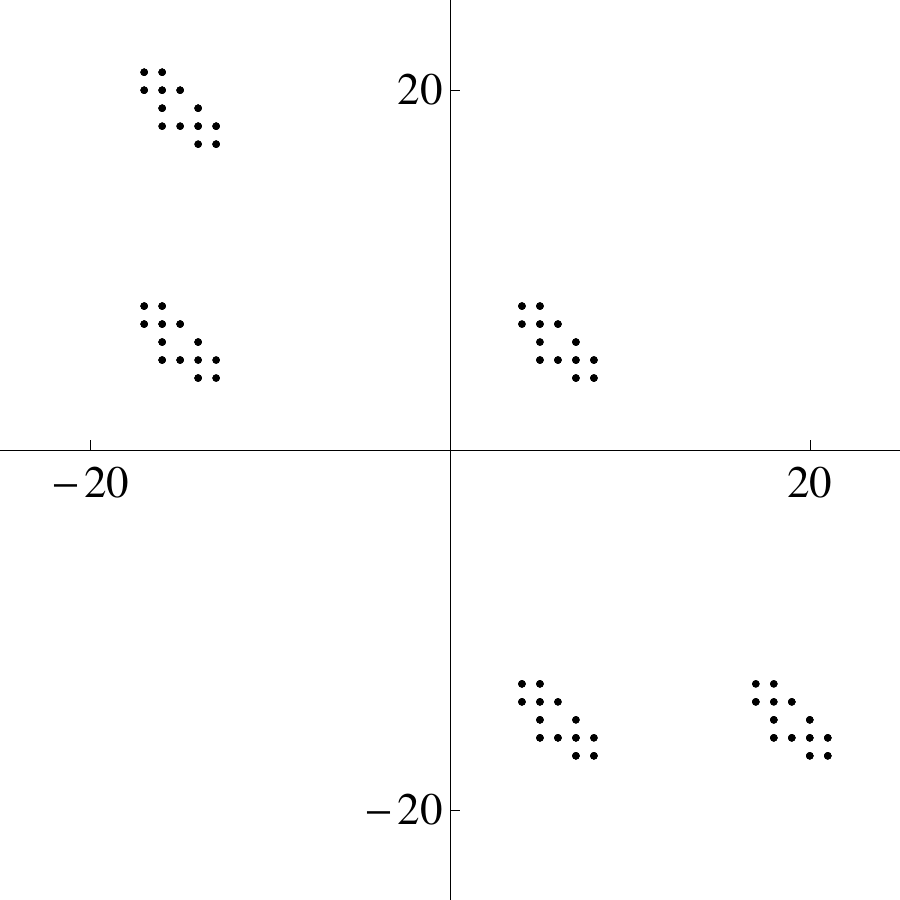} \hskip 1cm
\includegraphics[height=6cm]{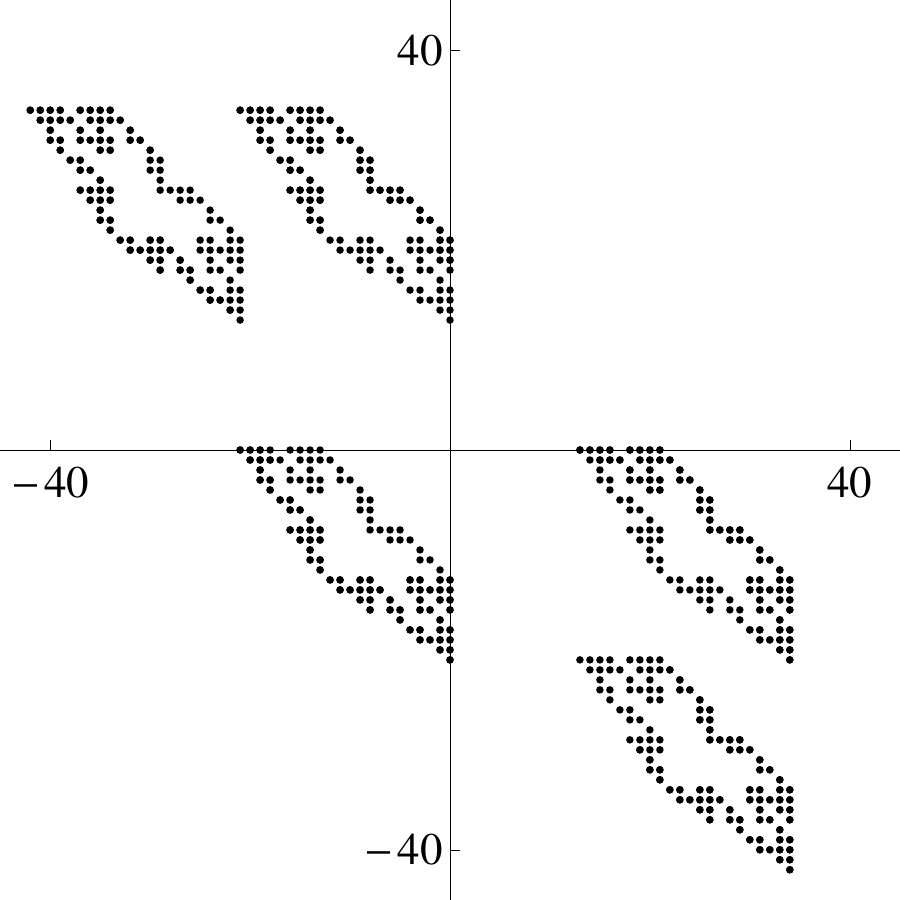}
\caption{Two examples of orbits of $\tau_{(1,\varphi)}$. The picture on the left is the orbit of  $(5,5)^t$. Its period is 65. The orbit of $(13,0)^t$ on the right has period 535.  \label{fig:goldenorbit}}
\end{figure}
with Diophantine approximation properties of the golden mean. In what follows we state the theorem and give a sketch of this proof (including some explanations to make it easier to read).

\begin{theorem}[{\cite[Theorem~5.1]{Akiyama-Brunotte-Pethoe-Steiner:06}}]\label{th:steinerproof}
Let $\varphi=\frac{1+\sqrt{5}}2$ be the golden mean. Then $(1,\varphi)\in \mathcal{D}_2$. In other words, the sequence $(a_n)_{n \in \mathbb{N}}$ defined by
\begin{equation}\label{eq:phirecurrence}
0 \le a_{n-1} + \varphi a_n + a_{n+1} < 1
\end{equation}
is ultimately periodic for each pair of starting values $(a_0,a_1)\in\mathbb{Z}^2$.
\end{theorem}

\begin{proof}[Sketch of the proof ({\em cf.}\ {\cite{Akiyama-Brunotte-Pethoe-Steiner:06}})]
First we make precise our observation that $\tau_{(1,\varphi)}^5({\mathbf z})$ is close to $\mathbf{z}$ for each $\mathbf{z} \in \mathbb{Z}^2$. In particular, we shall give an upper bound for the quantity $|a_{n+5} - a_n|$ when $(a_n)_{n\in\mathbb{Z}}$ satisfies \eqref{eq:phirecurrence}.  These inequalities immediately yield
\begin{equation}\label{eq:phi1}
a_{n+1} = -a_{n-1} - \varphi a_n + \{\varphi a_n\},
\end{equation}
a representation of $a_{n+1}$ in terms of $a_n$ and $a_{n-1}$. We wish to express $a_{n+5}$ by these two elements. To accomplish this we have to proceed in four steps. We start with a representation of $a_{n+2}$ in terms of $a_n$ and $a_{n-1}$. As
\begin{equation}\label{eq:phi2}
a_{n+2} = -\lfloor a_{n} + \varphi a_{n+1} \rfloor
\end{equation}
we first calculate
\begin{equation}\label{eq:phi2a}
\begin{array}{rclcl}
-a_n-\varphi a_{n+1} & = & (-1+\varphi^2) a_n + \varphi a_{n-1} - \varphi\{\varphi a_n\} &\;& \hbox{(by \eqref{eq:phi1})} \\
&=&\varphi a_n + \varphi a_{n-1}  - \varphi\{\varphi a_n\}  &\;& (\hbox{by }\varphi^2=\varphi+1) \\
&=& \lfloor\varphi a_n  \rfloor  +  \varphi a_{n-1} + (1-\varphi)\{\varphi a_n\} &&\\
&=&\lfloor  \varphi a_{n}  \rfloor + \lfloor  \varphi a_{n-1}  \rfloor   -  \varphi^{-1}\{\varphi a_n\}  +\{\varphi a_{n-1}\} &\;& (\hbox{by }\varphi^2=\varphi+1).
\end{array}
\end{equation}
Inserting \eqref{eq:phi2a} in \eqref{eq:phi2} we obtain
\[
a_{n+2} = \lfloor  \varphi a_{n}  \rfloor + \lfloor  \varphi a_{n-1}  \rfloor + c_n,
\quad
\hbox{where}
\quad
c_n = \begin{cases}
1, & \hbox{if }   \varphi^{-1}\{\varphi a_n\}  < \{\varphi a_{n-1}\};\\
0, & \hbox{if }   \varphi^{-1}\{\varphi a_n\}  \ge \{\varphi a_{n-1}\},
\end{cases}
\]
the desired representation of $a_{n+2}$. We can now go on like that for three more steps and successively gain representations of $a_{n+3}$, $a_{n+4}$, and $a_{n+5}$ in terms of $a_n$ and $a_{n-1}$. The formula for $a_{n+5}$, which is relevant for us, reads
\begin{equation}\label{eq:phi5}
a_{n+5} = a_n + d_n
\end{equation}
where
\[
d_n = \begin{cases}
1, &  \hbox{if }  \{\varphi a_{n-1}\} \ge \varphi\{\varphi a_{n}\},\;    \{\varphi a_{n-1}\}+\{\varphi a_{n}\} >1   \\
    &  \hbox{or } \varphi\{\varphi a_{n-1}\} \le\{\varphi a_{n}\},\;      \varphi\{\varphi a_{n}\}  \le  1 + \{\varphi a_{n-1}\};    \\
0, &  \hbox{if }   \varphi \{\varphi a_{n-1}\}  > \{\varphi a_{n}\},\;    \{\varphi a_{n-1}\}+\{\varphi a_{n}\} \le 1,\;   \varphi^2 \{\varphi a_{n-1}\}<1    \\
    &  \hbox{or } \{\varphi a_{n}\} < \varphi\{\varphi a_{n-1}\} < \varphi^2\{\varphi a_{n}\},\;    \{\varphi a_{n-1}\}+\{\varphi a_{n}\}>1;  \\
-1,&  \hbox{if }  \varphi\{\varphi a_{n-1}\}>\{\varphi a_{n}\},\;         \{\varphi a_{n-1}\}+\{\varphi a_{n}\}\le 1,\;  \varphi^2\{\varphi a_{n-1}\}\ge 1     \\
    &  \hbox{or } \varphi\{\varphi a_{n}\}> 1+\{\varphi a_{n-1}\}.    \\
\end{cases}
\]
In particular, this implies that
\begin{equation}\label{eq:diffbound}
|a_{n+5}-a_{n}| \le 1.
\end{equation}

To conclude the proof we use the Fibonacci numbers $F_k$ defined by $F_0=0$, $F_1=1$ and $F_k=F_{k-1}+F_{k-2}$ for $k \ge 2$.
In particular, we will use the classical identity (see {\it e.g.} \cite[p.12]{Rockett-Szusz:92})
\begin{equation}\label{eq:fiboid}
\varphi F_k = F_{k+1} + \frac{(-1)^{k+1}}{\varphi^k}.
\end{equation}
Let $(a_n)_{n\in \mathbb{N}}$ be an arbitrary sequence satisfying \eqref{eq:phirecurrence} and choose $m \in \mathbb{N}$ in a way that $a_n \le F_{2m}$ holds for $n\in \{0,\ldots, 5\}$. We claim that in this case $a_n \le F_{2m}$ holds for all $n\in \mathbb{N}$. Assume on the contrary, that this is not true. Then there is a smallest $n \in \mathbb{N}$ such that $a_{n+5} > F_{2m}$. In order to derive a contradiction, we distinguish two cases. Assume first that $a_n < F_{2m}$. In this case \eqref{eq:diffbound} implies that $a_{n+5} \le F_{2m}$, which already yields the desired contradiction. Now assume that $a_n = F_{2m}$. Here we have to be more careful. First observe that \eqref{eq:phirecurrence} implies that $\varphi a_{n-1} \ge -a_{n-2}- a_n \ge -2F_{2m}$ and, hence, \eqref{eq:fiboid} yields $a_{n-1} \ge - 2\varphi^{-1} F_{2m}= -2\varphi^{-2}F_{2m+1}+2\varphi^{-2m-2} > -F_{2m+1}$. Summing up, we have
$-F_{2m+1} < a_{n-1} \le F_{2m}.$
As the Fibonacci numbers are the denominators of the convergents of the continued fraction expansion $\varphi=[1;1,1,1,\ldots]$ ({\it cf. e.g.} \cite[p.12]{Rockett-Szusz:92}) we obtain that
\begin{equation}\label{cfConsequence}
\{\varphi a_{n}\} \le \{\varphi a_{n-1}\} \le 1 - \{\varphi a_{n}\}.
\end{equation}
This chain of inequalities rules out the case $d_{n}=1$ in \eqref{eq:phi5}. Thus we get $a_{n+5} - a_n \in \{-1,0\}$, hence, $a_{n+5} \le a_{n} \le F_{2m}$, and we obtain a contradiction again.

We have now shown that $a_n \le F_{2m}$ holds for each $n\in \mathbb{N}$. However, in view of \eqref{eq:phi1},  $a_{n+1}\le F_{2m}$ implies that $a_{n} \ge -(1+\varphi)F_{2m}-1$, which yields that $-(1+\varphi)F_{2m}-1 \le a_n \le F_{2m}$ holds for each $n\in\mathbb{N}$. Thus, the orbit $\{a_n\,:\, n\in\mathbb{N}\}$ is bounded and, hence, there are only finitely many possibilities for the pairs $(a_n,a_{n+1})$. In view of \eqref{eq:phirecurrence} this implies that $(a_n)_{n\in \mathbb{N}}$ is ultimately periodic.
\end{proof}

This proof depends on the very particular continued fraction expansion of the golden ratio $\varphi$. It seems not possible to extend this argument to other parameters (apart from its conjugate $\varphi'=\frac{1-\sqrt{5}}{2}$). In fact, inequalities of the form \eqref{cfConsequence} do not hold any more and so it cannot be guaranteed that the orbit does not ``jump'' over the threshold values.

\subsection{Quadratic irrationals that give rise to rational rotations}\label{sec:domainexchange}

One of the most significant results on Conjecture~\ref{Vivaldi-SRS-Conjecture} is contained in Akiyama~{\it et al.}~\cite{Akiyama-Brunotte-Pethoe-Steiner:07}. It reads as follows.

\begin{theorem}
If $\lambda \in \left\{ \frac{\pm1\pm\sqrt 5}2, \pm\sqrt 2, \pm \sqrt 3\right\}$, then $ (1,\lambda)\in \D_2$, {\em i.e.}, each orbit of $\tau_{(1,\lambda)}$ is ultimately periodic.
\end{theorem}

Observe that this settles all the instances $\lambda= -2\cos (\theta \pi)$ with rational rotation angle $\theta$ such that $\lambda$ is a quadratic irrational. The proof of this result is very long and tricky. In what follows, we outline the main ideas. The proof runs along the same lines for each of the eight values of $\lambda$. As the technical details are simpler for $\lambda = \varphi=\frac{1 + \sqrt{5}}{2}$ we use this instance as a guideline.

As in the first proof of periodicity in the golden ratio case given in \cite{Lowensteinetal:97}, a conjugate dynamical system that was also studied in Adler {\it et al}.~\cite{AdlerKitchensTresser:01} is used here. Indeed, let $((a_{k-1},a_k)^t)_{k\in \mathbb{N}}$ be an orbit of $\tau_{(1,\varphi)}$ and set $x=\{\varphi a_{k-1}\}$ and $y=\{\varphi a_{k}\}$. Then, by the definition of $\tau_{(1,\varphi)}$ we have that
\[
\{\varphi a_{k+1}\} = \{-\varphi a_{k-1} - \varphi^2 a_{k} + \varphi y\} = \{-x+(\varphi-1)y\}=\{-x-\varphi'y\}
\]
where $\varphi'=\frac{1 - \sqrt{5}}{2}$ is the algebraic conjugate of $\varphi$. Thus (the according restriction of) the mapping
\[
T:[0,1)^2 \to [0,1)^2, \quad (x,y) \mapsto (y, \{-x-\varphi'y\})
\]
is conjugate to $\tau_{(1,\varphi)}$ and it suffices to study the orbits of elements of $\mathbb{Z}[\varphi]^2\cap [0,1)^2$ under $T$ to prove the conjecture. (Computer assisted work on almost all orbits on $T$ was done in~\cite{Kouptsov-Lowenstein-Vivaldi:02}; however the results given there are not strong enough to imply the above theorem.) Let $A=\begin{pmatrix}0&-1\\ 1& 1/\varphi \end{pmatrix}$ and write $T$ in the form
\begin{equation}\label{Tcases}
T(x,y) = \begin{cases}
(x,y)A,& \hbox{for }y \ge \varphi x \\
(x,y)A + (0,1),  & \hbox{for }y < \varphi x.
\end{cases}
\end{equation}
We will now iteratively determine pentagonal subregions of $[0,1)^2$ whose elements admit periodic orbits under the mapping $T$. First define the pentagon
\[
R=\{(x,y)\in [0,1)^2\;:\; y <\varphi x,\ x+y > 1,\ x<\varphi y\}
\]
and split the remaining part $D = [0,1)^2 \setminus R$ according to the cases in the definition of $T$ in \eqref{Tcases}, {\it i.e.}, set
\begin{align*}
D_0 & = \{(x,y) \in D \;:\, y \ge \varphi x\} \setminus \{0,0\}, \\
D_1 &= D \setminus D_0.
\end{align*}
Using the fact that $A^5$ is the identity matrix one easily derives that $T^5(\mathbf{z})=\mathbf{z}$ for each $\mathbf{z} \in R$. This exhibits the first pentagon of periodic points of $T$. We will now use first return maps of $T$ on smaller copies of $[0,1)^2$ to exhibit more periodic pentagons.

\begin{figure}
\includegraphics[height=4cm]{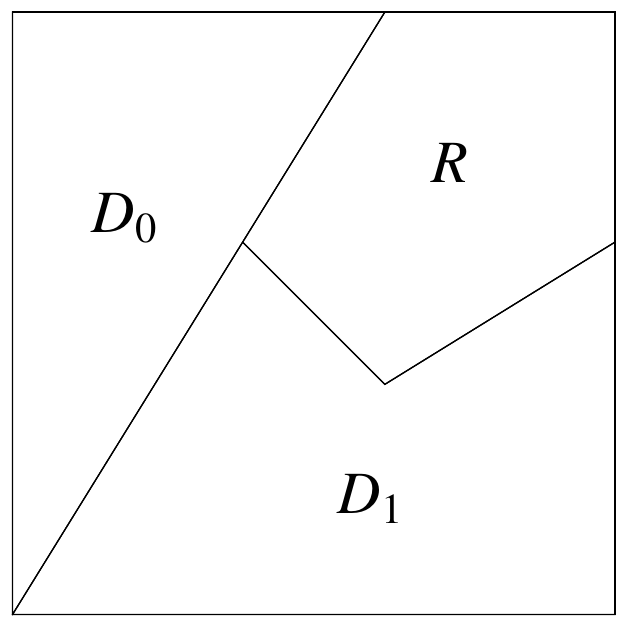} \hskip 1cm
\includegraphics[height=4cm]{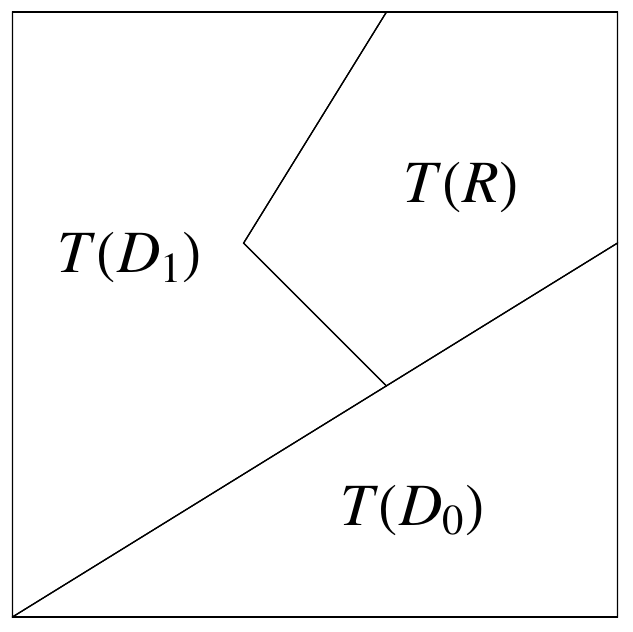} \hskip 1cm
\caption{The effect of the mapping $T$ on the regions $R$, $D_0$, and $D_1$ (compare~\cite[Figure~2.1]{Akiyama-Brunotte-Pethoe-Steiner:07}).
\label{fig:steiner1}}
\end{figure}

To this matter we first observe that by the definition of $D_0$ and $D_1$,  the mapping $T$ acts in an easy way on these sets (see Figure~\ref{fig:steiner1}).  Now we scale down $D$ by a factor $1/\varphi^2$ and  and follow the $T$-trajectory of each $\mathbf{z} \in D$ until it hits $D/\varphi^2$. First we determine all $\mathbf{z} \in D$ that never hit $D/\varphi^2$. By the mapping properties illustrated in Figure~\ref{fig:steiner1} one easily obtains that the set of these parameters is the union $P$ of the the five shaded pentagons drawn in Figure~\ref{fig:steiner2}.
\begin{figure}
\includegraphics[height=8cm]{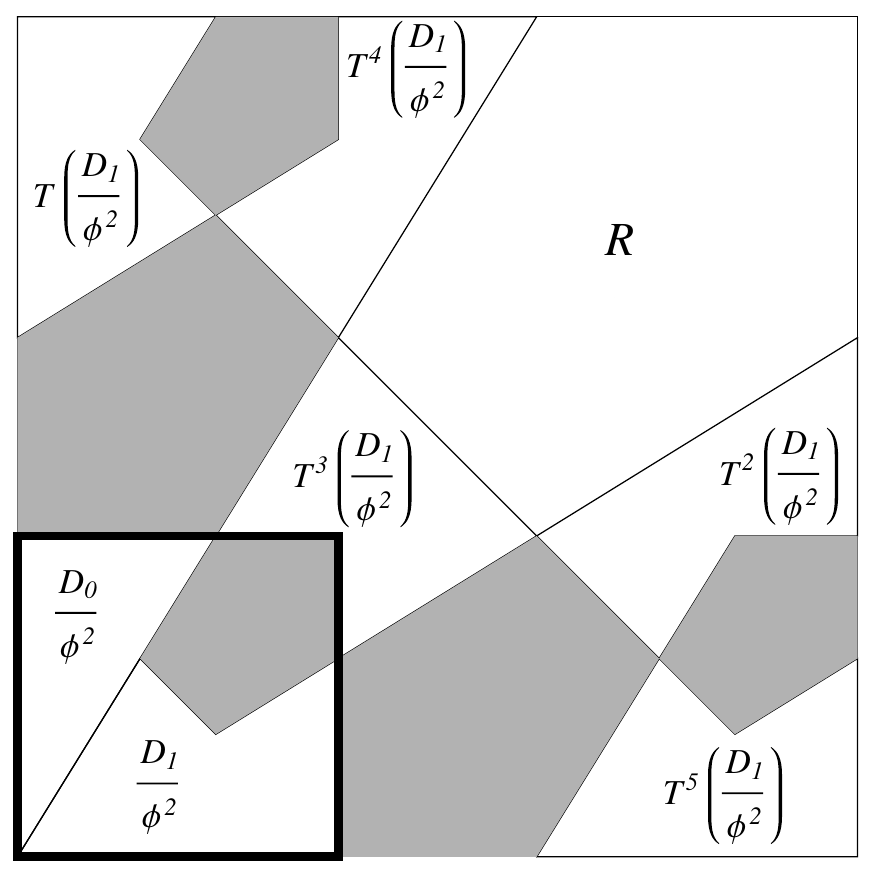} \hskip 1cm
\caption{The region of induction (black frame) of $T$ and the five (shaded) rectangles containing points with periodic orbits of $T$ (compare~\cite[Figure~2.2]{Akiyama-Brunotte-Pethoe-Steiner:07}).
\label{fig:steiner2}}
\end{figure}
Again we can use the mapping properties of Figure~\ref{fig:steiner1} to show that all elements of $P$ are periodic under $T$. Thus, to determine the periodic points of $T$ it is enough to study the map induced by $T$ on $D/\varphi^2$. The mapping properties of this induced map on the subsets $D_0/\varphi^2$ and $D_1/\varphi^2$ are illustrated in Figure~\ref{fig:steiner3}. They are (apart from the scaling factor) the same as the ones in Figure~\ref{fig:steiner1}.
\begin{figure}
\includegraphics[height=4cm]{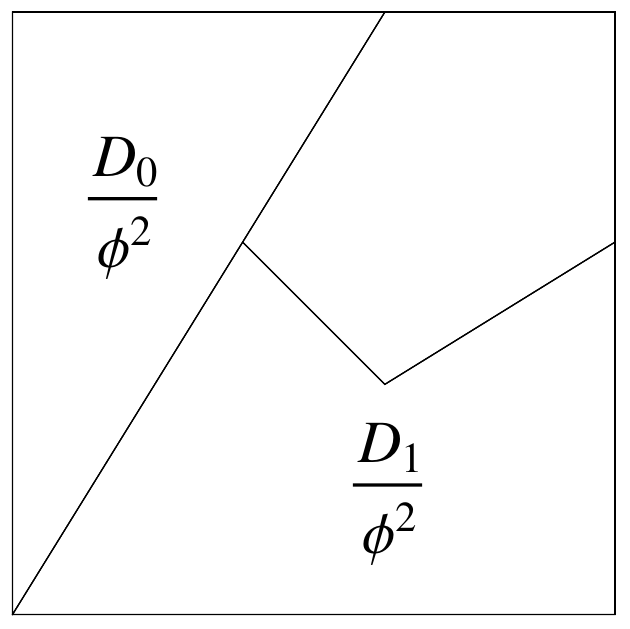} \hskip 1cm
\includegraphics[height=4cm]{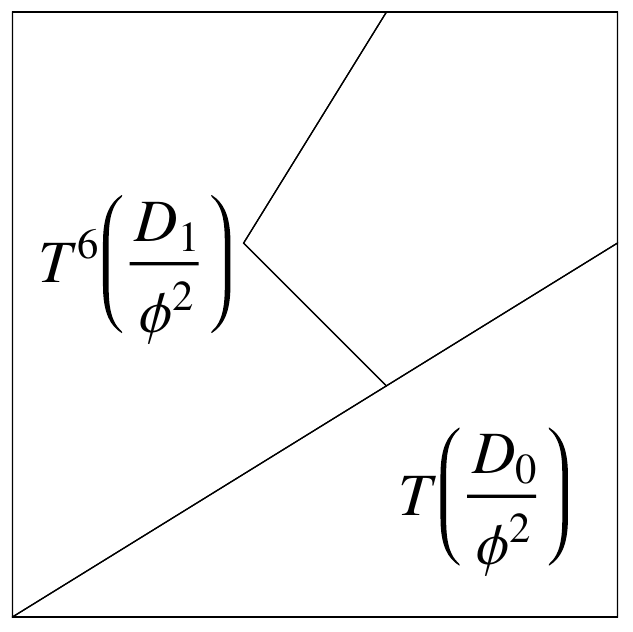} \hskip 1cm
\caption{The effect of the mapping $T$ on the induced regions $D_0/\varphi^2$ and $D_1/\varphi^2$. As can be seen by looking at the lower left corner of Figure~\ref{fig:steiner1},  the region $D_0/\varphi^2$ is mapped into the induced region $D/\varphi^2$ after one application of $T$. To map  $D_1/\varphi^2$ back to  $D/\varphi^2$ we need to apply the sixth iterate of $T$ (see Figure~\ref{fig:steiner2}, where $T^i(D_0/\varphi^2)$ is visualized for $i\in\{0,1,2,3,4,5\}$). The induced mapping on $D/\varphi^2$ shows the same behavior as $T$ on $D$, thus we say that $T$ is {\it self-inducing}.
\label{fig:steiner3}}
\end{figure}

Now we can iterate this procedure by defining a sequence of induced maps on $D/\varphi^{2k}$ each of which exhibits $5^k$ more pentagonal pieces of periodic points of $T$ in $[0,1)^2$. To formalize this, let $s(\mathbf{z})=\min\{m\in \mathbb{N}\,:\, T^m(\mathbf{z}) \in D/\varphi^2\}$ and define the mapping
\[
S: D \setminus P \to D, \quad \mathbf{z} \mapsto \varphi^2 T^{s(\mathbf{z})}(\mathbf{z}).
\]
The above mentioned iteration process then leads to the following result.

\begin{lemma}[{\cite[Theorem~2.1]{Akiyama-Brunotte-Pethoe-Steiner:07}}]\label{lem:steiner1}
The orbit $(T^k(\mathbf{z}))_{k\in \mathbb{N}}$ is periodic if and only if $\mathbf{z}\in R$ or $S^n(\mathbf{z})\in P$ for some $n \ge 0$.
\end{lemma}

This result cannot only be used to characterize all periodic points of $T$ in $[0,1)^2$, it even enables one to determine the exact periods (see \cite[Theorem~2.3]{Akiyama-Brunotte-Pethoe-Steiner:07}). An approximation of the set $X \subset [0,1)^2$ of aperiodic points of $T$ is depicted in Figure~\ref{fig:steiner4}.
\begin{figure}
\includegraphics[height=8cm]{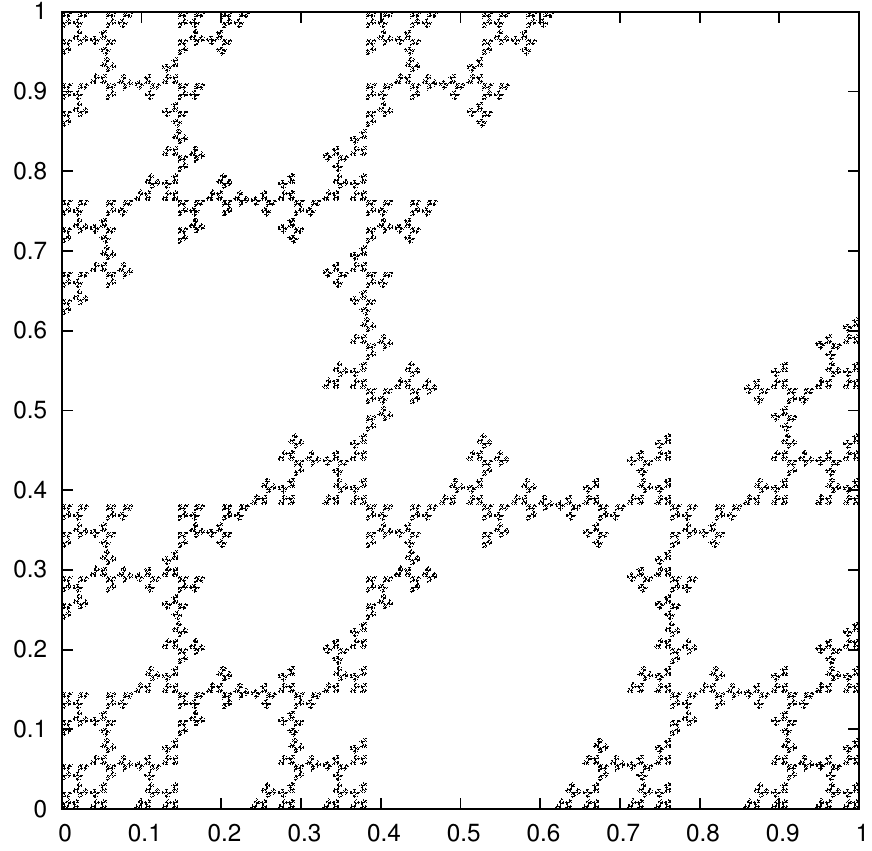} \hskip 1cm
\caption{The aperiodic set $X$ (see \cite[Figure~2.3]{Akiyama-Brunotte-Pethoe-Steiner:07}).
\label{fig:steiner4}}
\end{figure}
To prove the theorem it suffices to show that $X \cap \mathbb{Z}[\varphi]^2 = \emptyset$. By representing the elements of $X$ with help of some kind of digit expansion one can prove that this can be achieved by checking periodicity of the orbit of each point contained in a certain finite subset of $X$. This finally leads to the proof of the theorem. We mention that Akiyama and Harriss~\cite{AH:13} show that the dynamics of $T$ on the set of aperiodic points $X$ is conjugate to the addition of $1$ on the set of 2-adic integers $\mathbb{Z}_2$.

Analogous arguments are used to prove the other cases. However, the technical details get more involved (the worst case being $\pm\sqrt 3$). The most difficult part is to find a suitable region of induction for $T$ and no ``rule'' for its shape is known.  Although the difficulty of the technical details increases dramatically, similar arguments should be applicable for an arbitrary algebraic number $\alpha$ that induces a rational rotation. However, if $d$ is the degree of $\alpha$ the dynamical system $T$ acts on the set $[0,1)^{2d-2}$. This means that already in the cubic case we have to find a region of induction for $T$ in a $4$-dimensional torus.

\subsection{Rational parameters and $p$-adic dynamics}
Bosio and Vivaldi~\cite{Bosio-Vivaldi:00} study\footnote{see Remark~\ref{rem:vivaldi}} $\tau_{(1,\lambda)}$ for parameters $\lambda = q/p^n$ where $p$ is a prime and $q\in\mathbb{Z}$ with $|q|<2p^n$. They exhibit an interesting link to $p$-adic dynamics for these parameters. Before we give their main result we need some preparations.

For $p$ and $q$ given as above consider the polynomial
\[
p^{2n}\chi_{(1,\lambda)}\left(\frac{X}{p^n}\right)= X^2 +qX + p^{2n}.
 \]
If we regard this as a polynomial over the ring $\mathbb{Z}_p$ of $p$-adic integers, by standard methods from algebraic number theory one derives that it has two distinct roots $\theta,\theta' \in \mathbb{Z}_p$. Obviously we have
\begin{equation}\label{padicroots}
\theta \theta' = p^{2n} \quad\hbox{and}\quad \theta + \theta' = -q.
\end{equation}
With help of $\theta$ we now define the mapping
\begin{equation}\label{Ell}
\mathcal{L}: \mathbb{Z}^2 \to \mathbb{Z}_p,\quad (x,y)^t \mapsto  y-\frac{\theta x}{p^n}.
\end{equation}
If $\sigma:\mathbb{Z}_p \to \mathbb{Z}_p$ denotes the shift mapping
\[
\sigma\left(\sum_{i\ge 0} b_ip^i\right) =  \sum_{i\ge 0} b_{i+1}p^i
\]
we can state the following conjugacy of the SRS $\tau_{(1,\lambda)}$ to a mapping on $\mathbb{Z}_p$.

\begin{theorem}[{\cite[Theorem~1]{Bosio-Vivaldi:00}}]\label{BosioThm}
Let $p$ be a prime number and $q\in \mathbb{Z}$ with $|q|<2p^n$, and set $\lambda=q/p^n$. The function $\mathcal{L}$ defined in \eqref{Ell} embeds $\mathbb{Z}^2$ densely into $\mathbb{Z}_p$. The mapping $\tau^*_{(1,\lambda)}= \mathcal{L} \circ \tau_{(1,\lambda)}\circ \mathcal{L}^{-1}:\mathcal{L}(\mathbb{Z}^2)\to \mathcal{L}(\mathbb{Z}^2)$ is therefore conjugate to $ \tau_{(1,\lambda)}$. It can be extended continuously to $\mathbb{Z}_p$ and has the form
\[
\tau^*_{(1,\lambda)}(\psi) = \sigma^n(\theta' \psi).
\]
\end{theorem}

\begin{proof}[Sketch of the proof {\rm (see \cite[Proposition~4.2]{Bosio-Vivaldi:00})}]
The fact that $\mathcal{L}$ is continuous and injective is shown in \cite[Proposition~4.1]{Bosio-Vivaldi:00}. We establish the formula for $\tau^*_{(1,\lambda)}$ which immediately implies the existence of the continuous extension to $\mathbb{Z}_p$.

Let $\psi = y-\frac{\theta x}{p^n} \in \mathcal{L}(\mathbb{Z}^2) \subset \mathbb{Z}_p$ be given. Noting that $\lfloor qy/p^n \rfloor =  qy/p^n - c/p^n$ for some $c \equiv qy\; ({\rm mod}\, p^n )$  and using \eqref{padicroots} we get
\begin{align*}
\tau^*_{(1,\lambda)}(\psi) &= \mathcal{L}\circ \tau_{(1,\lambda)}((x,y)^t)  = \mathcal{L}\left(\left(y,-x-\left\lfloor \frac{qy}{p^n} \right\rfloor \right)^t\right)
= -x-\left\lfloor \frac{qy}{p^n} \right\rfloor - \frac{\theta y}{p^n} \\
&= \frac{1}{p^n}\left( -p^n x - (q+\theta)y     + c        \right)
= \frac{1}{p^n} \left( -\frac{\theta'\theta}{p^n} x+ \theta' y + c \right) = \frac{1}{p^n}(\theta' \psi + c).
\end{align*}
One can show that $z\in\mathbb{Z}$ inplies $\theta z \equiv 0\; ({\rm mod}\, p^{2n} )$, thus $c \equiv qy \equiv q\psi \equiv -\theta' \psi \; ({\rm mod}\, p^{n} )$ and the result follows.
\end{proof}

Theorem~\ref{BosioThm} is used by Vivaldi and Vladimirov~\cite{Vivaldi-Vladimirov:03} to set up a probabilistic model for the cumulative round off error caused by the floor function under iteration of $\tau_{(1,\lambda)}$. Furthermore, they prove a central limit theorem for this model.

\subsection{Newer developments}

We conclude this section with two new results related to Conjecture~\ref{Vivaldi-SRS-Conjecture}. Very recently, Akiyama and Peth\H{o}~\cite{Akiyama-Pethoe:13} proved the following very general result.

\begin{theorem}[{\cite[Theorem~1]{Akiyama-Pethoe:13}}]\label{AP13}
For each fixed $\lambda\in(-2,2)$ the mapping $\tau_{(1,\lambda)}$ has infinitely many periodic orbits.
\end{theorem}

The proof is tricky and uses the fact that (after proper rescaling of the lattice $\mathbb{Z}^2$) each unbounded orbit of  $\tau_{(1,\lambda)}$ has to hit a so called ``trapping region'' ${\rm Trap}(R)$ which is defined as the symmetric difference of two circles of radius $R$ whose centers have a certain distance (not depending on $R$) of each other. The proof is done by contradiction. If one assumes that there are only finitely many periodic orbits, there exist more unbounded orbits hitting ${\rm Trap}(R)$ than there are lattice points in ${\rm Trap}(R)$ if $R$ is chosen large enough. This contradiction proves the theorem.

Using lattice point counting techniques this result can be extended to variants of SRS (as defined in Section~\ref{sec:variants}), see \cite[Theorem~2]{Akiyama-Pethoe:13}.

\medskip

Reeve-Black and Vivaldi~\cite{Reeve-Black-Vivaldi:13} study the dynamics of\footnote{see Remark~\ref{rem:vivaldi}} $\tau_{(1,\lambda)}$ for $\lambda\to 0$, $\lambda < 0$. While the dynamics of $\tau_{(1,0)}$ is trivial, for each fixed small positive parameter $\lambda$ one observes that the orbits approximate polygons as visualized in Figure~\ref{fig:black}.
\begin{figure}
\includegraphics[height=6cm]{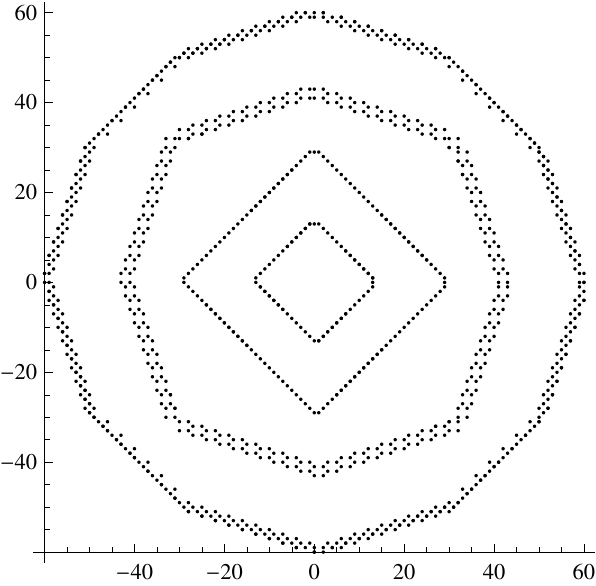} \hskip 1cm
\caption{Some examples of orbits of $\tau_{(1,1/50)}$.
\label{fig:black}}
\end{figure}
Close to the origin the orbits approximate squares, however, the number of edges of the polygons increases the farther away from the origin the orbit is located. Eventually, the orbits approximate circles. Moreover, the closer to zero the parameter $\lambda$ is chosen, the better is the approximation of the respective polygons (after a proper rescaling of the mappings $\tau_{(1,\lambda)}$). This behavior can be explained by using the fact that $\tau^4_{(1,\lambda)}(\mathbf{z})$ is very close to $\mathbf{z}$ for small values of $\lambda$.

The idea in \cite{Reeve-Black-Vivaldi:13} is now to construct a near integrable Hamiltonian function $P:\mathbb{R}^2 \to \mathbb{R}$ that models this behavior. In particular, $P$ is set up in a way that the orbits of the flow associated with the Hamiltonian vector field $(\partial P / \partial y, -\partial P/\partial x)$ are polygons. Moreover, if such a polygon passes through a lattice point of $\mathbb{Z}^2$ it is equal to the corresponding polygon approximated by the discrete systems.
Polygonal orbits of the flow passing through a lattice point are called {\it critical}. Critical polygons are used to separate the phase space into infinitely many {\it classes}.

There is a crucial difference between the orbits of the discrete systems and their Hamiltonian model: the orbits of the Hamiltonian flow surround a polygon once and then close up. As can be seen in Figure~\ref{fig:black} this need not be the case for the orbits of $\tau_{(1,\lambda)}$. These may well ``surround'' a polygon more often (as can be seen in the two outer orbits of Figure~\ref{fig:black}). This behavior leads to long periods and makes the discrete case harder to understand. Discrete orbits that surround a polygon only once are called {\it simple}. They are of particular interest because they {\it shadow} the orbit of the Hamiltonian system and show some kind of structural stability.

The main result in \cite{Reeve-Black-Vivaldi:13} asserts that there are many simple orbits. In particular, there exist infinitely many classes in the above sense, in which a positive portion of the orbits of $\tau_{(1,\lambda)}$ are simple for small values of $\lambda$. The numerical value for this portion can be calculated for $\lambda\to0$. These classes can be described by divisibility properties of the coordinates of the lattice points contained in a critical polygon (see~\cite[Theorems~A and~B]{Reeve-Black-Vivaldi:13}).

\section{The boundary of $\mathcal{D}_d$ and periodic expansions w.r.t.\ Salem numbers}

In Section~\ref{sec:D2} we considered periodicity properties of the orbits of $\tau_{\mathbf{r}}$ for $\mathbf{r}\in \partial\D_2$. While we get complete results for the regions $E_2^{(1)}$ and $E_2^{(-1)}$, the orbits of $\tau_\mathbf{r}$ for $\tau_{\mathbf{r}}\in E_2^{(\mathbb{C})}$ are hard to study and their periodicity is known only for a very limited number of instances. In the present section we want to discuss periodicity results for orbits of $\tau_\mathbf{r}$ with $\mathbf{r}\in \partial\D_d$. Here, already the ``real'' case is difficult. In Section~\ref{sec:realDd} we review a result due to Kirschenhofer {\it et al.}~\cite{Kirschenhofer-Pethoe-Surer-Thuswaldner:10} that shows a relation of this problem to the structure of $\D_p^{(0)}$ for $p < d$. The study of the ``non-real'' part $E_d^{(\mathbb{C})}$ is treated in the remaining parts of this section. Since we saw in Section~\ref{sec:D2} that the investigation of $E_2^{(\mathbb{C})}$ is difficult already in the case $d=2$, it is no surprise that no complete result exists in this direction. However, there are interesting relations to a conjecture of Bertrand~\cite{Bertrand:77} and Schmidt~\cite{Schmidt:80} on periodic orbits of beta-transformations w.r.t.\ Salem numbers and partial periodicity results starting with the work of Boyd~\cite{Boyd:89,Boyd:96,Boyd:97} that we want to survey.

\subsection{The case of real roots}\label{sec:realDd}
In Kirschenhofer {\it et al.}~\cite{Kirschenhofer-Pethoe-Surer-Thuswaldner:10}  the authors could establish strong relations between the sets $\D_d$ on the one side and $\D_e^{(0)}$ for $e<d$ as well as some related sets on the other side. This leads to a characterization of $\D_d$ in the regions $E_d^{(1)}$ and $E_d^{(-1)}$ (and even in some small subsets of $E_d^{(\mathbb{C})}$) in terms of these sets. The according result, which is stated below as Corollary~\ref{cor:real} follows from a more general theorem that will be established in this subsection.

Recall the definition of the operator $\odot$ in \eqref{eq:odot}. It was observed in \cite{Kirschenhofer-Pethoe-Surer-Thuswaldner:10} that the behavior of $\tau_{\bf r}, {\bf r} \in \RR^p$, can be described completely by the behavior of $\tau_{{\bf r}
\odot {\bf s}}$ if ${\bf s}\in\mathbb{Z}^q$. To be more specific,
 for  $q \in
\NN\setminus\{0\}$ and  ${\bf s}=(s_0,\ldots,s_{q-1}) \in
\ZZ^q$ let
\[
V_{\bf s}:\ZZ^{\infty} \rightarrow \ZZ^{\infty},\quad (x_n)_{n \in
\NN} \mapsto \left(\sum_{k=0}^{q-1} s_kx_{n+k}+x_{n+q}\right)_{n \in
\NN}.
\]
Then $V_{\bf s}$ maps each periodic sequence to a periodic sequence and each sequence that is eventually zero to a sequence that is eventually zero. Furthermore the following important fact holds ({\it cf.}~\cite{Kirschenhofer-Pethoe-Surer-Thuswaldner:10}).

\begin{proposition}
Let $p,q \geq 1$ be integers, ${\bf r} \in \RR^p$ and ${\bf s} \in
\ZZ^q$. Then
\[
V_{\bf s}\circ\tau_{{\bf r} \odot {\bf s}}^*(\ZZ^{p+q})=\tau_{\bf
r}^*(\ZZ^p).
\]\end{proposition}

Here we denote by $\tau_{\bf t}^*(\bf x)$ the integer sequence that is derived by concatenating successively the newly occurring entries of the iterates  $\tau_{\bf t}^n(\bf x)$ to the entries of the referring initial vector ${\bf x} =(x_0,\dots,x_{d-1})^t$.

\begin{proof}[Sketch of the proof {(\rm compare~\cite{Kirschenhofer-Pethoe-Surer-Thuswaldner:10})}.]
Let
\[
U=\left(\begin{array}{ccccccccc}
s_0 & s_1 & \cdots & s_{q-1} & 1 & 0 &\cdots &0 \\
0 & s_0 & & & \ddots & \ddots & \ddots & \vdots \\
\vdots & \ddots &\ddots & & & \ddots & \ddots & 0 \\
0 & \cdots  & 0 & s_0& \cdots & \cdots & s_{q-1} & 1
\end{array}\right) \in \ZZ^{p \times (p+q)}.
\]
Then $U$ has maximal rank $p$ and $U\ZZ^{p+q}=\ZZ^p$. The result will be proved if we show that for all ${\bf x} \in \ZZ^{p+q}$
\begin{equation}\label{for1}
V_{\bf s}\circ\tau_{{\bf r} \odot {\bf s}}^*({\bf x})=\tau_{\bf r}^*(U{\bf x}).
\end{equation}
holds. Supposing $(x_k)_{k \in \NN}=\tau_{{\bf r} \odot {\bf s}}^*({\bf x})$ and $(y_k)_{k \in \NN}=\tau_{\bf r}^*(U{\bf x})$ it has to be shown that
$$
y_n=s_0x_n+\cdots+s_{q-1}x_{n+q-1}+x_{n+q}
$$
holds for all $n \geq 0$. The latter fact  is now proved  by induction on $n$.
\end{proof}

\begin{example}[{see \cite{Kirschenhofer-Pethoe-Surer-Thuswaldner:10}}]
Let ${\bf r}=(\frac{11}{12},\frac{9}{5})$ and ${\bf s}=(1)$. The theorem says that the behavior of $\tau_{\bf r}$ is completely described by the behavior of $\tau_{{\bf r} \odot {\bf s}}$. For instance, suppose ${\bf y}:=(5,-3)^t$. We can choose ${\bf x}:=(4,1,-4)^t$ such that $U{\bf x}={\bf y}$ with
\[U=\left(\begin{array}{ccc}
1 & 1 & 0  \\
0 & 1 & 1
\end{array}\right).\]
Then ${\bf r} \odot {\bf s}=(\frac{11}{12},\frac{163}{60},\frac{14}{5})$ and
\[
\tau_{{\bf r} \odot {\bf s}}^*({\bf x})=4,1,-4,(5,-4,2,1,-4,7,-9,10,-9,7,-4,1,2,-4)^\infty.
\]
In our case the map $V_{\bf s}$ performs the addition of each two consecutive entries of a sequence. Therefore we find
\[
\tau_{{\bf r}}^*({\bf y})=V_{{\bf s}}\circ\tau_{{\bf r} \odot {\bf s}}^*({\bf x})=5,-3,(1,1,-2,3,-3,3,-2)^\infty.
\]
\end{example}

An important consequence of the last proposition is the following result.

\begin{corollary}[{{\em cf.}~\cite{Kirschenhofer-Pethoe-Surer-Thuswaldner:10}}]
Let ${\bf r} \in \RR^d$ and ${\bf s} \in \overline{\E_q} \cap
\ZZ^q$.
\begin{itemize}
\item  If ${\bf r} \odot {\bf s} \in \D_{d+q}$  then ${\bf r} \in
\D_{d}$.
\item  If ${\bf r} \odot {\bf s} \in \D_{d+q}^{(0)}$ then
${\bf r} \in \D_{d}^{(0)}$.
\end{itemize}
\end{corollary}
Unfortunately the converse of the corollary does not hold in general; for instance, we have (see Example~\ref{ex:KPT} below) $\left(1,\frac{1+\sqrt{5}}{2}\right) \in \D_2$, but
$\left(1,\frac{3+\sqrt{5}}{2},\frac{3+\sqrt{5}}{2}\right) =
(1)\odot\left(1,\frac{1+\sqrt{5}}{2}\right) \in \partial \E_3
\setminus \D_3$.
In the following we turn to results from \cite{Kirschenhofer-Pethoe-Surer-Thuswaldner:10} that allow to ``lift'' information
on some sets derived from $\D_{e}$ for $e<d$ to the boundary of
the sets $\D_{d}$.
We will need the following notations.
For ${\bf r} \in \D_d$ let
$\OR({\bf r})$ be the
set of all equivalence classes of cycles of $\tau_{\mathbf{r}}$.

For $p \in
\NN\setminus\{0\}$ and $\mathcal{B} \in
\OR({\bf r})$, define the function $S_p$ by
\[
 \mathcal{B}=\spk{x_0,
\ldots,x_{l(\mathcal{B})-1}} \mapsto
\begin{cases}
0 & \hbox{for } p \nmid l(\mathcal{B}) \mbox{ or } \sum_{j=0}^{l(\mathcal{B})-1} \xi_p^j x_j=0  \\
1 & \mbox{ otherwise},
\end{cases}
\]
where $\xi_p$ denotes a primitive $p$-th root of unity. Furthermore
let
$$
\D_d^{(p)}:=\{{\bf r} \in \D_d\; :\; \forall \mathcal{B} \in
\OR({\bf r}):S_p(\mathcal{B})=0\}.
$$

Observe that for $p=1$ we have that ${\bf r} \in \D_d$ lies in $\D_d^{(1)}$
iff the sum of the entries of each cycle of $\tau_{\mathbf r}$ equals 0. For $p=2$ the
alternating sums $x_0-x_1+x_2-x_3\pm \dots$ must vanish, and so on.

Let $\Phi_j$ denote the $j$th {\it cyclotomic polynomial}. Then
the following result could be proved in
\cite{Kirschenhofer-Pethoe-Surer-Thuswaldner:10}.

\begin{theorem}\label{iii}
Let $d,q \geq 1$, ${\bf r} \in \RR^d$ and ${\bf
s}=(s_0,\ldots,s_{q-1}) \in \ZZ^q$ such that $s_0 \not=0$. Then
${\bf r} \odot {\bf s} \in \D_{d+q}$ if and only if the following
conditions are satisfied:
\begin{itemize}
\item [(i)] $\chi_{\bf s}=\Phi_{\alpha_1}\Phi_{\alpha_2}\cdots\Phi_{\alpha_b}$ for
pairwise disjoint non-negative integers $\alpha_1,\ldots,\alpha_b$, and
\item [(ii)] ${\bf r} \in \bigcap_{j=1}^b \D_d^{(\alpha_j)}$.
\end{itemize}
\end{theorem}

\begin{proof}[Sketch of proof (compare~{\cite{Kirschenhofer-Pethoe-Surer-Thuswaldner:10}})]
In order to prove the sufficiency of the two conditions let us
assume that (i) and (ii) are satisfied. We have to show that
for arbitrary ${\bf x} \in
\ZZ^{d+q}$ the sequence $(x_n)_{n \in \NN}:=\tau_{{\bf r} \odot {\bf
s}}^*({\bf x})$ is ultimately
periodic. Let $s_q:=1$. Setting
\[
{\bf y}:=\left(\sum_{i=0}^q s_ix_i,\sum_{i=0}^q
s_ix_{i+1},\ldots,\sum_{i=0}^q s_ix_{i+d-1}\right)
\]
it follows from \eqref{for1} that
\begin{equation}\label{xy}
(y_n)_{n \in \NN}:=\tau_{\bf
r}^*({\bf y})=V_{{\bf s}}((x_n)_{n \in \NN}).
\end{equation}

Since ${\bf r} \in
\D_d^{(\alpha_1)},$ we have  ${\bf r} \in \D_d$.
Therefore there must exist a cycle
$\spk{y_{n_0},\ldots,y_{n_0+l-1}} \in \OR({\bf r})$,
from which we deduce  the recurrence
relation
\begin{equation}\label{zug1}
\sum_{h=0}^q  s_hx_{n_0+k+h} =y_{n_0+k}= y_{n_0+k+l}= \sum_{h=0}^q
s_hx_{n_0+k+l+h}
\end{equation}
for $k \geq 0$ for the sequence $(x_n)_{n \geq n_0}$.
Its characteristic equation is
\[
(t^l-1)\chi_{\bf s}(t)=0.
\]
Let us now assume that $\lambda_1,\ldots,\lambda_w$ are
the roots  of $t^l-1$ only, $\lambda_{w+1},\ldots,\lambda_l$ are the common roots
of $t^l-1$ and of $\chi_{\bf s}(t)$, and
$\lambda_{l+1},\ldots,\lambda_g$ are the roots of $\chi_{\bf s}(t)$ only. By Assertion~(i) $t^l-1$ and  $\chi_{\bf s}(t)$ have
only simple roots, so that
$\lambda_{w+1},\ldots,\lambda_l$ have multiplicity two
while all the other roots are simple. Therefore the solution of recurrence
\eqref{zug1} has the form
\begin{equation}\label{heu1}
x_{n_0+k}= \sum_{j=1}^g A^{(0)}_j \lambda_j^k + \sum_{j=w+1}^l
A^{(1)}_j k \lambda_j^k
\end{equation}
for $l+q$ complex constants $A_j^{(\nu)}$, and the (ultimate) periodicity of $(x_n)_{n
\geq n_0}$ is equivalent to $A_j^{(1)}=0$ for all $j \in
\{w+1,\ldots,l\}$.  For $\alpha_1\nmid l, \ldots, \alpha_b
\nmid l$ the polynomials $x^l-1$ and $\chi_{\bf s}$ have no
common roots and the result is immediate. Let us now suppose  $x^l-1$ and $\chi_{\bf s}$ have common roots, {\it
i.e.}, that $w<l$.

From
\eqref{zug1}, observing $\chi_{\bf s}(\lambda_j)=0$
for $j > w,$ we get with $k\in\{0,\ldots,l-1\}$ the following system of $l$
linear equalities for the $l$ constants
$A_1^{(0)},\ldots,A_w^{(0)},A_{w+1}^{(1)},\ldots,A_l^{(1)}$.
\begin{equation}\label{heu2}
\begin{array}{rl}
\displaystyle y_{n_0+k}& \displaystyle= \displaystyle \sum_{j=1}^g A^{(0)}_j \lambda_j^k\chi_{\bf s}(\lambda_j)  + \sum_{j=w+1}^l A^{(1)}_j
(k\lambda_j^k\chi_{\bf s}(\lambda_j)+\lambda_j^{k+1}\chi'_{\bf s}(\lambda_j))\\
&=\displaystyle \sum_{j=1}^w A^{(0)}_j \lambda_j^k\chi_{\bf s}(\lambda_j)  + \sum_{j=w+1}^l A^{(1)}_j
\lambda_j^{k+1}\chi'_{\bf s}(\lambda_j).
\end{array}
\end{equation}
It remains to show that $A_j^{(1)}=0$ for
$w+1\le j \le l$.

Rewriting the system as
\begin{equation}\label{EQ}
(y_{n_0},\ldots,y_{n_0+l-1})^t= G
(A_1^{(0)},\ldots,A_w^{(0)},A_{w+1}^{(1)},\ldots,A_l^{(1)})^t
\end{equation}
with
\[
G = \left(\begin{array}{cccccc}
\chi_{\bf s}(\lambda_1) & \cdots & \chi_{\bf s}(\lambda_w) & \chi'_{\bf s}(\lambda_{w+1})\lambda_{w+1} & \cdots & \chi'_{\bf s}(\lambda_{l})\lambda_{l} \\
\chi_{\bf s}(\lambda_1)\lambda_1 & \cdots & \chi_{\bf s}(\lambda_w)\lambda_w & \chi'_{\bf s}(\lambda_{w+1})\lambda_{w+1}^2 & \cdots & \chi'_{\bf s}(\lambda_{l})\lambda_{l}^2 \\
\vdots & &\vdots & \vdots & &\vdots \\
\chi_{\bf s}(\lambda_1)\lambda_1^{l-1} & \cdots & \chi_{\bf s}(\lambda_w)\lambda_w^{l-1} & \chi'_{\bf s}(\lambda_{w+1})\lambda_{w+1}^l & \cdots & \chi'_{\bf s}(\lambda_{l})\lambda_{l}^l
\end{array}\right).
\]
we have by Cramer's rule that
\[
A_j^{(1)} = \frac{\det G_j}{\det G},
\]
where $G_j$ denotes the matrix that is obtained by exchanging the $j$th column of $G$ by
the vector $(y_{n_0},\ldots,y_{n_0+l-1})^t$.
 ($\det G \not=0$ is easily detected using the  Vandermonde determinant.)

Now
\begin{align}\label{xxxx}
\det G_j = &\prod_{k=1}^w \chi_{\bf s}(\lambda_k)^l
\prod_{\begin{subarray}{c} k=w+1 \\ k \not=j \end{subarray}}^l
(\lambda_k\chi'_{\bf s}(\lambda_k))^l D_j,
\end{align}
where
\begin{align*}D_j := & \det\left(\begin{array}{ccccccc}
1 & \cdots & 1 & y_{n_0} & 1 & \cdots & 1 \\
\lambda_1 & \cdots & \lambda_{j-1} & y_{n_0+1} & \lambda_{j+1} & \cdots & \lambda_{l} \\
\vdots & &\vdots & \vdots &&\vdots \\
\lambda_1^{l-1} & \cdots & \lambda_{j-1}^{l-1} & y_{n_0+l-1} & \lambda_{j+1}^{l-1} & \cdots & \lambda_{l}^{l-1} \\
\end{array}\right).
\end{align*}

Adding the $\overline{\lambda_j}^{k+1}$-fold multiple of the
$k$th row to the last row for each $k\in\{1,\ldots,l-1\}$ we gain
\[
D_j=\det\left(\begin{array}{ccccccc}
1 & \cdots & 1 & y_{n_0} & 1 & \cdots & 1 \\
\lambda_1 & \cdots & \lambda_{j-1} & y_{n_0+1} & \lambda_{j+1} & \cdots & \lambda_{l} \\
\vdots & &\vdots & \vdots &&\vdots \\
\lambda_1^{l-2} & \cdots & \lambda_{j-1}^{l-2} &  y_{n_0+k+l-2} & \lambda_{j+1}^{l-2} & \cdots & \lambda_{l}^{l-2} \\
0 & \cdots & 0 & \sum_{k=0}^{l-1} \overline{\lambda_j}^{k+1} y_{n_0+k} & 0 & \cdots & 0 \\
\end{array}\right).
\]
If we can establish  $\sum_{k=0}^{l-1} \overline{\lambda_j}^k
y_{n_0+k} = 0$, we are done.  Now, since $\lambda_j$ is a
root of $x^l-1$ and $\chi_{\bf s}$,  Condition (i) yields that
there exists a $p \in \{1,\ldots,b\}$ with $\alpha_p \mid l$.
Thus $\lambda_j$, and $\overline{\lambda_j},$  are
primitive $\alpha_p$th roots of unity. It follows from
Condition (ii) that
\begin{equation}\label{ES}
S_{\alpha_p}(\spk{y_{n_0},\ldots,y_{n_0+l-1}})=\sum_{k=0}^{l-1}
\xi_{\alpha_p}^k y_{n_0+k}=0
\end{equation}
for each $\alpha_p$th root of unity $\xi_{\alpha_p}$. In particular, $\sum_{k=0}^{l-1}
\overline{\lambda_j}^k y_{n_0+k}=0$.

Let us turn to the necessity of Conditions (i) and (ii).

Since ${\bf r} \odot {\bf s} \in \D_{d+q}$ we have $\varrho(R({\bf s})) \le \varrho(R({\bf r} \odot
{\bf s}))\le 1$. Since $s_0\not=0$, $\chi_{s}$
is a polynomial over $\ZZ$ each of whose roots are
non-zero and bounded by one in modulus. This implies that each root
of this polynomial is a root of unity.

Suppose now that
$\chi_{\bf s}$ has a root of multiplicity at least $2$, say
$\lambda_{j_0}$. Let ${\bf x} \in \ZZ^{d+q}$ and $(x_n)_{n \in
\NN}:=\tau_{{\bf r} \odot {\bf s}}^*({\bf x})$.  Since $(x_n)_{n \in
\NN}$ is a solution of recurrence \eqref{zug1} it must have the
shape

\[
x_{n_0+k} = \sum_{j=1}^gA_j(k)\lambda_j^k
\]

with some polynomials $A_j$ ($1\le j \le g$).
Inserting \eqref{zug1} yields
\begin{equation}\label{Y2}
y_{n_0+k}= \sum_{j=1}^g \sum_{h=0}^q s_h A_j(k+h) \lambda_j^{k+h}.
\end{equation}
Taking $k\in \{1,\ldots, l\}$ we get a system of $l$ equations for the $l+q$
coefficients $A_j^{(\nu)}$ of the polynomials $A_j$. In a similar way as in
the treatment of \eqref{zug1} in the first part of this proof it can be shown
that
  $q$
of the $l+q$ coefficients do not occur in \eqref{Y2} for $k\in \{1,\ldots, l\}$, and the system can be used  to calculate the remaining $l$  coefficients $A_j^{(\nu)}$. The coefficient of  $A_{j_0}^{(1)}$
 in \eqref{Y2} equals
\[
k\lambda_{j_0}^k\chi_{\bf s}(\lambda_{j_0})+\lambda_{j_0}^{k+1}\chi'_{\bf s}(\lambda_{j_0}).
\]
Since $\lambda_{j_0}$ is a double zero of $\chi_{\bf s}$, the latter
expression, and thus the coefficient of $A_{j_0}^{(1)}$
 in \eqref{Y2}  vanishes.

 Let now $z_1,\ldots, z_q$ be a $q$-tuple of integers
and consider the system of $q$ equations
\begin{equation}\label{zk}
z_k = \sum_{j=1}^gA_j(k)\lambda_j^k \qquad(1\le k\le q).
\end{equation}
This system can be used in order to calculate the
remaining $q$ coefficients $A_j^{(\nu)}$, among which we have
$A_{j_0}^{(1)}$. Choosing $z_1,\ldots, z_q$ in a way that
$A_{j_0}^{(1)}\not=0$ allows to determine all
coefficients $A_j^{(\nu)}$.
We use equation \eqref{zk} now to define the integers $z_k$ for  $k > q$.
Then by \eqref{zug1} the sequence $\tau_{{\bf r} \odot {\bf
s}}^*((z_0,\ldots,z_{d+l-1}))$ satisfies the recurrence
relation $\sum_{i=0}^q  s_iz_{n_0+k+i} = \sum_{i=0}^q
s_iz_{n_0+k+l+i}$. As $A_{j_0}^{(1)} \not=0$ this sequence does
not end up periodically by the following auxiliary Lemma~\ref{aux}, a contradiction
to ${\bf r} \odot {\bf s} \in \D_{d+q}$. Thus we have proved the
necessity of Condition (i) of the theorem.

Let us now turn to Condition (ii).
 Since ${\bf r} \odot {\bf s} \in \D_{d+q}$,
$(x_n)_{n\in\NN}$ in \eqref{heu1} is ultimately periodic. By
Lemma~\ref{aux} this implies that $A_j^{(1)}=0$ for each
$j\in\{w+1,\ldots,l\}$. Adopting the notation and reasoning
 of the sufficiency part of
 the proof this is equivalent to $D_j=0,$ so that \eqref{ES}
 holds for $\alpha_1,\ldots,\alpha_b,$ proving the necessity of (ii).
\end{proof}

In the last proof we make use of the following auxiliary result, which,
in other terminology, can be found in \cite{Kirschenhofer-Pethoe-Surer-Thuswaldner:10}, too.
\begin{lemma}\label{aux}
Let the sequence $(x_n)_{n\ge 0}$ be the solution of a homogeneous linear
 recurrence with constant coefficients in $\CC,$ whose
 eigenvalues are pairwise disjoint roots of unity. If at least one
 of the eigenvalues has multiplicity greater than 1, then $x_n$ is not
bounded.
\end{lemma}

For the simple proof we also refer to~\cite{Kirschenhofer-Pethoe-Surer-Thuswaldner:10}.


Theorem~\ref{iii} allows in particular to give information on the
behavior of $\tau_{\bf r}$ on several parts of the boundary
of $\E_d$. Remember that $\partial \E_3 =
\partial \D_3$ consists of the two triangles $E_3^{(-1)}$ and
$E_3^{(1)}$ and of the surface $E_3^{(\mathbb{C})}$. Then we have the following result (see~\cite{Kirschenhofer-Pethoe-Surer-Thuswaldner:10}).

\begin{corollary}\label{cor:real}
The following assertions hold.
\begin{itemize}
\item
For $d \geq 2$ we have $\D_d \cap E_d^{(-1)} = (-1) \odot
\D_{d-1}^{(1)}$.
\item
For $d \geq 2$ we have $\D_d \cap E_d^{(1)} = (1) \odot
\D_{d-1}^{(2)}$.
\item
For $d \geq 3$ we have $\D_d \cap (1,0) \odot \overline{\E_{d-2}} =
(1,0) \odot \D_{d-2}^{(4)}$.
\item
For $d \geq 3$ we have $\D_d \cap (1,1) \odot \overline{\E_{d-2}} =
(1,1) \odot \D_{d-2}^{(3)}$.
\item
For $d \geq 3$ we have $\D_d \cap (1,-1) \odot \overline{\E_{d-2}} =
(1,-1) \odot \D_{d-2}^{(6)}$.

\end{itemize}
\end{corollary}

Combining the first two items of the last corollary and computing
an approximation of $\D_{2}^{(1)}$ and $\D_{2}^{(-1)}$ yields
an approximation of $\D_3 \cap E_3^{(-1)}$ resp. $\D_3 \cap E_3^{(-1)}$
as depicted in Figure~\ref{E3-1} (algorithms for $\D_d^{(0)}$ are presented in Section~\ref{sec:algorithms}; they can be adapted to $\D_d^{(p)}$ in an obvious way).
\begin{figure}
\hskip 0.7cm \includegraphics[width=0.3\textwidth, bb=0 0 300 400]{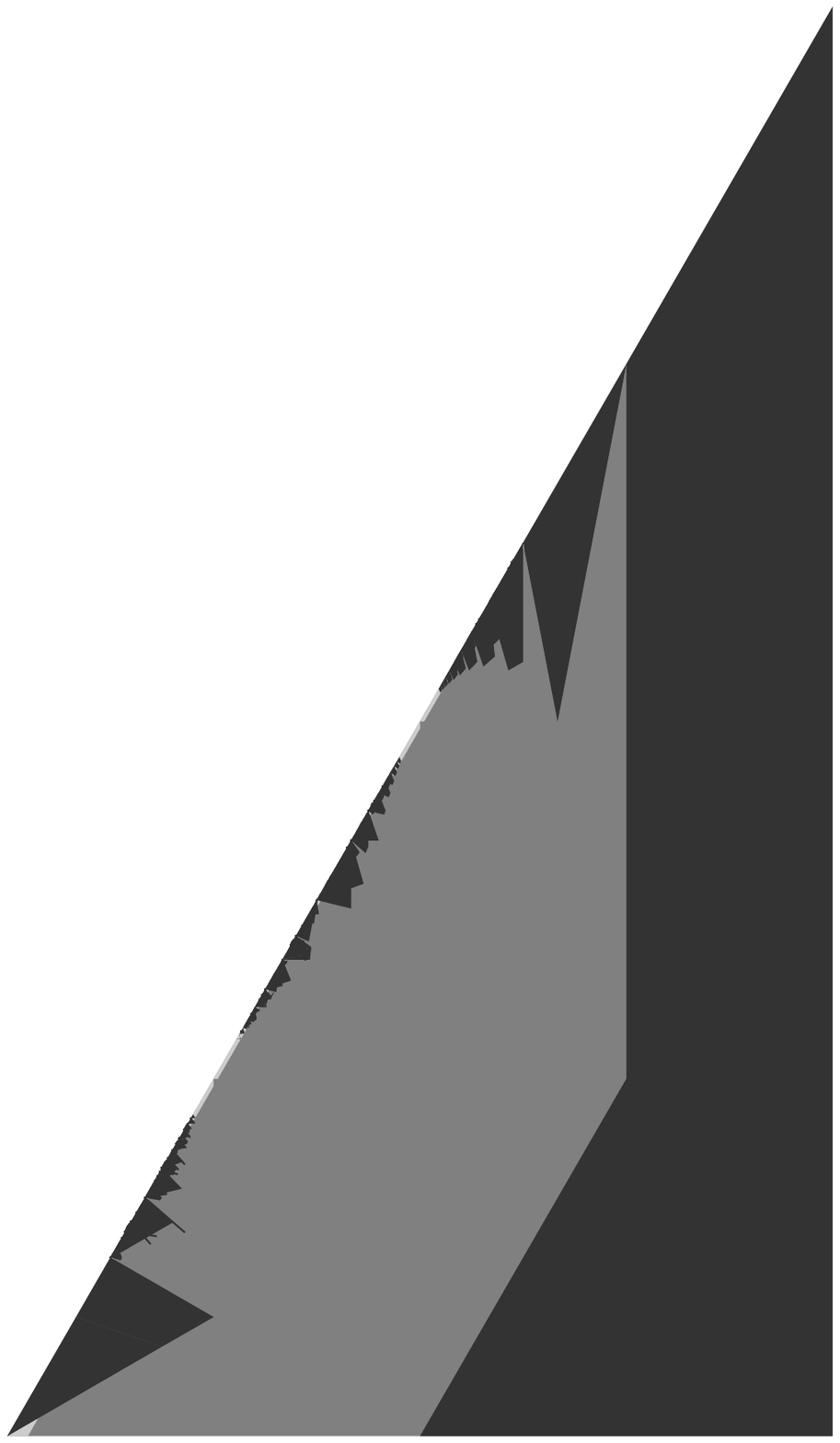}
\hskip 1cm
\includegraphics[width=0.37\textwidth, bb=0 0 300 400]{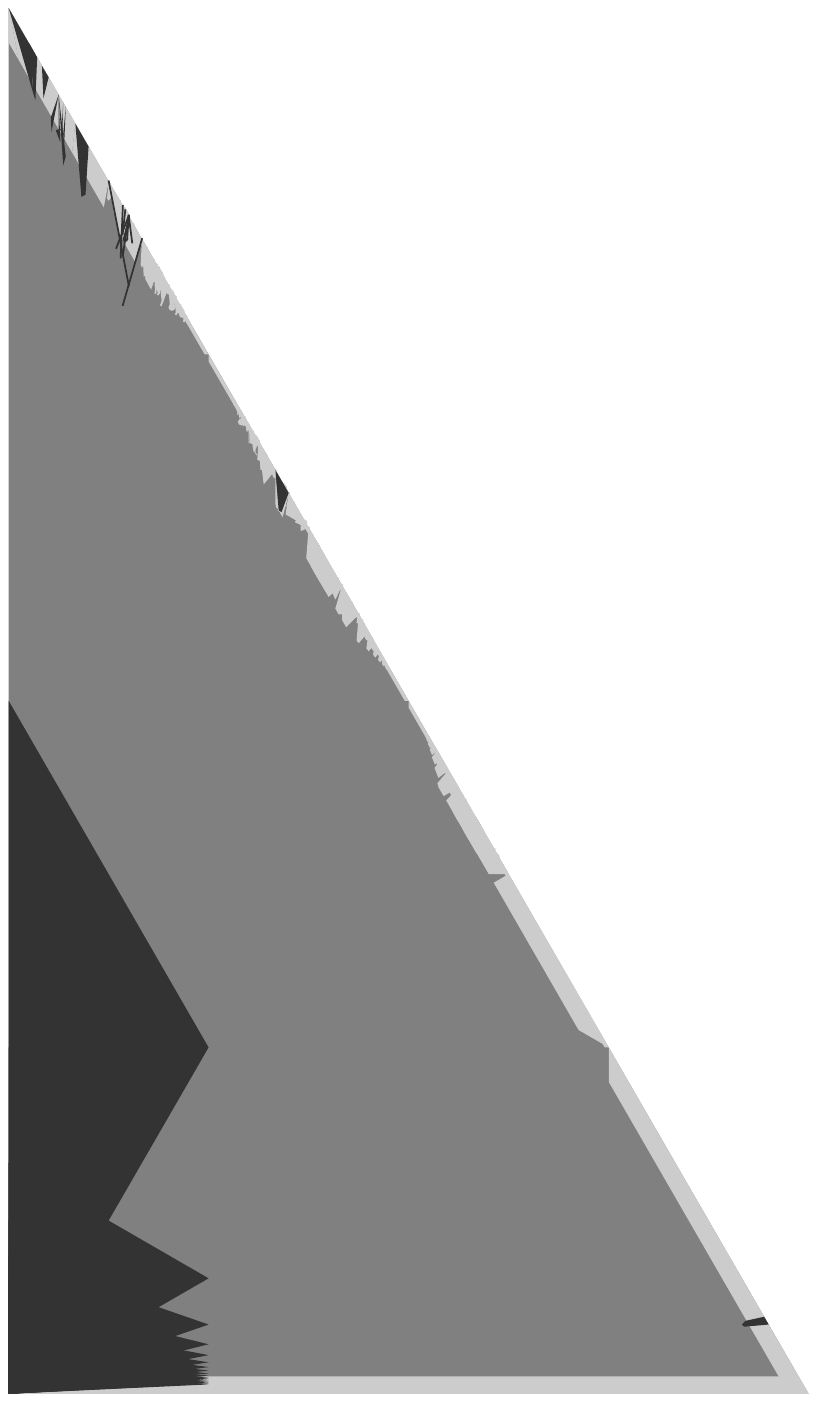}
  \caption{The triangles $E_3^{(-1)}$ (left hand side) and $E_3^{(1)}$
(right hand side). Dark grey:  parameters $\mathbf{r}$ for which
$\tau_\mathbf{r}$ is ultimately periodic for each starting value
$\mathbf{z}\in\mathbb{Z}^3$. Black: there exists a starting value
$\mathbf{z}\in\mathbb{Z}^3$ such that the orbit
$(\tau_\mathbf{r}^k(\mathbf{z}))_{k\in\mathbb{N}}$ becomes
unbounded. Light grey: not yet characterized. See \cite[Figure~3]{Kirschenhofer-Pethoe-Surer-Thuswaldner:10}.}
\label{E3-1}
\end{figure}
The last three items of the corollary allow to characterize {\it e.g.}
some lines of the surface $\D_3 \cap E_3^{(\mathbb{C})}$.

The study of concrete parameters $\mathbf{r}\in E_d^{(1)} \cup  E_d^{(-1)}$ shows interesting behavior as illustrated in the following three dimensional example.

\begin{example}[{\em cf.}~\cite{Kirschenhofer-Pethoe-Thuswaldner:08}]\label{ex:KPT}
Let $\varphi = \frac{1+\sqrt{5}}{2}$. From our considerations above one can derive that
 $$
\left(1,\varphi^2,\varphi^2\right)\in\partial \D_3 \setminus \D_3.
$$
We want to study the orbits of $\tau_{\left(1,\varphi^2,\varphi^2\right)}$ more closely for certain starting values $(z_0,z_1,z_2)$. In particular, let $z_0=z_1=0$. Interestingly, the behavior of the orbit depends on the starting digits of the {\it Zeckendorf representation} of $z_2$:
$$
z_2 = \sum_{j\ge 2} z_{2,j} F_j
$$
such that $z_{2,j}\in \{0,1\}, z_{2,j}z_{2,j+1}=0, j\ge 2$ (as in the proof of Theorem~\ref{th:steinerproof}, $(F_n)_{n\in\mathbb{N}}$ is the sequence of Fibonacci numbers). More precisely the following results hold (see~\cite[Theorems~4.1,5.1,5.2, and~5.3]{Kirschenhofer-Pethoe-Thuswaldner:08}).
\begin{itemize}
\item If $z_{2,2}=z_{2,3}=0$ the sequence $(z_n)$ is divergent;
\item if $z_{2,2}=0, z_{2,3}=1$ the sequence $(z_n)$ has period 30;
\item if $z_{2,2}=1, z_{2,3}=z_{2,4}=0$ the sequence $(z_n)$ has period 30;
\item if $z_{2,2}=1, z_{2,3}=0, z_{2,4}=1$ the sequence $(z_n)$ has period 70.
\end{itemize}
\end{example}

\subsection{The conjecture of Klaus Schmidt on Salem numbers}\label{sec:Salem}

Schmidt~\cite{Schmidt:80} (see also  Bertrand \cite{Bertrand:77}) proved the following result on beta-expansions of Salem numbers (recall that $T_\beta$ is the beta-transformation defined in \eqref{betatransform}).

\begin{theorem}[{\cite[Theorems~2.5 and 3.1]{Schmidt:80}}]\label{thm:Schmidt}
Let $\beta > 1$ be given.
\begin{itemize}
\item If $T_\beta$ has an ultimately periodic orbit for each element of $\mathbb{Q}\cap[0,1)$, then $\beta$ is either a Pisot or a Salem number.

\item If $\beta$ is a Pisot number, then $T_\beta$ has an ultimately periodic orbit for each element of $\mathbb{Q}(\beta) \cap [0,1)$.
\end{itemize}
\end{theorem}

We do not reproduce the proof of this result here, however, if we replace the occurrences of $\mathbb{Q}$ as well as $\mathbb{Q}(\beta)$ in the theorem by $\mathbb{Z}[\beta]$ and assume that $\beta$ is an algebraic integer then the according slightly modified result follows immediately from Propositions~\ref{prop:betanumformula} and \ref{EdDdEd}. Just observe that the conjugacy between $T_\beta$ and $\tau_\mathbf{r}$ stated in Proposition~\ref{prop:betanumformula} relates Pisot numbers to SRS parameters $\mathbf{r}\in \E_d$ and Salem numers to SRS parameters $\mathbf{r}\in \partial\E_d$. Therefore, this modification of Theorem~\ref{thm:Schmidt} is a special case of Proposition~\ref{EdDdEd}.

Note that Theorem~\ref{thm:Schmidt} does not give information on whether beta-expansions w.r.t.\ Salem numbers are periodic or not. Already Schmidt~\cite{Schmidt:80} formulated the following conjecture.

\begin{conjecture}\label{con:Schmidt}
If $\beta$ is a {\em Salem} number, then $T_\beta$ has an ultimately periodic orbit for each element of $\mathbb{Q}(\beta) \cap [0,1)$.
\end{conjecture}

So far, no example of a non-periodic beta-expansion w.r.t.\ a Salem number $\beta$ has been found although Boyd~\cite{Boyd:96} gives a heuristic argument that puts some doubt on this conjecture (see Section~\ref{sec:Heuristic}). In view of Proposition~\ref{prop:betanumformula} (apart from the difference between $\mathbb{Z}[\beta]$ and $\mathbb{Q}(\beta)$) this conjecture is a special case of the following generalization of Conjecture~\ref{Vivaldi-SRS-Conjecture} to arbitrary dimensions (which, because of Boyd's heuristics, we formulate as a question).

\begin{question}\label{qu:Schmidt}
Let $\mathbf{r}\in E_d^{(\mathbb{C})}$ be given. Is it true that each orbit of $\tau_\mathbf{r}$ is ultimately periodic?
\end{question}

As Proposition~\ref{D2easyboundary} and Corollary~\ref{cor:real} show, the corresponding question cannot be answered affirmatively for all parameters contained in $E_d^{(1)}$ as well as in $E_d^{(-1)}$.

\subsection{The expansion of $1$}

Since it seems to be very difficult to verify Conjecture~\ref{con:Schmidt} for a single Salem number $\beta$, Boyd~\cite{Boyd:89} considered the simpler problem of studying the orbits of $1$ under $T_\beta$ for Salem numbers of degree $4$.
In \cite[Theorem~1]{Boyd:89} he shows that these orbits are always ultimately periodic and -- although there is no uniform bound for the period -- he is able to give the orbits explicitly. We just state the result about the periods and omit the description of the concrete structure of the orbits.

\begin{theorem}[{see \cite[Lemma~1 and Theorem~1]{Boyd:89}}]\label{thm:boyd4}
Let $X^4 + b_1 X^3 + b_2 X^2 + b_1 X + 1$ be the minimal polynomial of a Salem number of degree $4$. Then $\lfloor \beta\rfloor \in \{-b_1-2,-b_1-1,-b_1,-b_1+1  \}$. According to these values we have the following periods $p$ for the orbits of $T_\beta(1)$:
\begin{itemize}
\item[(i)] If $\lfloor \beta \rfloor = -b_1+1$ then $2b_1-1 \le b_2 \le b_1-1$ and
\begin{itemize}
\item[(a)] if $b_2 = 2b_1 -1$ then $p=9$, and
\item[(b)] if $b_2 > 2b_1 -1$ then $p=5$.
\end{itemize}
\item[(ii)] If $\lfloor \beta \rfloor = -b_1$ then $p=3$.
\item[(iii)] If $\lfloor \beta \rfloor = -b_1-1$ then $p=4$.
\item[(iv)] If $\lfloor \beta \rfloor = -b_1-2$ then $-b_1+1 <  b_2 \le -2b_1-3$. Let $c_k=(-2b_1-2)-(-b_1-3)/k$ for $k\in \{1,2,\ldots, -b_1 - 3\}$. Then $-b_1+1 =c_1 < c_2 < \cdots < c_{-b_1-3} = -2 b_1-3$ and $c_{k-1} < b_2 \le c_k$ implies that $p=2k+2$.
\end{itemize}
\end{theorem}

According to Proposition~\ref{prop:betanumformula} the dynamical systems $(\tau_\mathbf{r}, \mathbb{Z}^d)$ and $(T_\beta, \mathbb{Z}[\beta] \cap [0,1))$ are conjugate by the conjugacy $\Phi_{\mathbf r}(\mathbf{z})=\{\mathbf{r}\mathbf{z}\}$ when $\mathbf{r}=(r_0,\ldots,r_{d-1})$ is chosen as in this proposition. Thus {\it a priori} $T_\beta(1)$ has no analogue in $(\tau_\mathbf{r}, \mathbb{Z}^d)$. However, note that
\begin{align*}
\tau_\mathbf{r}((1,0,\ldots, 0)^t) &= (0,\ldots,0, -\lfloor r_0 \rfloor)^t = (0,\ldots,0, -\lfloor -1/\beta \rfloor)^t  = (0,\ldots,0,1)^t \quad \hbox{and}\\
T_\beta(1)&= \{ \beta \},
\end{align*}
since, as a Salem number is a unit we have $b_0=1$ and, hence, $r_0=-1/\beta \in(-1,0)$. Because $\Phi_{\mathbf{r}}((0,\ldots,0,1)^t)=\{\beta\}$ we see that the orbit of $(1,0,\ldots,0)^t$ under $\tau_{\mathbf{r}}$ has the same behavior as the orbit of $1$ under $T_\beta$.

Let us turn back to Salem numbers of degree $4$. If $\beta$ is such a Salem number then, since $\beta$ has non-real conjugates on the unit circle, the minimal polynomial of $\beta$ can be written as $(X-\beta)(X^3 +r_2X^2 + r_1X +r_0)$
with $\mathbf{r}=(r_0,r_1,r_2) \in E_3^{(\mathbb{C})}$. Thus Theorem~\ref{thm:boyd4} answers the following question for a special class of parameters.

\begin{question}\label{qu:salem4}
Given $\mathbf{r} \in E_3^{(\mathbb{C})}$, is the orbit of $(1,0,0)^t$ under $\tau_{\mathbf r}$ ultimately periodic and, if so, how long is its period?
\end{question}

As mentioned in Section~\ref{sec:Ed}, the set $ E_3^{(\mathbb{C})}$ is a surface in $\mathbb{R}^3$. Using the definition of $E_3^{(\mathbb{C})}$ one easily derives that (see equation~\eqref{eq:EC} and \cite{Kirschenhofer-Pethoe-Surer-Thuswaldner:10})
\begin{equation}\label{ECn}
 E_3^{(\mathbb{C})}=\{(t,st+1,s+t) \;:\; -2< s< 2 ,\, -1\le t\le
1\}.
\end{equation}
Figure~\ref{fig:salemparameters} illustrates which values of the parameters $(s,t)$ correspond to Salem numbers.
\begin{figure}
\includegraphics[height=6cm]{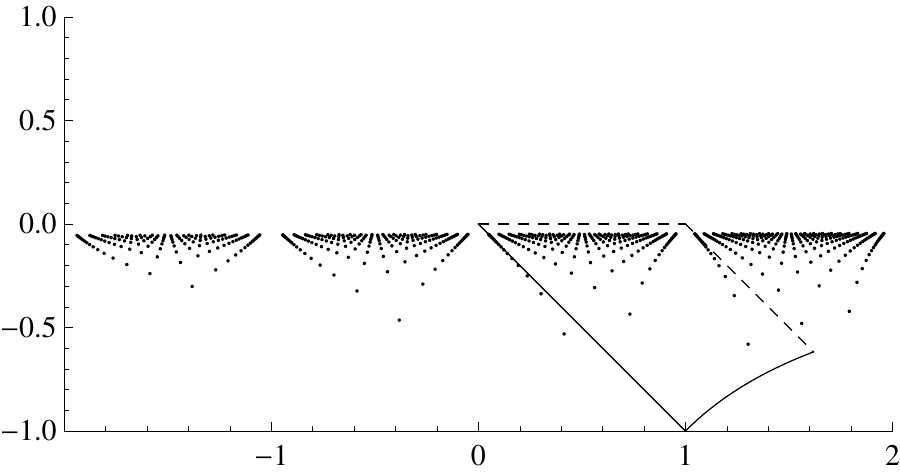}
\caption{The black dots mark the parameters corresponding to Salem numbers of degree $4$ in the parameterization of $E_3^{(\mathbb{C})}$ given in \eqref{ECn}. The marked region indicates a set of parameters that share the same orbit of $(1,0,0)^t$.
 \label{fig:salemparameters}}
\end{figure}
By Theorem~\ref{thm:boyd4} and the above mentioned remark on the conjugacy of the dynamical systems $(\tau_\mathbf{r}, \mathbb{Z}^d)$ and $(T_\beta, \mathbb{Z}[\beta] \cap [0,1))$, for each of the indicated points we know that the orbit of $\tau_{(t,st+1,s+t) }$ is periodic with the period given in this theorem. How about the general answer to Question~\ref{qu:salem4}? In Figure~\ref{fig:schmidtshaded} we illustrate the periods of the orbit of $(1,0,0)^t$ for the values $(s,t) \in (-2,2)\times [0,1]$.
\begin{figure}
\includegraphics[width=0.8\textwidth]{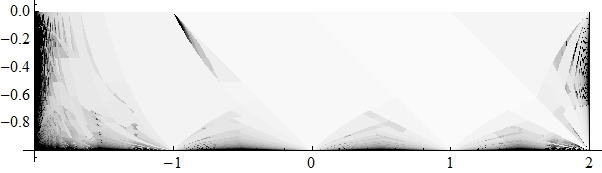}
\caption{Lengths of the orbit of $(1,0,0)^t$ in the parameter region $(s,t)\in (-2,2)\times[0,1]$. The lighter the point, the shorter the orbit.  Comparing this with Figure~\ref{fig:salemparameters} we see that for ``most'' Salem numbers of degree $4$ the orbit of $1$ under $T_\beta$ has short period. This agrees with Theorem~\ref{thm:boyd4}.
 \label{fig:schmidtshaded}}
\end{figure}
Although it is not hard to characterize the period of large subregions of this parameter range, we do not know whether $(1,0,0)^t$ has an ultimately periodic orbit for each parameter. In fact, it looks like a ``fortunate coincidence'' that Salem parameters lie in regions that mostly correspond to small periods. In particular, we have no explanation for the black ``stain'' southeast to the point $(-1,0)$ in Figure~\ref{fig:schmidtshaded} that corresponds to a spot with very long periods.

Boyd~\cite{Boyd:96} studies orbits of $1$ under $T_\beta$ for Salem numbers of degree $6$. There seems to be no simple ``formula'' for the period as in the case of degree $4$. Moreover, for some examples no periods have been found so far (see also~\cite{Hare-Tweedle:08} where orbits of $1$ under $T_\beta$ are given for classes of Salem numbers). We give two examples that illustrate the difficulty of the situation.

\begin{example}
Let $\beta >1$ be the Salem number defined by the polynomial
\[
x^6 - 3x^5 - x^4 - 7x^3 -x^2 - 3x + 1.
\]
Let $m$ be the pre-period of the orbit of $1$ under $T_\beta$, and $p$ its period (if these values exist). Boyd~\cite{Boyd:96} showed with computer assistance that $m+p > 10^9$.
Hare and Tweedle~\cite{Hare-Tweedle:08} consider the Salem number $\beta > 1$ defined by
\[
x^{12}-3 x^{11}+3 x^{10}-4 x^9+5 x^8-5 x^7+5 x^6-5 x^5+5 x^4-4 x^3+3 x^2-3 x+1.
\]
They compute that, if it exists, the period of the orbit of $1$ under $T_\beta$ is greater than $5\cdot 10^5$ in this case. We emphasize that for both of these examples it is {\em not} known whether the orbit of $1$ under $T_\beta$ is ultimately periodic or not.
\end{example}

\subsection{The heuristic model of Boyd for shift radix systems}\label{sec:Heuristic}

Let $\beta$ be a Salem number. In \cite[Section~6]{Boyd:96} a heuristic probabilistic model for the orbits of $1$ under $T_\beta$ is presented. This model suggests that for Salem numbers of degree $4$ and $6$ ``almost all'' orbits should be finite, and predicts the existence of ``many'' unbounded orbits for Salem numbers of degree $8$ and higher.
This suggests that there exist counter examples to Conjecture~\ref{con:Schmidt}. Here we present an SRS version of Boyd's model in order to give heuristics for the behavior of the orbit of $(1,0,\ldots,0)^t$ under $\tau_\mathbf{r}$ for $\mathbf{r} \in E_d^{(\mathbb{C})}$.

Let $\mathbf{r}\in E_d^{(\mathbb{C})}$ be given. To keep things simple we assume that the characteristic polynomial $\chi_\mathbf{r}$ of the matrix $R(\mathbf{r})$ defined in \eqref{mata} is irreducible. Let $\beta_1,\ldots, \beta_d$ be the roots of $\chi_\mathbf{r}$ grouped in a way that $\beta_1,\ldots, \beta_r$ are real and $\beta_{r+j}= \bar \beta_{r+s+j}$ ($1\le j \le s$) are the non-real roots ($d=r+2s$). Let $D={\rm diag}(\beta_1,\ldots,\beta_d)$. Since $R(\mathbf{r})$ is a {\it companion matrix} we have $R(\mathbf{r}) = VDV^{-1}$, where $V=(v_{ij})$ with $v_{ij}=\beta_j^{i-1}$ is the {\it Vandermonde matrix} formed with the roots of $\chi_\mathbf{r}$ ({\it cf. e.g.}~\cite{Brand:64}). Iterating \eqref{linear} for $k$ times we get
\begin{align}
\tau_\mathbf{r}^k((1,0,\ldots,0)^t) &=
\sum_{j=0}^{k-1} R(\mathbf{r})^j \mathbf{d}_j + R(\mathbf{r})^k(1,0,\ldots, 0)^t\nonumber \\
&=\sum_{j=0}^{k-1} VD^jV^{-1} \mathbf{d}_j + VD^kV^{-1}(1,0,\ldots, 0)^t, \label{eq:salemsum}
\end{align}
where $\mathbf{d}_j=(0,\ldots,0,\varepsilon_j)^t$ with $\varepsilon_j \in [0,1)$. Let $V^{-1}=(w_{ij})$. From \cite[Section~3]{Soto-Moya:11} we easily compute that $w_{id}=\prod_{\ell \not=i}(\beta_i - \beta_\ell)^{-1}$. Multiplying \eqref{eq:salemsum} by $V^{-1}$, using this fact we arrive at
\begin{equation}\label{eq:SalemBox}
V^{-1} \tau_\mathbf{r}^k\begin{pmatrix} 1\\0\\ \vdots\\0 \end{pmatrix} = \begin{pmatrix}
\prod_{\ell\not=1}(\beta_1-\beta_\ell)^{-1}\sum_{j=0}^{k-1} \varepsilon_j\beta_1^j \\
\vdots\\
\prod_{\ell\not=d}(\beta_d-\beta_\ell)^{-1}\sum_{j=0}^{k-1} \varepsilon_j\beta_d^j
\end{pmatrix}+ D^kV^{-1}\begin{pmatrix}1\\ 0\\ \vdots \\ 0\end{pmatrix} \in \mathbb{R}^r\times \mathbb{C}^{2s}.
\end{equation}
Note that the $(r+j)$-th coordinate of \eqref{eq:SalemBox} is just the complex conjugate of its $(r+s+j)$-th coordinate ($1\le j \le s$). Thus two points in the orbit of $(1,0,\ldots,0)^t$ under $\tau_\mathbf{r}$ are equal if and only if the first $r+s$ coordinates under the image of $V^{-1}$ are equal. So, using the fact that $|\varepsilon_j|< 1$ and picking $\mathbf{z}=(z_1,\ldots, z_d)^t \in \{V^{-1} \tau_\mathbf{r}^k((1,0,\ldots,0)^t)\;:\; 0\le k < n\}$ implies that
\begin{itemize}
\item[(i)] $\mathbf{z}$ is an element of the lattice $V^{-1} \mathbb{Z}^d$.
\item[(ii)] If $i\in\{1,\ldots, r\}$ then $z_i\in \mathbb{R}$ with
\[
|z_i| \le  \prod_{\ell\not=i} |\beta_i-\beta_\ell|^{-1}\sum_{j=0}^{k-1} \left|\beta_i \right|^j.
\]
\item[(iii)] If $i\in\{1,\ldots, s\}$ then $z_{r+i}=\bar z_{r+s+i}\in \mathbb{C}$ with
\begin{align*}
|z_{r+i}| &\le \prod_{\ell\not=r+i}|\beta_{r+i}-\beta_\ell|^{-1}\sum_{j=0}^{k-1}  \left|\beta_{r+i} \right|^j\\
&= \sqrt{\prod_{\ell\not=r+i}|\beta_{r+i}-\beta_\ell|^{-1}\prod_{\ell\not=r+s+i}|\beta_{r+s+i}-\beta_\ell|^{-1}\sum_{j=0}^{k-1} |\beta_{r+i}|^j\sum_{j=0}^{k-1} |\beta_{r+s+i}|^j}.
\end{align*}
\end{itemize}
Let ${\rm disc}(\chi_\mathbf{r})=\prod_{i\not=j}(\beta_i-\beta_j)$ be the discriminant of $\chi_\mathbf{r}$. Then the three items above imply that a point in $\{V^{-1} \tau_\mathbf{r}^k((1,0,\ldots,0)^t)\;:\; 0\le k < n\}$ is a point of the lattice $V^{-1}\mathbb{Z}^d$ that is contained in a product $K_n$ of disks and intervals with volume
\[
{\rm Vol}(K_n)=\frac{c}{|{\rm disc}(\chi_\mathbf{r})|} \prod_{i=1}^d \sum_{j=0}^{k-1} \left|\beta_{i}^j\right|,
\]
where $c$ is an absolute constant. As $\det(V)=\sqrt{{\rm disc}(\chi_\mathbf{r})}$ is the mesh volume of the lattice $V^{-1}\mathbb{Z}^d$ we get that this box cannot contain more than approximately
\[
N_n=\frac{c}{\sqrt{|{\rm disc}(\chi_\mathbf{r})|}} \prod_{i=1}^d\sum_{j=0}^{k-1}\left| \beta_{i}^j\right|
\]
elements. If $|\beta_i|<1$ then $ \sum_{j=0}^{k-1} \left|\beta_{i}^j\right| = O(1)$. Since $\chi_\mathbf{r}$ is irreducible $|\beta_i|=1$ implies that $\beta_i$ is non-real.  If $|\beta_i|=1$ then we have the estimate $ \sum_{j=0}^{k-1}\left|\beta_{i}^j\right| = O(n)$ for this sum as well for the conjugate sum. Let $m$ be the number of pairs of non-real roots of $\chi_\mathbf{r}$ that have modulus $1$. Then these considerations yield that
\begin{equation}\label{withoutbirthday}
N_n \le \frac{c}{\sqrt{|{\rm disc}(\chi_\mathbf{r})|}} n^{2m}.
\end{equation}
Unfortunately, this estimate doesn't allow us to get any conclusion on the periodicity of the orbit of $(1,0,\ldots,0)^t$. We thus make the following {\it assumption}: we assume that for each fixed $\beta_i$ with $|\beta_i|=1$ the quantities $|\eps_j\beta_i^j|$ ($0\le j\le k-1$) in \eqref{eq:SalemBox} behave like identically distributed independent random variables. Then, according to the central limit theorem, we have that the sums in \eqref{eq:SalemBox}
can be estimated by
\[
 \left|\sum_{j=0}^{k-1} \varepsilon_j\beta_{i}^j\right| = O(\sqrt{n}).
\]
Using this argument,  we can replace \eqref{withoutbirthday} by the better estimate
\[
N_n \le \frac{c}{\sqrt{|{\rm disc}(\chi_\mathbf{r})|}} n^{m}.
\]
Suppose that $m=1$. If $|{\rm disc}(\chi_\mathbf{r})|$ is large enough, the set $\{V^{-1} \tau_\mathbf{r}^k((1,0,\ldots,0)^t)\;:\; 0\le k < n\}$ would be contained in $K_n$ which contains less than $n$ points of the lattice $V^{-1}\mathbb{Z}^d$ in it. Thus there have to be some repetitions in the orbit of $(1,0,\ldots,0)^t$. This implies that it is periodic.

For $m=2$ and a sufficiently large discriminant the set $\{V^{-1} \tau_\mathbf{r}^k((1,0,\ldots,0)^t)\;:\; 0\le k < n\}$ would contain considerably more than $\sqrt{N_n}$ ``randomly chosen'' points taken from a box with $N_n$ elements. Thus, according to the ``birthday paradox'' (for $n\to\infty$) with probability $1$ the orbit ``picks'' twice the same point, which again implies periodicity.

For $m > 3$ this model suggests that there may well exist aperiodic orbits as there are ``too many'' choices to pick points. Summing up we come to the following conjecture.

\begin{conjecture}
Let $\mathbf{r} \in E_d^{(\mathbb{C})}$ be a  parameter with irreducible polynomial $\chi_\mathbf{r}$. Let $m$ be the number of pairs of complex conjugate roots $(\alpha, \bar\alpha)$ of $\chi_\mathbf{r}$ with $|\alpha|=1$. Then almost every orbit of $(1,0,\ldots,0)^t$ under $\tau_\mathbf{r}$ is periodic if $m=1$ or $m=2$ and aperiodic if $m \ge 3$.
\end{conjecture}

Note that the cases $m=1$ and $m=2$ contain the Salem numbers of degree $4$ and $6$, respectively. Moreover, Salem numbers of degree $8$ and higher are contained in the cases $m\ge 3$. This is in accordance with \cite[Section~6]{Boyd:96}.

\section{Shift radix systems with finiteness property: the sets $\D_d^{(0)}$}\label{sec:Dd0}

As was already observed by Akiyama {\it et al.}~\cite{Akiyama-Borbeli-Brunotte-Pethoe-Thuswaldner:05} the set $\D_d^{(0)}$ can be constructed from the set $\D_d$ by ``cutting out'' families of convex polyhedra. Moreover, it is known that for $d\ge 2$ infinitely many such ``cut out polyhedra'' are needed to characterize $\D_d^{(0)}$ in this way (see Figure~\ref{d20} for an illustration of $\D_2^{(0)}$).

A list $\pi$ of pairwise distinct vectors
\begin{equation}\label{eq:cycle}
(a_{j},\ldots,a_{d-1+j})^t \qquad (0\le j\le L-1)
\end{equation}
with $a_{L}=a_0,\ldots, a_{L+d-1}=a_{d-1}$ is called a {\em cycle of vectors}.  To the cycle $\pi$ we associate the (possibly degenerate or empty) polyhedron
\[
P(\pi) = \{(r_0,\ldots,r_{d-1})\; : \; 0\le r_0 a_{j} + \cdots + r_{d-1} a_{d-1+j} + a_{d+j} <1 \hbox{ holds for } 0\le
j\le L-1\}.
\]
By definition the cycle in \eqref{eq:cycle} forms a periodic orbit of $\tau_{{\bf r}}$ if and only if  $\mathbf{r}\in{P}(\pi)$. Since ${\bf r}\in \D_d^{(0)}$ if and only if $\tau_{{\bf r}}$ has no non-trivial periodic orbit it follows that
$$
\label{cutout} \D_d^{(0)} = \D_d \setminus \bigcup_{\pi\not=\mathbf{0}}
{P}(\pi),
$$
where the union is taken over all non-zero cycles $\pi$ of vectors. The family of all (non-empty) polyhedra corresponding to this choice is called the family of {\it cut out polyhedra of $\D_d^{(0)}$}.

\begin{example}[{see \cite{Akiyama-Borbeli-Brunotte-Pethoe-Thuswaldner:05}}]
Let $\pi$ be a cycle of period $5$ in $\mathbb{Z}^2$ given by
$$
(-1,-1)^t \rightarrow (-1,1)^t \rightarrow (1,2)^t \rightarrow (2,1)^t
\rightarrow (1,-1)^t \rightarrow (-1,-1)^t.
$$
Then $P(\pi)$ gives the topmost cut out triangle in the approximation of $\D_2^{(0)}$ in Figure~\ref{d20}.
\end{example}

For $d=1$, the set $\D_1^{(0)}$ can easily be characterized.

\begin{proposition}[{{\em cf.}~\cite[Proposition~4.4]{Akiyama-Borbeli-Brunotte-Pethoe-Thuswaldner:05}}]\label{pro:D10}
$$
\D_1^{(0)}= [0,1).
$$
\end{proposition}

The proof is an easy exercise.

\subsection{Algorithms to determine $\D_d^{(0)}$} \label{sec:algorithms}
To show that a given point $\mathbf{r}\in \D_d$ does not belong to $\D_d^{(0)}$ it is sufficient to show that $\tau_\mathbf{r}$ admits a non-trivial periodic orbit, {\it i.e.}, to show that there is a polyhedron $\pi$ with $\mathbf{r}\in P(\pi)$. To prove the other alternative is often more difficult. We provide an algorithm (going back to Brunotte~\cite{Brunotte:01}) that decides whether a given $\mathbf{r}\in\E_d$ is in $\D_d^{(0)}$ or not.

As usual, denote the standard basis vectors of $\mathbb{R}^d$ by $\{\mathbf{e}_1,\ldots,\mathbf{e}_d\}$.

\begin{definition}[Set of witnesses]\label{def:sow}
A {\em set of witnesses} associated with a parameter $\mathbf{r}\in \mathbb{R}^d$ is a set $\mathcal{V}_\mathbf{r}$ satisfying
\begin{itemize}
\item[(i)] $\{\pm \mathbf{e}_1,\ldots,\pm\mathbf{e}_d\}\subset \mathcal{V}_\mathbf{r}$ and
\item[(ii)] $\mathbf{z}\in \mathcal{V}_\mathbf{r}$ implies that $\{ \tau_\mathbf{r}(\mathbf{z}),-\tau_\mathbf{r}(-\mathbf{z})\} \subset \mathcal{V}_\mathbf{r}$,
\end{itemize}
\end{definition}

The following theorem justifies the terminology ``set of witnesses''.

\begin{theorem}[{see {\em e.g.}~\cite[Theorem~5.1]{Akiyama-Borbeli-Brunotte-Pethoe-Thuswaldner:05}}]\label{thm:Brunotte}
Choose $\mathbf{r}\in \mathbb{R}^d$  and let $\mathcal{V}_\mathbf{r}$ be a set of witnesses for $\mathbf{r}$. Then
\[
\mathbf{r}\in\D_d^{(0)} \quad\Longleftrightarrow\quad \hbox{for each } \mathbf{z}\in\mathcal{V}_\mathbf{r} \hbox{ there is } k\in\mathbb{N} \hbox{ such that } \tau_\mathbf{r}^k(\mathbf{z})=\mathbf{0}.
\]
\end{theorem}

\begin{proof}
It is obvious that the left hand side of the equivalence implies the right hand side. Thus assume that the right hand side holds. Assume that $\mathbf{a}\in \mathbb{Z}^d$ has finite SRS expansion, {\it i.e.}, there exists $\ell\in\mathbb{N}$ such that $\tau_\mathbf{r}^\ell(\mathbf{a})=\mathbf{0}$ and choose $\mathbf{b}\in \{\pm \mathbf{e}_1,\ldots,\pm\mathbf{e}_d\}$. We show now that also $\mathbf{a}+\mathbf{b}$ has finite SRS expansion. As $\mathcal{V}_\mathbf{r}$ is a set of witnesses, using  Definition~\ref{def:sow}~(ii) we derive from the almost linearity condition stated in \eqref{almostlinear} that
\[
\tau_\mathbf{r}(\mathbf{a} + \mathcal{V}_\mathbf{r}) \subset \tau_\mathbf{r}(\mathbf{a}) + \mathcal{V}_\mathbf{r}.
\]
Iterating this for $\ell$ times and observing that $\mathbf{b} \in \mathcal{V}_\mathbf{r}$ holds in view of Definition~\ref{def:sow}~(i), we gain
\[
\tau_\mathbf{r}^\ell(\mathbf{a} + \mathbf{b}) \in \tau_\mathbf{r}^\ell(\mathbf{a}) + \mathcal{V}_\mathbf{r} = \mathcal{V}_\mathbf{r}.
\]
Thus our assumption implies that $\mathbf{a} + \mathbf{b}$ has finite SRS expansion. Since $\mathbf{0}$ clearly has finite SRS expansion, the above argument inductively proves that $\mathbf{r}\in\D_d^{(0)}$.
\end{proof}

For each $\mathbf{r} \in \E_d$ we can now check algorithmically whether $\mathbf{r}\in\D_d^{(0)}$ or not. Indeed, if $\mathbf{r}\in\E_d$ the matrix $R(\mathbf{r})$ is contractive. In view of \eqref{linear} this implies that Algorithm~\ref{alg:1} yields a finite set of witnesses $\mathcal{V}_\mathbf{r}$ for $\mathbf{r}$ after finitely many steps. Since Proposition~\ref{EdDdEd} ensures that each orbit of $\tau_\mathbf{r}$ is ultimately periodic for $\mathbf{r} \in \E_d$,  the criterion in Theorem~\ref{thm:Brunotte} can be checked algorithmically for each $\mathbf{z}\in\mathcal{V}_\mathbf{r}$.

\begin{algorithm}
\begin{algorithmic}
\Require{$\mathbf{r} \in \E_d$}
\Ensure{A set of witnesses $\mathcal{V}_\mathbf{r}$ for $\mathbf{r}$}
\State $W_0 \gets \{\pm \mathbf{e}_1,\ldots,\pm \mathbf{e}_d\}$
\State $i \gets 0$
\Repeat
\State $W_{i+1} \gets W_i \cup \tau_\mathbf{r}(W_i) \cup (-\tau_\mathbf{r}(-W_i))$
\State $i\gets i+1$
\Until{$W_i = W_{i-1}$}
\State $\mathcal{V}_\mathbf{r} \gets W_i$
\end{algorithmic}
\caption{An algorithm to calculate the set of witnesses of a parameter $\mathbf{r} \in \E_d$ (see~\cite[Section~5]{Akiyama-Borbeli-Brunotte-Pethoe-Thuswaldner:05})}
\label{alg:1}
\end{algorithm}

We can generalize these ideas and set up an algorithm that allows to determine small regions of the set $\D_d^{(0)}$. To this matter we define a set of witnesses for a compact set.

\begin{definition}[Set of witnesses for a compact set]\label{def:sowR}
Let $H \subset \mathbb{R}^d$ be a non-empty compact set and for $\mathbf{z}=(z_0,\ldots,z_{d-1})\in \mathbb{Z}^d$ define the functions
\begin{align}
M(\mathbf{z}) &= \max\{-\lfloor\mathbf{r}\mathbf{z}\rfloor\;:\; \mathbf{r}\in H\} ,\label{eq:M} \\
T(\mathbf{z}) &= \{(z_1,\ldots,z_{d-1},j)^t\;:\; -M(-\mathbf{z}) \le j \le M(\mathbf{z})\}. \nonumber
\end{align}
A set $\mathcal{V}_H$ is called a {\em set of witnesses for the region $H$}  if it satisfies
 \begin{itemize}
\item[(i)] $\{\pm \mathbf{e}_1,\ldots,\pm\mathbf{e}_d\}\subset \mathcal{V}_H$ and
\item[(ii)] $\mathbf{z}\in \mathcal{V}_H$ implies that $T(\mathbf{z}) \subset \mathcal{V}_H$.
\end{itemize}

A {\em graph $\mathcal{G}_H$ of witnesses for $H$} is a directed graph whose vertices are the elements of a set of witnesses $\mathcal{V}_H$ for $H$ and with a directed edge from $\mathbf{z}$ to $\mathbf{z}'$ if and only if $\mathbf{z}'\in T(\mathbf{z})$.
\end{definition}

Each cycle of a graph of witnesses $\mathcal{G}_H$ is formed by a cycle of vectors (note that cycles of graphs are therefore considered to be {\em simple} in this paper). If $\mathbf{0}$ is a vertex of $\mathcal{G}_H$ then  $\mathcal{G}_H$ contains the cycle $\mathbf{0} \to \mathbf{0}$. We call this cycle {\em trivial}. All the other cycles in $\mathcal{G}_H$ will be called {\em non-trivial}.

\begin{lemma}[{see~\cite[Section~5]{Akiyama-Borbeli-Brunotte-Pethoe-Thuswaldner:05}}]\label{lem:Brunotte}
The following assertions are true.

\begin{itemize}
 \item[(i)] A set of witnesses for $\mathbf{r}$ is a set of witnesses for the region $H=\{\mathbf{r}\}$ and vice versa.

\item[(ii)] Choose $\mathbf{r}\in \D_d$  and let $\mathcal{G}_{H}$ be a graph of witnesses for $H=\{\mathbf{r}\}$. If
$\mathbf{r}\not\in\D_d^{(0)}$ then $\mathcal{G}_{H}$ has a non-trivial cycle $\pi$ with $\mathbf{r}\in P(\pi)$.

\item[(iii)] A graph of witnesses for a compact set $H$ is a graph of witnesses for each non-empty compact subset of $H$.
\end{itemize}
\end{lemma}

\begin{proof}
All three assertions are immediate consequences of Definitions~\ref{def:sow} and~\ref{def:sowR}.
\end{proof}

We will use this lemma in the proof of the following result.

\begin{theorem}[{see {\em e.g.}~\cite[Theorem~5.2]{Akiyama-Borbeli-Brunotte-Pethoe-Thuswaldner:05}}]\label{thm:BrunotteR}
Let $H$ be the convex hull of the finite set $\{\mathbf{r}_1,\ldots,\mathbf{r}_k\}\subset\D_d$. If $\mathcal{G}_H$ is a graph of witnesses for $H$ then
\[
\D_d^{(0)} \cap H = H \setminus \bigcup_{\begin{subarray}{c}\pi \in \mathcal{G}_H\\ \pi \not = \mathbf{0} \end{subarray}} P(\pi)
\]
where the union is extended over all non-zero cycles of $\mathcal{G}_H$. Thus the set $\D_d^{(0)} \cap H$ is described by the graph $\mathcal{G}_H$.
\end{theorem}

\begin{proof}
As obviously
\[
\D_d^{(0)} \cap H = H \setminus \bigcup_{\pi \not = \mathbf{0}} P(\pi) \subset  H \setminus \bigcup_{\begin{subarray}{c}\pi \in \mathcal{G}_H\\ \pi \not = \mathbf{0} \end{subarray}} P(\pi)
\]
it suffices to prove the reverse inclusion. To this matter assume that $\mathbf{r} \not\in \D_d^{(0)} \cap H$. W.l.o.g.\ we may also assume that $\mathbf{r}\in H$. Then by Lemma~\ref{lem:Brunotte}~(iii) the graph $\mathcal{G}_H$ is a graph of witnesses for $\{\mathbf{r}\}$. Thus Lemma~\ref{lem:Brunotte}~(ii) implies that $\mathcal{G}_H$ has a non-trivial cycle $\pi$ with $\mathbf{r}\in P(\pi)$ and, hence, $\mathbf{r}\not\in H \setminus \bigcup_{\begin{subarray}{c}\pi \in \mathcal{G}_H\\ \pi \not = \mathbf{0} \end{subarray}} P(\pi)$.
\end{proof}

Theorem~\ref{thm:BrunotteR} is of special interest if there is an algorithmic way to construct the graph $\mathcal{G}_H$. In this case it leads to an algorithm for the description of $\D_d^{(0)}$ in the region $H$.

To be more precise, assume that $H \subset \E_d$ is the convex hull of a finite set $\mathbf{r}_1,\ldots, \mathbf{r}_k$. Then the maximum $M(\mathbf{z})$ in \eqref{eq:M} is easily computable and analogously to Algorithm~\ref{alg:1} we can set up Algorithm~\ref{alg:2} to calculate the set of vertices of a graph of witnesses $\mathcal{G}_H$ for $H$. As soon as we have this set of vertices the edges can be constructed from the definition of a graph of witnesses. The cycles can then be determined by classical algorithms ({\it cf.\ e.g.}~\cite{Johnson:75}).

\begin{algorithm}
\begin{algorithmic}
\Require{$H \subset \E_d$ which is the convex hull of $\mathbf{r}_1,\ldots, \mathbf{r}_k$}
\Ensure{The states $\mathcal{V}_H$ of a graph of witnesses $\mathcal{G}_H$ for $H$}
\State $W_0 \gets \{\pm \mathbf{e}_1,\ldots,\pm \mathbf{e}_d\}$
\State $i \gets 0$
\Repeat
\State $W_{i+1} \gets W_i \cup T(W_i)$
\State $i\gets i+1$
\Until{$W_i = W_{i-1}$}
\State $\mathcal{V}_H\gets W_i$
\end{algorithmic}
\caption{An algorithm to calculate the set of witnesses of $H \subset \E_d$ (see~\cite[Section~5]{Akiyama-Borbeli-Brunotte-Pethoe-Thuswaldner:05})}
\label{alg:2}
\end{algorithm}

We need to make sure that Algorithm~\ref{alg:2} terminates. To this matter set $I(\mathbf{z})=\{\mathbf{sz}\,:\, \mathbf{s}\in H\}$. As $H$ is convex, this set is an interval. Thus, given $\mathbf{z}$, for each $\mathbf{z}'\in T(\mathbf{z})$ we can find $\mathbf{r}\in H$ such that
\[
\mathbf{z}' = \tau_{(\mathbf{r}}(\mathbf{z})=R(\mathbf{r}) \mathbf{z} + \mathbf{v} \qquad (\hbox{for some } \mathbf{v} \hbox{ with } ||\mathbf{v}||_\infty<1).
\]
As $\mathbf{r}\in \E_d$ we can choose a norm that makes $R(\mathbf{r})$ contractive for a particular $\mathbf{r}$. However, in general it is not possible to find a norm that makes $R(\mathbf{r})$ contractive for each $\mathbf{r}\in H$ unless $H$ is small enough in diameter. Thus, in order to ensure that  Algorithm~\ref{alg:2} terminates we have to choose the set $H$ sufficiently small.

There seems to exist no algorithmic way to determine sets $H$ that are ``small enough'' to make Algorithm~\ref{alg:2} finite. In practice one starts with some set $H$. If the algorithm does not terminate after a reasonable amount of time one has to subdivide $H$ into smaller subsets until the algorithm terminates for each piece. This strategy has been used so far to describe large parts of $\D_2^{(0)}$ (see {\it e.g.}~\cite{Akiyama-Brunotte-Pethoe-Thuswaldner:06,Surer:07}). Very recently, Weitzer~\cite{Weitzer:13} was able to design a new algorithm which describes $\D_d^{(0)} \cap H$ for arbitrary compact sets $H \subset \E_d$.  He does not need any further assumptions on $H$. Moreover, he is able to show that the set $\D_2^{(0)}$ is not connected and has non-trivial fundamental group.

\begin{remark}
These algorithms can easily be adapted to characterize the sets $\D_d^{(p)}$ used in Section~\ref{sec:realDd} (see \cite[Section~6.1]{Kirschenhofer-Pethoe-Surer-Thuswaldner:10}).
\end{remark}

We conclude this section with rhe following fundamental problem.

\begin{problem}
Give a complete description for $\D_d$ if $d \ge 2$. 
\end{problem}

\subsection{The finiteness property on the boundary of $\E_d$} Let us now focus on the relation between the sets $\D_d^{(0)}$ and  $\E_d$. We already observed that the application of $\tau_\mathbf{r}$ performs a multiplication by the matrix $R(\mathbf{r})$ followed by a round-off. If $\mathbf{r} \in \partial \mathcal{D}_d$, then $R(\mathbf{r})$ has at least one eigenvalue of modulus $1$. Thus multiplication by $R(\mathbf{r})$ will not contract along the direction of the corresponding eigenvector $\mathbf{v}$. If we consider a typical (large) orbit of $\tau_\mathbf{r}$, it is reasonable to assume that the successive ``round-off errors'' will --- even though they may not cancel out by the heuristics given in Section~\ref{sec:Heuristic} --- not always draw the orbit towards $\mathbf{0}$. This would imply that such an orbit will sometimes not end up at $\mathbf{0}$ if it starts far enough away from the origin in the direction of $\mathbf{v}$. More precisely, the following conjecture was stated by Akiyama {\it et al.}~\cite{Akiyama-Borbely-Brunotte-Pethoe-Thuswaldner:06}.

\begin{conjecture}\label{c69}
For $d\in\mathbb{N}$ we have
$$
\D_d^{(0)}\subset \E_d.
$$
\end{conjecture}

In other words: {\it Let $\mathbf{r}\in\mathbb{R}^d$. If
$\tau_{\mathbf r}$ has the finiteness property then, according to
the conjecture, each of the
eigenvalues of $R(\mathbf{r})$ has modulus strictly less than one.} Since by Proposition~\ref{EdDdEd} we have $\D_d^{(0)}\subset\D_d \subset \overline{\E_{d}}$ it remains to check all parameters $\mathbf{r}$ giving rise to a matrix
$R(\mathbf{r})$ whose eigenvalues have modulus at most one with
equality in at least one case.
Therefore Conjecture~\ref{c69} is equivalent to
$$
\mathcal{D}_d^{(0)} \cap \partial \mathcal{D}_d = \emptyset.
$$
This is of course trivially true for
$d=1$ (see Proposition~\ref{pro:D10}). It has been proved for $d=2$ by Akiyama {\it et al.}~\cite{Akiyama-Brunotte-Pethoe-Thuswaldner:06} (see Corollary~\ref{D2boundarycorollary}).
In the proofs for the cases $d=1$ and $d=2$ for all
$\mathbf{r}\in \E_d$, explicit orbits that do not end up at $\mathbf{0}$ are constructed.
For $d=3$ this seems no
longer possible for all parameters $\mathbf{r}\in \partial\E_d$.
Nevertheless  Brunotte and the authors could
settle the instance $d=3$.

\begin{theorem}[{{\em cf.}~\cite{Brunotte-Kirschenhofer-Thuswaldner:12}}]\label{thm:D3}
$$
\D_3^{(0)}\subset \E_3.
$$
\end{theorem}

In the following we give a very rough outline of the idea of the proof. In Figure \ref{E3} we see the set $\E_3$. The boundary of this set can be decomposed according to \eqref{Edboundary}. Moreover, the following parameterizations hold (see \cite{Kirschenhofer-Pethoe-Surer-Thuswaldner:10})
\begin{eqnarray}
E_3^{(1)}&=&\{(s,s+t+st,st+t+1) \;:\; -1\le s,t\le 1\}, \label{eq:E1}
\\
E_3^{(-1)}&=& \{(-s,s-t-st,st+t-1) \;:\; -1\le s,t\le
1\}\label{eq:E-1},\quad\hbox{and}\\
E_3^{(\mathbb{C})}&=&\{(t,st+1,s+t) \;:\; -2 < s < 2 ,\, -1\le t\le
1\}.\label{eq:EC}
\end{eqnarray}
The sets $E_3^{(1)}$ and $E_3^{(-1)}$ can be treated easily, see Proposition~\ref{DdboundaryPartialResults}~(i) and~(ii).
The more delicate instance is constituted by the elements of $E_3^{(\mathbb{C})}$. Here the decomposition of the parameter region depicted in Figure \ref{Aufteilung} is helpful.

\begin{figure}
\centering \leavevmode
\includegraphics[width=0.6\textwidth]{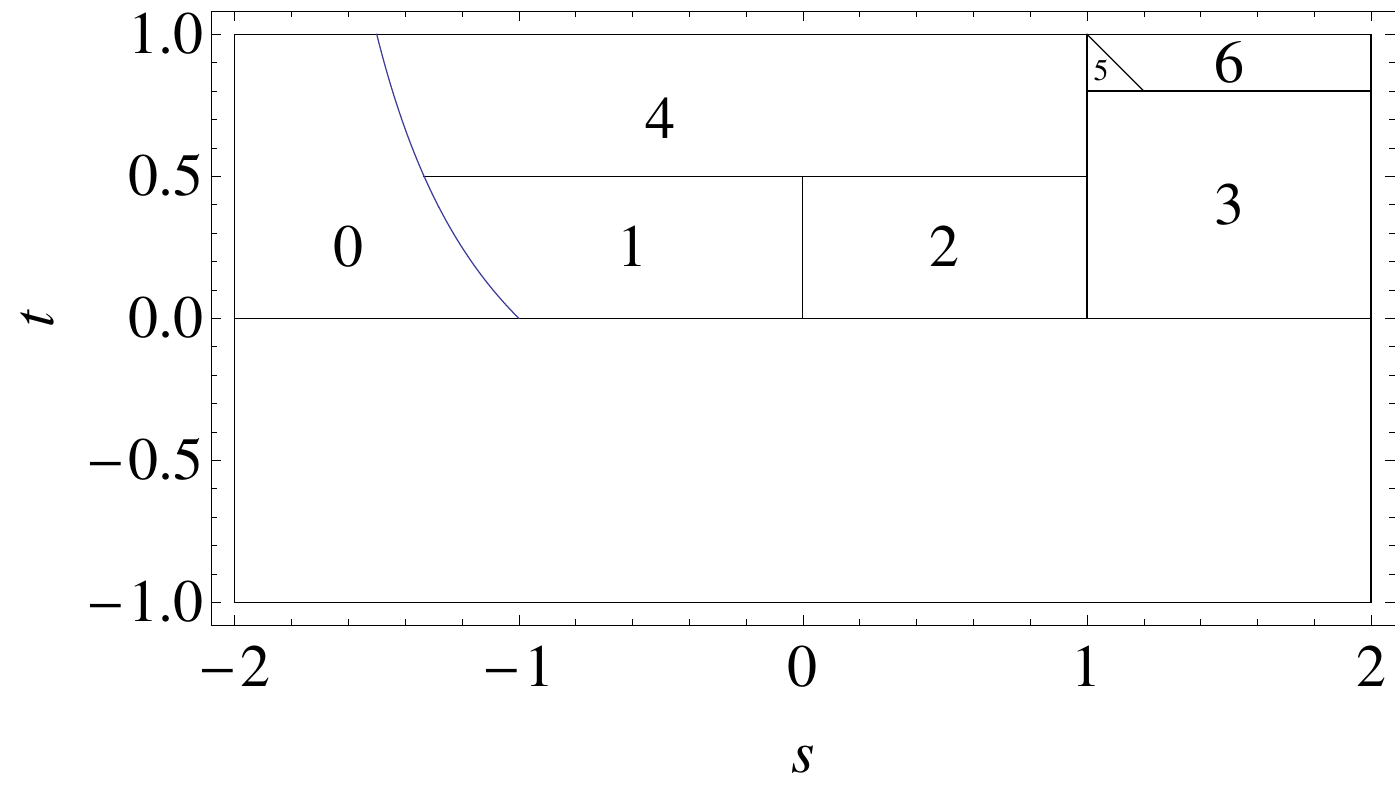}
\caption{The subdivision of the parameter region used in the proof of Theorem~\ref{thm:D3}}
\label{Aufteilung}
\end{figure}

Whereas for several subregions it is again possible to explicitly construct non-trivial cycles as in the instances mentioned above, this seems no longer the case {\it e.g.} for the regions
labelled $1$, $2$, $3$ or $5$ in Figure~\ref{Aufteilung}. Here the following idea can be applied.
For the parameters in the regions in question it can be proved that for large $n$ the set $\tau_{\mathbf{r}}^{-n}(\mathbf{0})$, where $\tau_{\mathbf{r}}^{-1}$ denotes the preimage of $\tau_{\mathbf{r}}$, has finite intersection with a subspace that is bounded by two hyperplanes. Thereby it can be concluded that some elements of this subspace belong to  periodic orbits of $\tau_{\mathbf{r}}$ that do not end up at $\mathbf{0}$ without constructing these orbits explicitly.

\medskip

Peth\H{o}~\cite{Pethoe:09} has studied the instance of the latter problem
where some eigenvalues of $R(\mathbf{r})$ are roots of unity. In the following proposition we give a summary of the partial results known for arbitrary dimensions.

\begin{proposition}\label{DdboundaryPartialResults}
Assume that $\mathbf{r}=(r_0,\ldots, r_{d-1})\in \partial\mathcal{D}_d$. Then $\mathbf{r}\not\in\mathcal{D}_d^{(0)}$ holds if one of the following conditions is true.
\begin{itemize}
\item[(i)] $\mathbf{r} \in E_d^{(1)}$.
\item[(ii)] $\mathbf{r} \in E_d^{(-1)}$.
\item[(iii)] $r_0<0$.
\item[(iv)] Each root of $\chi_\mathbf{r}$ has modulus $1$.
\item[(v)] There is a Salem number $\beta$ such that $(X-\beta)\chi_\mathbf{r}(X)$, with $\chi_\mathbf{r}$ as in \eqref{chi},  is the minimal polynomial of $\beta$ over $\mathbb{Z}$.
\item[(vi)]  $\mathbf{r}=(\frac{\pm 1}{p_0},\frac{p_{d-1}}{p_0},\ldots,\frac{p_1}{p_0})$ with $p_0,\ldots, p_{d-1}\in \mathbb{Z}$.
\end{itemize}
\end{proposition}

\begin{remark}
Item (iii) is a special case of \cite[Theorem~2.1]{Akiyama-Brunotte-Pethoe-Thuswaldner:07}. 
Item (v) is a restatement of the fact that beta expansions w.r.t.\ Salem numbers never satisfy property (F), see {\it e.g.}~\cite[Section~2]{Boyd:89} or \cite[Lemma~1(b)]{Frougny-Solomyak:92}. Item (vi) is equivalent to the fact that CNS polynomials satisfying the finiteness property need to be expanding ({\it cf.}~\cite[Theorem~6.1]{Pethoe:91};  see also~\cite{Kovacs-Pethoe:91}).
\end{remark}

\begin{proof}
In Item (i) we have that $r_0+\cdots+r_{d-1}=-1$. Thus, choosing $\mathbf{z}=(n,\ldots,n)^t\in \mathbb{Z}^d$ we get that
\[
\tau_{\mathbf{r}}(\mathbf{z}) = (n,\ldots, n, -\lfloor \mathbf{rz} \rfloor)^t = (n,\ldots,n, -\lfloor n(r_0+\cdots+r_{d-1}) \rfloor)^t = (n,\ldots,n)^t
\]
which exhibits a non-trivial cycle for each $n\not=0$. Similarly, in Item (ii), we use the fact that $r_0-r_1+r_2-+\cdots+(-1)^{d-1}r_{d-1}=(-1)^{d-1}$ in order to derive $\tau_\mathbf{r}^2(\mathbf{z})=\mathbf{z}$ for each $\mathbf{z}=(n,-n,\ldots,(-1)^{d-1}n)^t\in \mathbb{Z}^d$.

To prove Item (iii) one shows that for $r_0<0$ there is a half-space whose elements cannot have orbits ending up in~$\mathbf{0}$. 

To show that the result holds when (iv) is in force, observe that this condition implies that $r_0 \in\{-1,1\}$. This immediately yields $\tau_{\mathbf r}^{-1}(\mathbf{0})=\{\mathbf{0}\}$ and, hence, no orbit apart from the trivial one can end up in~$\mathbf{0}$.

In the proof of (v) one uses that a Salem number $\beta$ has the positive conjugate $\beta^{-1}$. As in (iii) this fact allows to conclude that there is a half-space whose elements cannot have orbits ending up in~$\mathbf{0}$. 

Item (vi) is proved using the fact that under this condition the polynomial $\chi_\mathbf{r}$ has a root of unity among its roots.
\end{proof}

\section{The geometry of shift radix systems}

\subsection{SRS tiles}

Let us assume that the matrix $R(\mathbf{r})$ is contractive and $\mathbf{r}$ is {\it reduced} in the sense that $\mathbf{r}$ does not include leading zeros. Then, as it was observed by Berth\'e {\it et al.}\ \cite{BSSST2011} the mapping
$\tau_\mathbf{r}$ can be used to define so-called SRS tiles in
analogy with the definition of tiles for other dynamical systems
related to numeration (\emph{cf.\ e.g.}\
\cite{Akiyama:02,Barat-Berthe-Liardet-Thuswaldner:06,Berthe-Siegel:05,Ito-Rao:06,Katai-Koernyei:92,Rauzy:82,Scheicher-Thuswaldner:01,Thurston:89}).
As it will turn out some of these tiles
are related to CNS tiles and beta-tiles in a way corresponding to
the conjugacies established already in Section 2.
Formally we have the following objects.

\begin{definition}[SRS tile]
Let $\mathbf{r}=(r_0,\ldots,r_{d-1}) \in \mathcal{E}_d$ with $r_0\not=0$ and $\mathbf{x} \in \mathbb{Z}^d$ be given.
The set
\[ 
\mathcal{T}_\mathbf{r}(\mathbf{x})= \mathop{\rm Lim}_{n\to\infty} R(\mathbf{r})^n \tau_\mathbf{r}^{-n}(\mathbf{x}) 
\]
(where the limit is taken with respect to the Hausdorff metric) is called the {\em SRS tile associated with $\mathbf{r}$}.
$\mathcal{T}_\mathbf{r}(\mathbf{0})$ is called the {\em central SRS tile associated with $\mathbf{r}$ located at $\mathbf{x}$}.
\end{definition}

In other words in order to build $\mathcal{T}_\mathbf{r}(\mathbf{x})$, the vectors are considered whose SRS expansion
coincides with the expansion of $\mathbf{x}$ up to an added finite
prefix and afterwards the expansion is renormalized. We mention that the existence of this limit is not trivially true but can be assured by using the contractivity of the operator $R(\mathbf{r})$.
\begin{example}
Let $\mathbf{r}=(\frac{4}{5},-\frac{49}{50})$. Approximations of the tile $\mathcal{T}_\mathbf{r}(\mathbf{0})$ are illustrated in Figure~\ref{fig:SRSspirale}.
\begin{figure}
\includegraphics[width=0.3\textwidth]{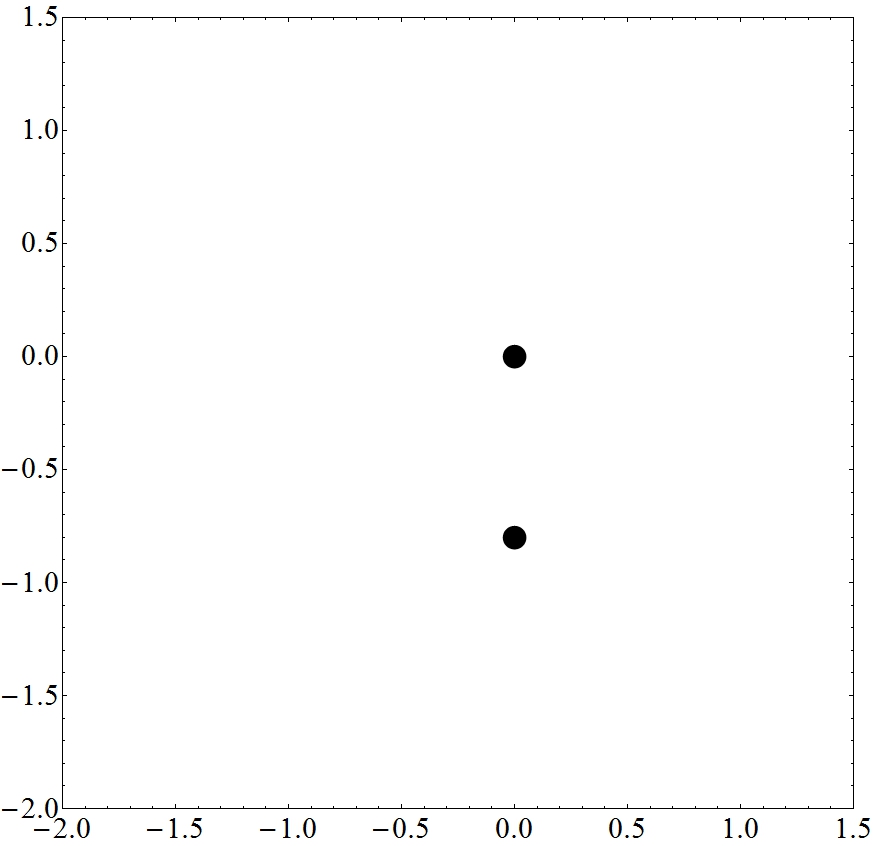}
\includegraphics[width=0.3\textwidth]{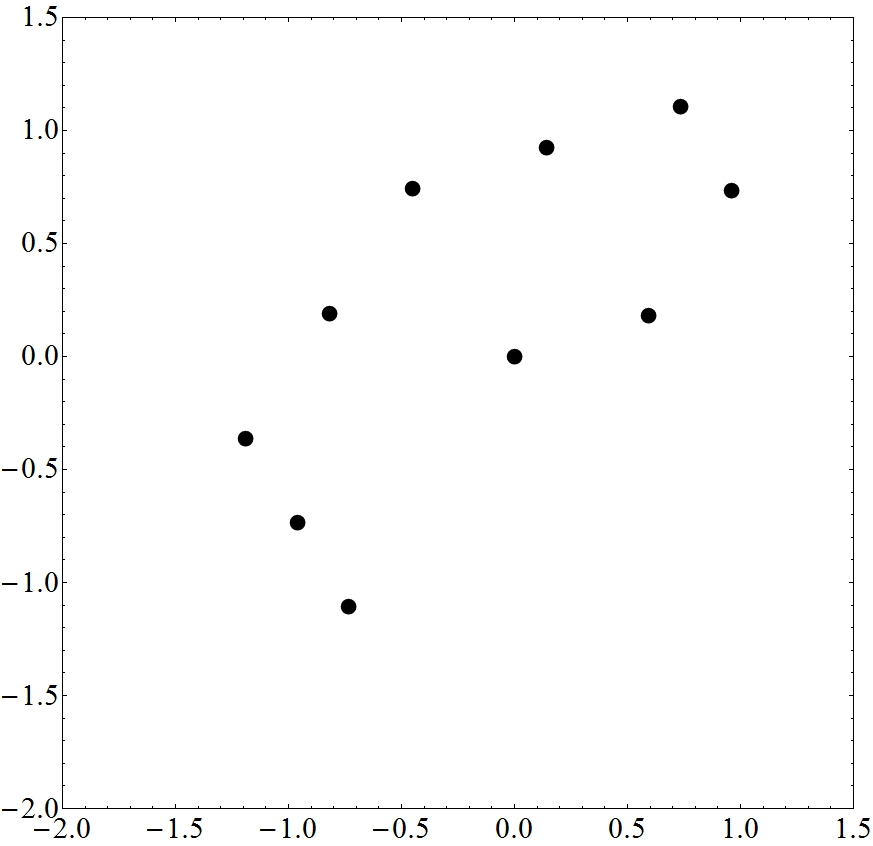}
\includegraphics[width=0.3\textwidth]{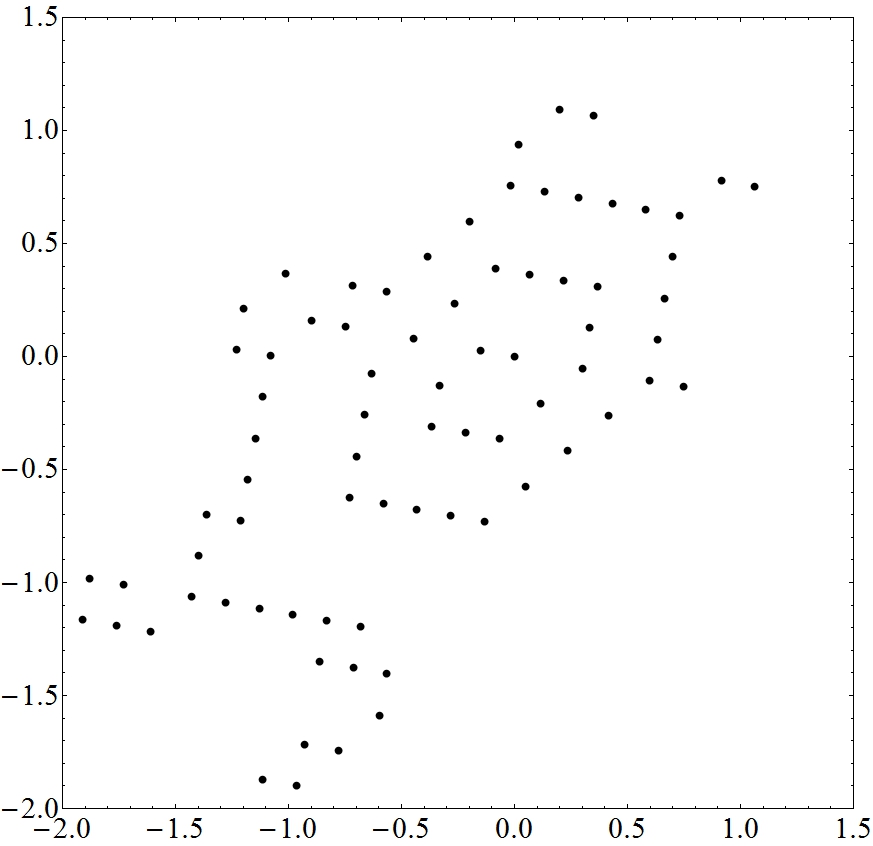} \\[.3cm]
\includegraphics[width=0.3\textwidth]{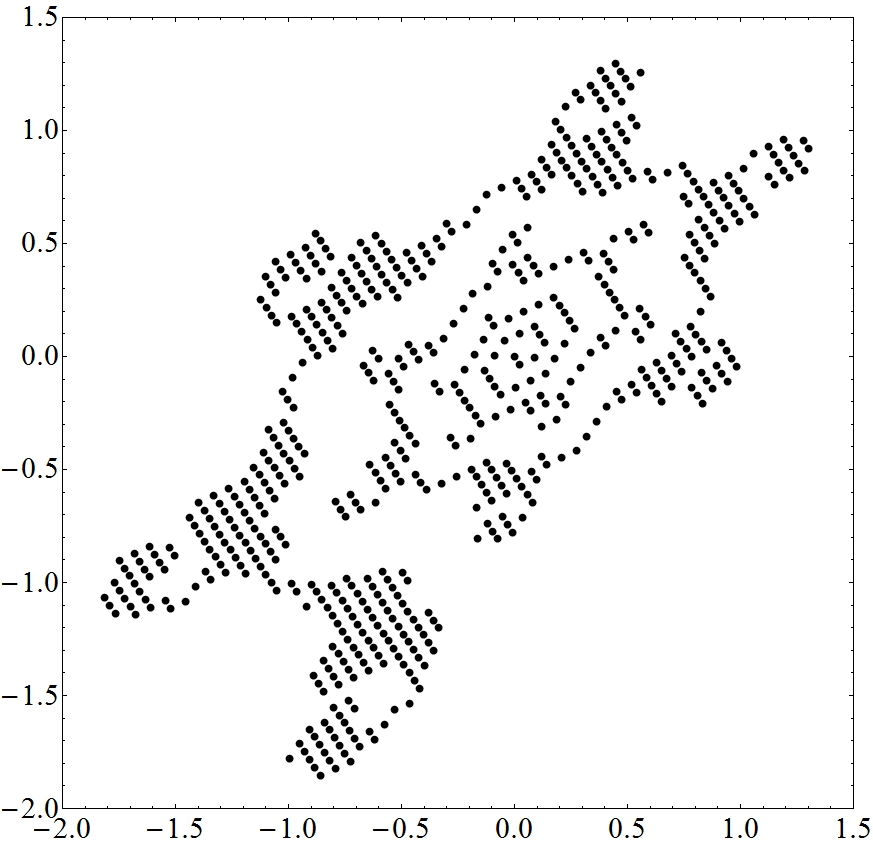}
\includegraphics[width=0.3\textwidth]{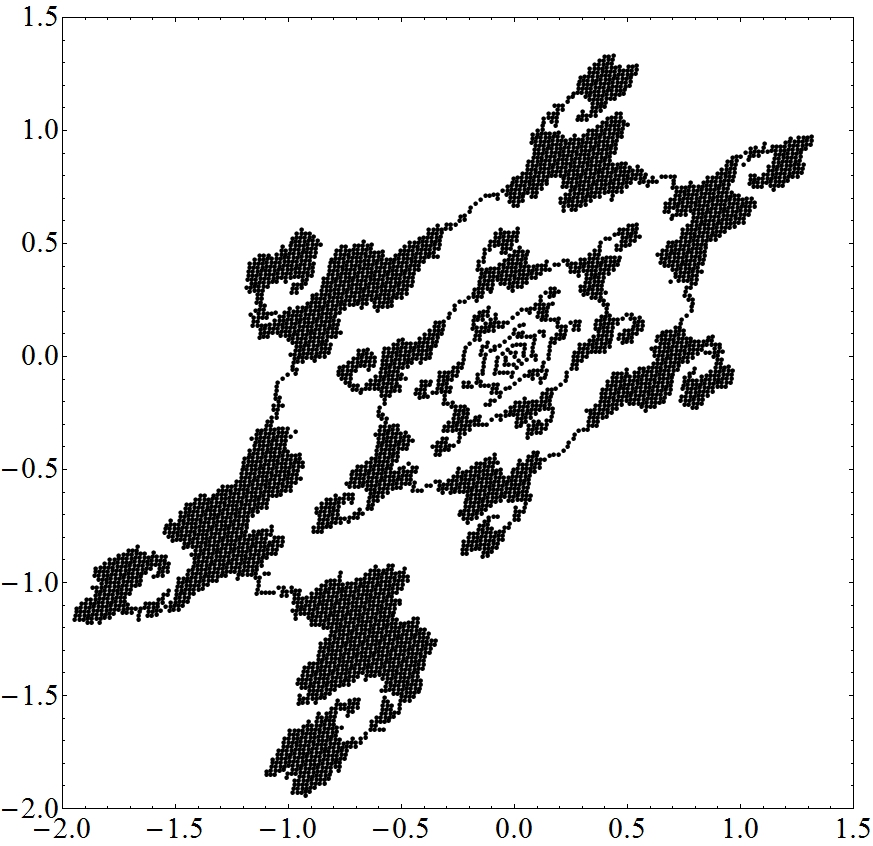}
\includegraphics[width=0.3\textwidth]{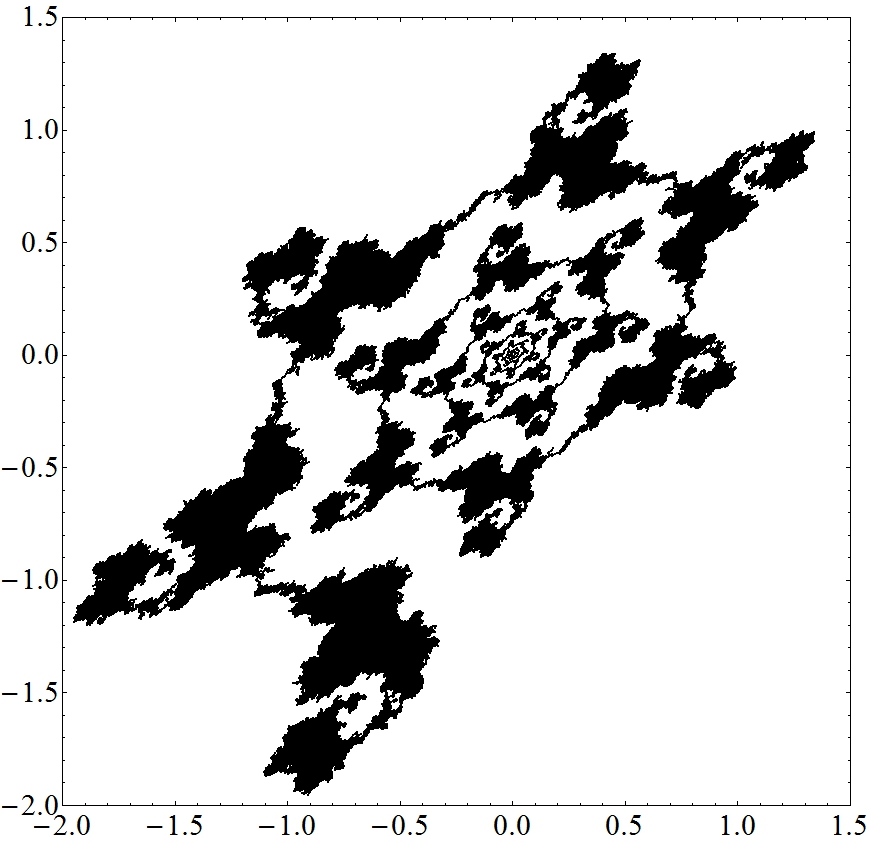}
\caption{Approximations of $\mathcal{T}_{\mathbf{r}}(\mathbf{0})$ for the parameter $\mathbf{r}=(\frac{4}{5},-\frac{49}{50})$: the images show $R(\mathbf{r})^k\tau_{\mathbf{r}}^{-k}(\mathbf{0})$ for $k=1,6,16,26,36,46$.
 \label{fig:SRSspirale}}
\end{figure}
\end{example}

The following Proposition summarizes some of the basic properties of SRS tiles.
\begin{proposition}[{\em cf.} {\cite[Section~3]{BSSST2011}}]\label{prop:basicSRS}
For each $\mathbf{r}=(r_0,\ldots,r_{d-1}) \in \mathcal{E}_d$ with $r_0\not=0$ we have the following results.
\begin{itemize}
\item $\mathcal{T}_\mathbf{r}(\mathbf{x})$ is compact for all $\mathbf{x} \in \mathbb{Z}^d$.
\item The family $\{\mathcal{T}_\mathbf{r}(\mathbf{x}) \;:\; \mathbf{x} \in \mathbb{Z}^d\}$ is locally finite.
\item $\mathcal{T}_\mathbf{r}(\mathbf{x})$ satisfies the set equation
\[
\mathcal{T}_\mathbf{r}(\mathbf{x}) = \bigcup_{\mathbf{y}\in\tau_\mathbf{r}^{-1}(\mathbf{x})} R(\mathbf{r}) \mathcal{T}_\mathbf{r}(\mathbf{y}).
\]
\item The collection $\{\mathcal{T}_\mathbf{r}(\mathbf{x})  \;:\; \mathbf{x} \in \mathbb{Z}^d\}$ covers $\mathbb{R}^d$, {\em i.e.},
\begin{equation}\label{eq:cov}
\bigcup_{\mathbf{x} \in \mathbb{Z}^d} \mathcal{T}_\mathbf{r}(\mathbf{x}) = \mathbb{R}^d.
\end{equation}
\end{itemize}
\end{proposition}

\begin{proof}[Sketch of the proof]
With respect to compactness observe first that Hausdorff limits are closed by definition. Using inequality \eqref{norminequ} it follows that every  $\mathcal{T}_\mathbf{r}(\mathbf{x})$, $\mathbf{x}\in\mathbb{Z}^d$, is contained in the closed ball of radius $R$ with center~$\mathbf{x}$ where
\begin{equation} \label{eq:R}
R := \sum_{n=0}^\infty \big\|R(\mathbf{r})^n(0,\ldots,0,1)^t\big\| \le \frac{\|(0,\ldots,0,1)^t\|}{1-\tilde\varrho}\,
\end{equation}
which establishes boundedness.
\medskip

Since $\mathcal{T}_\mathbf{r}(\mathbf{x})$ is uniformly bounded in $\mathbf{x}$ and the set of ``base points'' $\mathbb{Z}^d$ is a lattice, we also get that the family of SRS tiles $\{\mathcal{T}_\mathbf{r}(\mathbf{x}) \,:\, \mathbf{x}\in\mathbb{Z}^d\}$ is \emph{locally finite}, that is, any open ball intersects only a finite number of tiles of the family.

\medskip

In order to establish the set equation observe that
\begin{align*}
\mathcal{T}_\mathbf{r}(\mathbf{x}) &= \mathop{\rm
Lim}_{n\rightarrow\infty} R(\mathbf{r})^n
\tau_\mathbf{r}^{-n}(\mathbf{x}) = R(\mathbf{r})\mathop{\rm
Lim}_{n\rightarrow\infty}
\bigcup_{\mathbf{y}\in\tau_\mathbf{r}^{-1}(\mathbf{x})}
R(\mathbf{r})^{n-1} \tau_\mathbf{r}^{-n+1}(\mathbf{y}) \\
&= R(\mathbf{r})
\bigcup_{\mathbf{y}\in\tau_\mathbf{r}^{-1}(\mathbf{x})} \mathop{\rm
Lim}_{n\rightarrow\infty}R(\mathbf{r})^{n-1}
\tau_\mathbf{r}^{-n+1}(\mathbf{y})=R({\mathbf{r}})\bigcup_{\mathbf{y}\in\tau_\mathbf{r}^{-1}(\mathbf{x})}
\mathcal{T}_\mathbf{r}(\mathbf{y}). \hfill
\end{align*}

\medskip

It remains to prove the covering property. The lattice $\mathbb{Z}^d$ is obviously contained in the union in~\eqref{eq:cov}. Thus, by the set equation, the same is true for $R(\mathbf{r})^k\mathbb{Z}^d$ for each $k\in\mathbb{N}$. By the contractivity of $R(\mathbf{r})$, compactness of $\mathcal{T}_\mathbf{r}(\mathbf{x})$, and local finiteness of  $\{\mathcal{T}_\mathbf{r}(\mathbf{x}) \,:\, \mathbf{x}\in\mathbb{Z}^d\}$ the result follows.
\end{proof}

In the following we analyze the specific role of the central tile ({\em cf.}~\cite {BSSST2011}).
Since $\mathbf{0}\in\tau_{\mathbf{r}}^{-1}(\mathbf{0})$ holds for each $\mathbf{r}\in\E_d$, the origin is always an element of the central tile. However, the question  whether or not $\mathbf{0}$ is contained exclusively in the central tile plays an important role in numeration. In the case of beta-numeration, $\mathbf{0}$ is contained exclusively in the central beta-tile (see
Definition~\ref{def:betatile} below) if and only if property (F) (compare \eqref{PropertyF}) is satisfied (\cite{Akiyama:02,Frougny-Solomyak:92}), and a similar criterion holds for CNS ({\it cf.}~\cite{Akiyama-Thuswaldner:00}). It turns out that there is a corresponding characterization for SRS with finiteness property (see \cite{BSSST2011}).

\begin{definition}[Purely periodic point]
Let $\mathbf{r} \in \mathbb{R}^d$.
 An element $\mathbf{z} \in \mathbb{Z}^d$ is called \emph{purely periodic} point if $\tau_\mathbf{r}^p(\mathbf{z})=\mathbf{z}$ for some $p\ge 1$.
\end{definition}
Then we have the announced characterization.

\begin{theorem}[see~{\cite[Theorem~3.10]{BSSST2011}}]\label{t63}
Let $\mathbf{r}=(r_0,\ldots,r_{d-1}) \in \mathcal{E}_d$ with $r_0\not=0$ and $\mathbf{x} \in \mathbb{Z}^d$.
Then $\mathbf{0} \in \mathcal{T}_\mathbf{r}(\mathbf{x})$ if and only if $\mathbf{x}$ is purely periodic.
There are only finitely many purely periodic points.
\end{theorem}

\begin{proof}[Sketch of the proof.]
In order to establish that pure periodicity of $\mathbf{x}$ implies that $\mathbf{0}\in
\mathcal{T}_\mathbf{r}(\mathbf{x})$ observe first that by assumption we have
 $\mathbf{x} = \tau_\mathbf{r}^{kp}(\mathbf{x})$.
By the contractivity of the operator $R(\mathbf{r})$ it follows that $\mathbf{0}=\lim_{p\to  \infty} R(\mathbf{r})^{kp}\mathbf{x} \in \mathcal{T}_\mathbf{r}(\mathbf{x})$.

\medskip

With respect to the other direction observe that
by the set equation there is a sequence $(\mathbf{z}_n)_{n\ge 1}$ with $\mathbf{z}_n= \tau_\mathbf{r}^{-n}(\mathbf{x})$ and $\mathbf{0} \in R(\mathbf{r})^{n} \mathcal{T}_\mathbf{r}(\mathbf{z}_n)$. Thus $\mathbf{0} \in \mathcal{T}_\mathbf{r}(\mathbf{z}_n)$.
Therefore by the local finiteness of $\{\mathcal{T}_\mathbf{r}(\mathbf{x})  \;:\; \mathbf{x} \in \mathbb{Z}^d\}$ there are $n,k\in\mathbb{N}$ such that $\mathbf{z}_n = \mathbf{z}_{n+k},$
hence $\mathbf{x} = \tau_\mathbf{r}^{k}(\mathbf{x})$.

\medskip

Observing the proof of the last proposition it follows that only points $\mathbf{x}\in\mathbb{Z}^d$ with $\|\mathbf{x}\|\le R$ with $R$ as it \eqref{eq:R}, can be purely periodic.
Note that the latter property was already proved in~\cite{Akiyama-Borbeli-Brunotte-Pethoe-Thuswaldner:05}.
\end{proof}

There is an immediate consequence of the last theorem for SRS with finiteness property (see~\cite{BSSST2011}).

\begin{corollary}\label{cor:exc}
Let $\mathbf{r}=(r_0,\ldots,r_{d-1}) \in \mathcal{E}_d$ with $r_0\not=0$ be given.
Then $\mathbf{r} \in \mathcal{D}_d^{(0)}$ if and only if $\mathbf{0} \in \mathcal{T}_\mathbf{r}(\mathbf{0}) \setminus
\bigcup_{\mathbf{y}\neq \mathbf{0}} \mathcal{T}_\mathbf{r}(\mathbf{y})$.
\end{corollary}

In particular, for $\mathbf{r} \in \mathcal{D}_d^{(0)}$ the central tile $\mathcal{T}_\mathbf{r}(\mathbf{0})$ has non-empty interior. Nevertheless the following example demonstrates that  the interior of $\mathcal{T}_\mathbf{r}(\mathbf{x})$  may be empty for certain choices of $\mathbf{r}$ and $\mathbf{x}$.

\begin{example}[see \cite{BSSST2011}] \label{ex:pp}
Let $\mathbf{r}=(\frac{9}{10},-\frac{11}{20})$.
Then, with the points
\[
\mathbf{z}_1=(-1,-1)^t,\ \mathbf{z}_2=(-1,1)^t,\ \mathbf{z}_3=(1,2)^t,\ \mathbf{z}_4=(2,1)^t,\ \mathbf{z}_5=(1,-1)^t,
\]
we have the cycle
\[
\tau_\mathbf{r}:\mathbf{z}_1 \mapsto \mathbf{z}_2 \mapsto \mathbf{z}_3 \mapsto \mathbf{z}_4 \mapsto \mathbf{z}_5 \mapsto \mathbf{z}_1.
\]
Thus, each of these points is purely periodic.
By direct calculation we see that
\[
\tau_\mathbf{r}^{-1}(\mathbf{z}_1) =  \big\{(1,-1)^t\big\}=\{\mathbf{z}_5\},
\]
and similarly $\tau_\mathbf{r}^{-1}(\mathbf{z}_i)=\{\mathbf{z}_{i-1}\}$ for $i \in \{2,3,4,5\}$.
Hence every tile $\mathcal{T}_\mathbf{r}(\mathbf{z}_i)$, $i\in \{1,2,3,4,5\}$, consists of the single point $\mathbf{0}$.
\end{example}

\subsection{Tiling properties of SRS tiles}\label{sec:srstiling}
We saw in Proposition~\ref{prop:basicSRS} that the collection $\{\mathcal{T}_\mathbf{r}(\mathbf{x})  \;:\; \mathbf{x} \in \mathbb{Z}^d\}$ is a covering of $\mathbb{R}^d$. In \cite{BSSST2011} and \cite{ST:11} tiling properties of SRS tiles are proved. In the present section we review these results. As their proofs are involved we refrain from reproducing them here and confine ourselves to mention some main ideas.

We start with a basic definition ({\em cf.}~\cite[Definition~4.1]{BSSST2011}).

\begin{definition}[Weak $m$-tiling]\label{def:tilings}
Let $\mathcal{K}$ be a locally finite collection of subsets of $\mathbb{R}^d$ that cover $\mathbb{R}^d$. The {\em covering degree} of $\mathcal{K}$ is given by the number
\[
\min \{ \#\{ K \in \mathcal{K} \;:\; \mathbf{t}\in K\} \;:\; \mathbf{t} \in \mathbb{R}^d \}.
\]
The collection $\mathcal{K}$ is called a {\it weak $m$-tiling} if its covering degree is $m$, and
$\bigcap_{j=1}^{m+1} {\rm int}(K_j)= \emptyset$ for each choice of pairwise disjoint elements $K_1,\ldots, K_{m+1}$ of $\mathcal{K}$. A weak $1$-tiling is called a {\em weak tiling}.
\end{definition}

There are several reasons why we emphasize on {\it weak} tilings. For a collection $\mathcal{K}$ of subsets of $\mathbb{R}^d$ to be a tiling one commonly assumes that
\begin{itemize}
\item each $K \in \mathcal{K}$ is the closure of its interior,
\item $\mathcal{K}$ contains only finitely many different elements up to translation, and
\item the $d$-dimensional Lebesgue measure of $\partial K$ is zero for each $K\in \mathcal{K}$.
\end{itemize}
In our setting, namely for the collection $\{\mathcal{T}_\mathbf{r}(\mathbf{x})  \;:\; \mathbf{x} \in \mathbb{Z}^d\}$, we already saw in Example~\ref{ex:pp} that there exist elements having no inner points. Moreover, for some parameters we get infinitely many different shapes of the tiles $\mathcal{T}_\mathbf{r}(\mathbf{x})$, so that we do not have finiteness up to translation (this is the case for instance for the SRS parameter $r=-\frac23$ associated with the $\frac32$-number system, see Example~\ref{ex32_2}). Finally, in general, there seems to exist no known proof for the fact that the boundary of $\mathcal{T}_\mathbf{r}(\mathbf{x})$ has measure zero (although we conjecture this to be true).

We start with a tiling result that is contained in \cite[Theorem~4.6]{BSSST2011}.

\begin{theorem}\label{thm:multitiling}
Let $\mathbf{r}=(r_0,\ldots,r_{d-1}) \in \E_d$ with $r_0\not=0$ be given and assume that $\mathbf{r}$ satisfies one of the following conditions.
\begin{itemize}
\item $\mathbf{r} \in \mathbb{Q}^d$, or
\item $(X-\beta)(X^d + r_{d-1}X^{d-1} + \cdots + r_0) \in\mathbb{Z}[X]$ for some $\beta > 1$, or
\item $r_0,\ldots, r_{d-1}$ are algebraically independent over $\mathbb{Q}$.
\end{itemize}
Then the collection $\{\mathcal{T}_\mathbf{r}(\mathbf{x})  \;:\; \mathbf{x} \in \mathbb{Z}^d\}$ is a weak $m$-tiling for some $m\in \mathbb{N}$.
\end{theorem}

If $\mathcal{K}$ is a covering of degree $m$, then an {\it $m$-exclusive point} is a point that has a neighborhood $U$ such that each $\mathbf{x}\in U$ is covered by exactly $m$ elements of $\mathcal{K}$. The proof of Theorem~\ref{thm:multitiling} is technical and deals with the construction of a dense set of $m$-exclusive points. To prove that a given parameter satisfying the conditions of Theorem~\ref{thm:multitiling} actually induces a weak tiling it is obviously sufficient to exhibit a single $1$-exclusive point. For a given example this can often be done algorithmically. If $\mathbf{r}\in\D_d^{(0)}$,  Corollary~\ref{cor:exc} and Theorem~\ref{thm:multitiling} can be combined to the following tiling result.

\begin{corollary}[{see \cite[Corollary~4.7]{BSSST2011}}]\label{cor:tiling}
Let $\mathbf{r}=(r_0,\ldots,r_{d-1}) \in \D_d^{(0)} \cap \E_d$ with $r_0\not=0$. If $\mathbf{r}$ satisfies one of the three items listed in the statement of Theorem~\ref{thm:multitiling}, then  the collection $\{\mathcal{T}_\mathbf{r}(\mathbf{x})  \;:\; \mathbf{x} \in \mathbb{Z}^d\}$ is a weak tiling.
\end{corollary}

In \cite{ST:11}, for rational vectors $\mathbf{r}$ a tiling result without restrictions was established.

\begin{theorem}\label{p:tiling}
Let $\mathbf{r}=(r_0,\ldots, r_{d-1})\in \E_d$ have rational coordinates and assume that $r_0\not=0$. Then   $\{\mathcal{T}_\mathbf{r}(\mathbf{x})  \;:\; \mathbf{x} \in \mathbb{Z}^d\}$ is a weak tiling of $\mathbb{R}^d$.
\end{theorem}

The proof of this theorem is quite elaborate. Extending a theorem of Lagarias and Wang~\cite{Lagarias-Wang:97}, in \cite{ST:11} a tiling theorem for so-called {\it rational self-affine} tiles is proved. These tiles are defined as subsets of $\mathbb{R}^d \times \prod_{\mathfrak{p}}K_\mathfrak{p}$, where $K_\mathfrak{p}$ are completions of a number field $K$ that is defined in terms of the roots of the characteristic polynomial of $R(\mathbf{r})$. As the intersection of these tiles with the ``Euclidean part'' $\mathbb{R}^d \times \prod_{\mathfrak{p}}\{0\}$ of the representation space turn out to be SRS tiles corresponding to rational parameters, this tiling theorem can be used to prove Theorem~\ref{p:tiling}.

In the one-dimensional case the situation becomes much easier and we get the following result (here we identify the vector $(r)$ with the scalar $r$; see \cite[Theorem 4.9 and its proof]{BSSST2011}).

\begin{theorem}\label{1dimtiling}
Let $r \in \E_1\setminus \{0\}$. Then  $\{\mathcal{T}_{r}({x})  \;:\; {x} \in \mathbb{Z}\}$   is a tiling whose elements are intervals. In particular,
\[
\bigcup_{x \in \mathbb{Z}} \mathcal{T}_r(x) = \mathbb{R} \quad\hbox{with}\quad \#  (\mathcal{T}_r(x) \cap  \mathcal{T}_r(x')) \in \{0,1\} \hbox{ for distinct }x,x'\in\mathbb{Z}.
\]
\end{theorem}

\begin{proof}
We confine ourselves to $r>0$ (the case $r<0$ can be treated similar). Choose $x,y\in \mathbb{Z}$ with $x_0 < y_0$. By the definition of $\tau_r$ we get that $-x_1<-y_1$ for all $x_1 \in \tau_r^{-1}(x_0)$ and all $y_1 \in \tau_r^{-1}(y_0)$. Iterating this for $k$ times and multiplying by $R(r)^k=(-r)^k$ we obtain that
\[
x \in R(r)^k\tau_r^{-k}(x_0), \quad  y \in R(r)^k\tau_r^{-k}(y_0) \quad\hbox{implies that}\quad x < y.
\]
Taking the Hausdorff limit for $k \to \infty$, the result follows by the definition of SRS tiles and taking into account the fact that $\{\mathcal{T}_{r}({x})  \;:\; {x} \in \mathbb{Z}\}$ is a covering of $\mathbb{R}$ by Proposition~\ref{prop:basicSRS}.
\end{proof}

There are natural questions related to the results of this subsection. Although it seems to be unknown whether the collection $\{\mathcal{T}_\mathbf{r}(\mathbf{x})  \;:\; \mathbf{x} \in \mathbb{Z}^d\}$ forms a weak $m$-tiling for some $m$ for each $\mathbf{r}\in\E_d$ we conjecture the following stronger result (which also contains the {\it Pisot conjecture} for beta-tiles, see {\it e.g.}~\cite[Section~7]{BBK:06}).

\begin{conjecture}
Let $\mathbf{r}=(r_0,\ldots,r_{d-1}) \in \mathcal{E}_d$ with $r_0\not=0$. Then $\{\mathcal{T}_\mathbf{r}(\mathbf{x})  \;:\; \mathbf{x} \in \mathbb{Z}^d\}$ is a weak tiling of $\mathbb{R}^d$.
\end{conjecture}

Moreover, we state the following conjecture on the boundary of SRS tiles.

\begin{conjecture}
Let $\mathbf{r}=(r_0,\ldots,r_{d-1}) \in \mathcal{E}_d$ with $r_0\not=0$. Then the $d$-dimensional Lebesgue measure of $\partial \mathcal{T}_\mathbf{r}(\mathbf{x})$ is zero for each $\mathbf{x} \in \mathbb{Z}^d$.
\end{conjecture}

Finally, we state a problem related to the connectivity of  central SRS tiles (see also~\cite[Section~7]{BSSST2011}).
For $d\in\mathbb{N}$ define the Mandelbrot set  
\[
\mathcal{M}_d = \{ \mathbf{r} \in \E_d \;:\; \mathcal{T}_\mathbf{r}(\mathbf{0}) \hbox{ is connected}\}.
\]
It is an easy consequence of Theorem~\ref{1dimtiling} that $\mathcal{M}_1=(-1,1)$. However, we do not know anything about $\mathcal{M}_d$ in higher dimensions.

\begin{problem}
Describe the Mandelbrot sets $\mathcal{M}_d$ for $d\ge 2$.
\end{problem}

\subsection{SRS tiles and their relations to beta-tiles and self-affine tiles}

Let $\beta$ be a Pisot number and write the minimal polynomial of $\beta$ as
\[ 
(X-\beta) (X^d + r_{d-1} X^{d-1} + \cdots + r_0) \in \mathbb{Z}[X]. 
\]
Let $\mathbf{r} = (r_0,\ldots,r_{d-1})$. Then, for every $\mathbf{x} \in \mathbb{Z}^d$, the {\it SRS tile associated with $\beta$} is the set $$\mathcal{T}_\mathbf{r}(\mathbf{x}) = \lim_{n\to\infty} R(\mathbf{r})^n \tau_\mathbf{r}^{-n}(\mathbf{x}),$$
with $R(\mathbf{r})$ as in \eqref{mata},
%
%
%

The conjugacy between $T_\beta$ and $\tau_\mathbf{r}$ proved in Proposition~\ref{prop:betanumformula} suggests that there is some relation between the SRS tiles $\mathcal{T}_\mathbf{r}(\mathbf{x})$, $\mathbf{x}\in \mathbb{Z} ^d$, and the tiles associated with beta-numeration (which have been studied extensively in the literature, see {\it e.g.} \cite{Akiyama:02,Rauzy:82}). We recall the definition of these ``beta-tiles''.

Let $\beta_1, \ldots, \beta_d$ be the Galois conjugates of~$\beta$, numbered in such a way that  $\beta_1, \ldots, \beta_r \in \mathbb{R}$, $\beta_{r+1} = \overline{\beta_{r+s+1}}, \ldots, \beta_{r+s} = \overline{\beta_{r+2s}} \in \mathbb{C}$, $d = r + 2s$. Let further
$x^{(j)}$ be the corresponding conjugate of $x \in \mathbb{Q}(\beta)$, $1 \le j \le d$, and
$ \Xi_\beta:\ \mathbb{Q}(\beta) \to \mathbb{R}^d$,
 be the map
  $$\ x \mapsto  \big(x^{(1)}, \ldots, x^{(r)}, \Re\big(x^{(r+1)}\big), \Im\big(x^{(r+1)}\big), \ldots, \Re\big(x^{(r+s)}\big), \Im\big(x^{(r+s)}\big)\big).$$

Then we have the following definition.

\begin{definition}[Beta-tile, see~\cite{Akiyama:02, BSSST2011,Thurston:89}]\label{def:betatile}
For $x \in \mathbb{Z}[\beta] \cap [0,1)$, the {\it beta-tile} is the (compact) set
\[ 
\mathcal{{R} }_\beta(x) = \lim_{n\to\infty} \Xi_\beta\big(\beta^n T_\beta^{-n}(x) \big). 
\]
The {\it {integral }}{beta-tile} is the (compact) set
\[ 
\mathcal{{S}}_\beta(x) = \lim_{n\to\infty} \Xi_\beta\big(\beta^n {\big(} T_\beta^{-n}(x) {{\cap \mathbb{Z}[\beta]}\big)}\big). 
\]
\end{definition}

With these definitions it holds that
$\mathbf{t} \in \mathcal{{R} }_\beta(x)$ if and only if there exist $c_k \in \mathbb{Z}$ with
\[ \mathbf{t} = \Xi_\beta(x) + \sum_{k=1}^\infty \Xi_\beta(\beta^{k-1} c_k),\ \frac{c_n}{\beta} + \cdots + \frac{c_1}{\beta^n} + \frac{x}{\beta^n} \in [0,1) \ \forall n {\ge 1},\]
and
$\mathbf{t} \in \mathcal{{S}}_\beta(x)$ if and only if there exist $c_k \in \mathbb{Z}$ with
\[ \mathbf{t} = \Xi_\beta(x) + \sum_{k=1}^\infty \Xi_\beta(\beta^{k-1} c_k),\ \frac{c_n}{\beta} + \cdots + \frac{c_1}{\beta^n} + \frac{x}{\beta^n} \in [0,1) \cap \mathbb{Z}[\beta]\ \forall n {\ge 1}.\]

Observe that the ``digits'' $c_k$ fulfill the greedy condition, compare \eqref{greedycondition}. The following result shows how SRS-tiles are related to integral beta-tiles by a linear transformation. 

\begin{theorem}[{compare \cite[Theorem~6.7]{BSSST2011}}]\label{thm:betatilecorrespondence}
Let $\beta$ be a Pisot number with minimal polynomial
$(X-\beta)(X^d+r_{d-1}X^{d-1}+\cdots+r_0)$ and $d=r+2s$ Galois
conjugates $\beta_1,\ldots,\beta_r\in\mathbb{R}$,
$\beta_{r+1},\ldots,\beta_{r+2s}\in\mathbb{C}\setminus\mathbb{R}$, ordered such that
$\beta_{r+1}=\overline{\beta_{r+s+1}},\,\ldots,\,\beta_{r+s}=\overline{\beta_{r+2s}}$.
Let
\[
X^d+r_{d-1}X^{d-1}+\cdots+r_0 =
(X-\beta_j)(X^{d-1}+q_{d-2}^{(j)}X^{d-2}+\cdots+q_0^{(j)})
\]
for $1\le j\le d$ and
\[
U=\left(\begin{array}{ccccc}
q^{(1)}_{0} & q^{(1)}_{1} & \cdots & q^{(1)}_{d-2} & 1 \\
\vdots & \vdots & &\vdots & \vdots \\
q^{(r)}_{0} & q^{(r)}_{1} & \cdots & q^{(r)}_{d-2} & 1 \\
\Re(q^{(r+1)}_{0}) & \Re(q^{(r+1)}_{1}) & \cdots & \Re(q^{(r+1)}_{d-2}) & 1 \\
\Im(q^{(r+1)}_{0}) & \Im(q^{(r+1)}_{1}) & \cdots & \Im(q^{(r+1)}_{d-2}) & 0 \\
\vdots & \vdots & &\vdots & \vdots \\
\Re(q^{(r+s)}_{0}) & \Re(q^{(r+s)}_{1}) & \cdots & \Re(q^{(r+s)}_{d-2}) & 1 \\
\Im(q^{(r+s)}_{0}) & \Im(q^{(r+s)}_{1}) & \cdots & \Im(q^{(r+s)}_{d-2}) & 0 \\
\end{array}\right) \in \mathbb{R}^{d \times d}.
\]
Then we have
\[
\mathcal{S}_\beta(\{\mathbf{r}\mathbf{x}\}) = U (R(\mathbf{r})-\beta I_d) \mathcal{T}_\mathbf{r}(\mathbf{x})
\]
for every $\mathbf{x} \in \mathbb{Z}^d$, where $\mathbf{r}=(r_0,\ldots,r_{d-1})$ and $I_d$ is the $d$-dimensional identity matrix.
\end{theorem}

We omit the technical proof that, obviously, makes use of the conjugacy in Proposition~\ref{prop:betanumformula} and refer the reader to \cite{BSSST2011}. One reason for the technical difficulties come from the fact that although the integral beta-tile associated with $\mathcal{T}_\mathbf{r}(\mathbf{x})$ is given by $\mathcal{S}_\beta(\{\mathbf{r} \mathbf{x}\}) = U (R(\mathbf{r}) - \beta I_d) \mathcal{T}_\mathbf{r}(\mathbf{x})$, its ``center'' is $\Xi_\beta(\{\mathbf{r}\mathbf{x}\}) = U (\tau_\mathbf{r}(\mathbf{x}) - \beta\mathbf{x}) = U (R(\mathbf{r}) - \beta I_d) \mathbf{x} + U (0,\ldots,0,\{\mathbf{r}\mathbf{x}\})^t$.

\begin{example}
Let $\beta_1$ be the Pisot unit given by the dominant root of $X^3-X^2-X-1$. According to Proposition~\ref{prop:betanumformula} the associated SRS parameter is $\mathbf{r} = (1/\beta_1,\beta_1-1)$. Using the algorithm based on Theorem~\ref{thm:Brunotte} one can easily show that $\mathbf{r} \in \D_2^{(0)}$. Thus Corollary~\ref{cor:tiling} implies that the collection $\{\mathcal{T}_\mathbf{r}(\mathbf{x})\;:\; \mathbf{x}\in \mathbb{Z}^d\}$ induces a tiling of $\mathbb{R}^2$.  On the right hand side of Figure~\ref{fig:betatiling} a patch of this tiling is depicted. In view of Theorem~\ref{thm:betatilecorrespondence} the beta-tiles associated with $\beta_1$ also form a tiling which can be obtained from the SRS tiling just by an affine transformation.

\begin{figure}
\includegraphics[width=.45\textwidth]{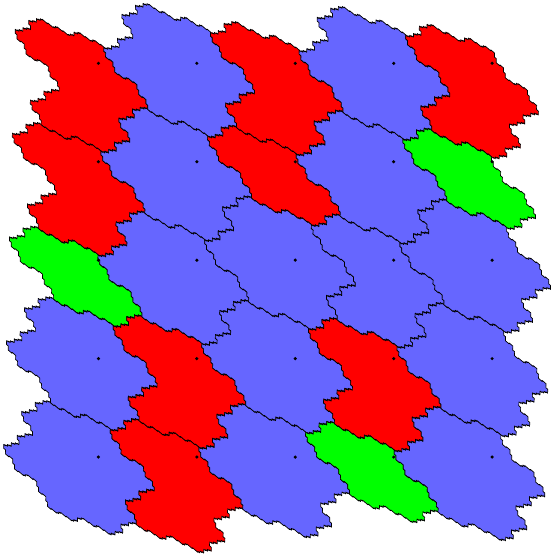} 
\includegraphics[width=.45\textwidth]{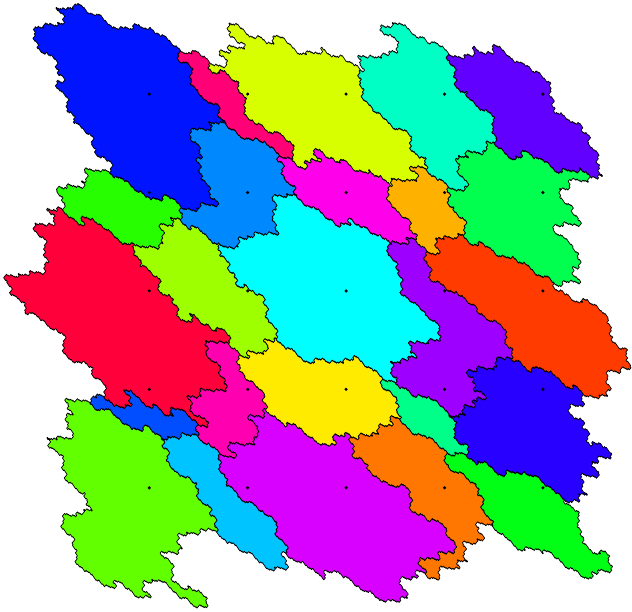}
\caption{Two patches of tilings induced by SRS tiles: the left figure shows the tiling associated with the parameter $\mathbf{r} = (1/\beta_1,\beta_1-1)$ where $\beta_1$ is the Pisot unit given by $\beta_1^3 = \beta_1^2 + \beta_1 + 1$. It is an affine image of the tiling induced by the (integral) beta-tiles associated with the Pisot unit $\beta_1$. The central tile is the classical {\it Rauzy fractal}. The right patch corresponds to the parameter $\mathbf{r} =  (2/\beta_2,\beta_2-2)$ where $\beta_2$ is the (non-unit) Pisot number given by $\beta_2^3 = 2\beta_2^2 + 2\beta_2 + 2$. It can also be regarded as (an affine transformation of) the tiling induced by the integral beta-tiles associated with $\beta_2$ . \label{fig:betatiling}}
\end{figure}

Similarly, let $\beta_2$ be the (non-unit) Pisot root of $X^3-2X^2-2X-2$.  Proposition~\ref{prop:betanumformula} yields the associated SRS parameter $\mathbf{r} =  (2/\beta_2,\beta_2-2)$. Again, one can check that $\mathbf{r} \in \D_2^{(0)}$, and the resulting tiling is depicted on the right hand side of Figure~\ref{fig:betatiling}. As $\beta_2$ is not a unit, the structure of this tiling turns out to be more involved.
\end{example}

\begin{proposition}[compare \cite {Akiyama:02, Sirvent-Wang:00a}]
If $\beta$ is a Pisot unit ($\beta^{-1} \in \mathbb{Z}[\beta]$), then
\begin{itemize}
\item[(i)] $\mathcal{R}_\beta(x) = \mathcal{S}_\beta(x)$ for every $x \in \mathbb{Z}[\beta] \cap [0,1)$,
\item[(ii)] we have only finitely many tiles up to translation,
\item[(iii)] the boundary of each tile has zero Lebesgue measure,
\item[(iv)] each tile is the closure of its interior,
\item[(v)] $\{\mathcal{S}_\beta(x) \,:\, x \in \mathbb{Z}[\beta] \cap [0,1)\}$ forms a multiple tiling of $\mathbb{R}^d$,
\item[(vi)] $\{\mathcal{S}_\beta(x)  \,:\, x \in \mathbb{Z}[\beta] \cap [0,1)\}$ forms a tiling if (F) holds,
\item[(vii)] $\{\mathcal{S}_\beta(x)  \,:\, x \in \mathbb{Z}[\beta] \cap [0,1)\}$ forms a tiling iff (W) holds: \\
for every $x \in \mathbb{Z}[\beta] \cap [0,1)$ and every $\varepsilon > 0$, there exists some $y \in [0,\varepsilon)$ with finite beta-expansion such that $x+y$ has finite beta-expansion.
\end{itemize}
\end{proposition}

\begin{proof}
Assertion~(i) is an immediate consequence of the definition, since $T_\beta^{-1}(\mathbb{Z}[\beta]) \subset \mathbb{Z}[1/\beta]=\mathbb{Z}[\beta]$ holds for a Pisot unit $\beta$. Assertion~(ii) is contained in \cite[Lemma~5]{Akiyama:02}. Assertion~(iii) is proved in \cite[Theorem~5.3.12]{BST:10} and Assertion~(iv) is contained in \cite[Theorem~4.1]{Sirvent-Wang:00a} (both of these results are stated in terms of substitutions rather than beta-expansions). The tiling properties are proved in \cite{Akiyama:02}; property (W) has been further studied {\it e.g.} in \cite{Akiyama-Rao-Steiner:04}.
\end{proof}

In the following we turn our attention to tiles associated with expanding polynomials. Here Proposition~\ref{p:CNSconjugacy} suggests a relation between certain SRS tiles and the self-affine tiles associated with expanding monic polynomials defined as follows.

\begin{definition}[Self-affine tile, {\it cf.}~\cite{Katai-Koernyei:92}]
Let $A(X) = X^d + a_{d-1} X^{d-1} + \cdots + a_0 \in \mathbb{Z}[X]$ be an expanding polynomial and $B$ the transposed companion matrix with characteristic polynomial~$A$.
\[ \mathcal{F} := \left\{\mathbf{t} \in \mathbb{R}^d\; :\; \mathbf{t} = \sum_{i=0}^\infty B^{-i} (c_i,0,\ldots,0)^t, c_i \in \mathcal{N} \right\} \]
($\mathcal{N}=\{0,\ldots,|a_0|-1\}$) is called the {\it self-affine tile associated with~$A$}.
\end{definition}

For this class of tiles the following properties hold.

\begin{proposition}[\cite{Katai-Koernyei:92,Lagarias-Wang:97,Wang:00a}] 
\mbox{}
\begin{itemize}
\item[(i)] $\mathcal{F}$ is compact and self-affine.
\item[(ii)] $\mathcal{F}$ is the closure of its interior.
\item[(iii)] $\{\mathbf{x} + \mathcal{F}  \,:\,\mathbf{x} \in \mathbb{Z}^d\}$ induces a (multiple) tiling of $\mathbb{R}^d$. If $A$ is irreducible $\{\mathbf{x} + \mathcal{F}  \,:\, \mathbf{x} \in \mathbb{Z}^d\}$ forms a tiling of $\mathbb{R}^d$.
\end{itemize}
\end{proposition}

\begin{proof}
Assertion (i) is an immediate consequence of the definition of $\mathcal{F}$, see {\it e.g.}~\cite{Katai-Koernyei:92}. In particular, note that $\mathcal{F}=\bigcup_{c \in \mathcal{N}} B^{-1}(\mathcal{F}+(c,0,\ldots,0)^t)$
where $B=VR(\mathbf{r})^{-1}V^{-1}$ with
\begin{equation}\label{rV}
\mathbf{r} = \left(\frac{1}{a_0}, \frac{a_{d-1}}{a_0}, \ldots, \frac{a_1}{a_0}\right) \quad\hbox{and}\quad V =  \left(\begin{array}{cccc}
1 & a_{d-1} & \cdots &  a_1 \\
0 & \ddots  & \ddots & \vdots \\
\vdots &  \ddots & \ddots & a_{d-1} \\
0 & \cdots & 0 & 1
\end{array}\right).
\end{equation}
To prove Assertion~(ii) one first shows that $\mathcal{F}+\mathbb{Z}^d$ forms a covering of $\mathbb{R}^d$. From this fact a Baire type argument yields that ${\rm int}(\mathcal{F})\not=\emptyset$. Using this fact the self-affinity of $ \mathcal{F}$ yields~(ii), see \cite[Theorem~2.1]{Wang:00a}. In assertion~(iii) the multiple tiling property is fairly easy to prove, the tiling property is hard to establish and was shown in~\cite{Lagarias-Wang:97} in a more general context.
\end{proof}

There is a close relation between the tile $\mathcal{F}$ and the central SRS-tile studied above.

\begin{theorem}\label{thm:cnssrstile}
Let $A(X) = X^d + a_{d-1} X^{d-1} + \cdots + a_0 \in \mathbb{Z}[X]$ be an expanding polynomial.
For all $\mathbf{x} \in \mathbb{Z}^{d}$ we have
\begin{align*}
\mathcal{F} & = V \mathcal{T}_\mathbf{r}(\mathbf{0}), \\
\mathbf{x} + \mathcal{F} & = V \mathcal{T}_\mathbf{r}(V^{-1}(\mathbf{x}))
\end{align*}
where $V$ is given in \eqref{rV}.
\end{theorem}

The result follows immediately from \cite[Corollary~5.14]{BSSST2011}.

\begin{example}
Continuing Example~\ref{exKnuth}, let $X^2 + 2X + 2$ be given. The self-affine tile associated with this polynomial is Knuth's famous twin dragon. In view of Proposition~\ref{p:CNSconjugacy}, the associated SRS parameter is $\mathbf{r}=(\frac12,1)$. Using Theorem~\ref{thm:Brunotte} we see that $\mathbf{r}\in\D_2^{(0)}$, and Corollary~\ref{cor:tiling} can be invoked to show that the associated SRS tiles induce a tiling (see the left side of Figure~\ref{fig:CNStiling}). According to Theorem~\ref{thm:cnssrstile} this tiling is an affine transformation of the tiling $\mathcal{F}+\mathbb{Z}^d$.

\begin{figure}
\includegraphics[width=.45\textwidth]{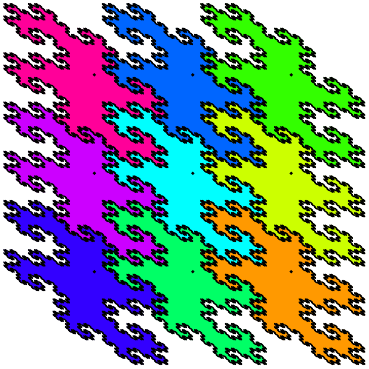} 
\includegraphics[width=.45\textwidth]{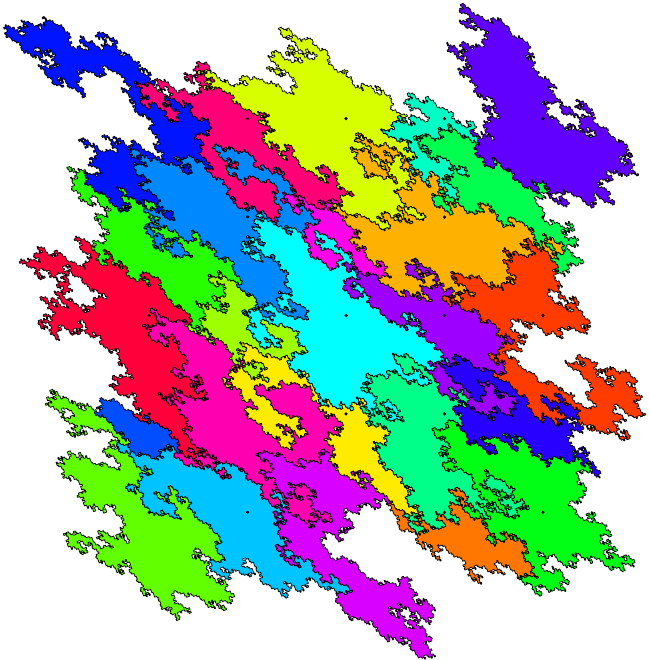} 
\caption{Two patches of tiles induced by SRS tiles: the left figure shows the tiling associated with the parameter $\mathbf{r} = (1/2,1)$. It is an affine image of the tiling induced by the CNS defined by $X^2 + 2X + 2$. The central tile is Knuth's twin dragon. The right patch corresponds to the parameter $\mathbf{r} = (2/3,1)$. It can also be regarded as the {\it Brunotte tiling} associated with the non-monic expanding polynomial $2X^2 + 3X + 3$. \label{fig:CNStiling}}
\end{figure}

Starting with the non-monic polynomial $2X^2 + 3X + 3$ we get $\mathbf{r}=(\frac23,1)$ and the tiling depicted on the right side of Figure~\ref{fig:CNStiling}. We mention that also in the case of non-monic polynomials we have a tiling theory. These so-called {\it Brunotte tiles} are defined and discussed in \cite[Section~5]{BSSST2011}.
\end{example}

We conclude this section with a continuation of Example~\ref{ex32}.

\begin{example}[The $\frac32$-number system, continued]\label{ex32_2}
As shown in Example~\ref{ex32} the $\frac32$-number system defined in~\cite{Akiyama-Frougny-Sakarovitch:07} is conjugate to the SRS $\tau_{-2/3}$. It has been shown in \cite[Section~5.4]{BSSST2011} that the tiling (see~Theorem~\ref{1dimtiling}) induced by the associated SRS tiles
\begin{equation}\label{32-tiling}
\{\mathcal{T}_{-2/3}(x)\;:\; x\in \mathbb{Z} \}
\end{equation}
consists of (possibly degenerate) intervals with infinitely many different lengths. Essentially, this is due to the fact that for each $k\in\mathbb{N}$ we can find $N_k\in \mathbb{Z}$ such that $\# \tau_{-2/3}^{-k}(N_k) =2$ (see \cite[Lemma~5.18]{BSSST2011}). This can be used to show that the length $\ell_k$ of the interval $\mathcal{T}_{-2/3}(N_k)$ satisfies $\ell_k\in[(\frac23)^k, 3(\frac23)^k]$, which immediately yields the existence of intervals of infinitely many different lengths in \eqref{32-tiling}.

It was observed in \cite[Example~2.1]{ST:11} that the length $\ell$ of the interval $\mathcal{T}_{-2/3}(0)$ which is equal to $\ell=1.6227\cdots$ is related to a solution of some case of the Josephus problem presented in~\cite{OW:91}. 
\end{example}

\section{Variants of shift radix systems}\label{sec:variants}

In the recent years some variants of SRS have been studied. Akiyama and Scheicher~\cite{Akiyama-Scheicher:07} investigated ``symmetric'' SRS. They differ from the ordinary ones just by replacing $-\lfloor \mathbf{rz} \rfloor$ by $-\lfloor \mathbf{rz} + \frac12 \rfloor$, {\it i.e.},
the {\it symmetric SRS} $\hat\tau_\mathbf{r}: \mathbb{Z}^d \to \mathbb{Z}^d$ is defined by
$$
\hat\tau_{{\bf r}}({\bf z})=\left(z_1,\dots,z_{d-1},-\left\lfloor {\bf r} {\bf z}+\frac12\right\rfloor\right)^t \qquad ({\bf z}=(z_0,\dots,z_{d-1})^t).
$$
It turns out that the characterization of the (accordingly defined) finiteness property is easier in this case and complete results have been achieved for dimension two (see~\cite{Akiyama-Scheicher:07}) and three (see~\cite{Huszti-Scheicher-Surer-Thuswaldner:06}). As mentioned above, for SRS it is conjectured that SRS tiles always induce weak tilings. Interestingly, this is not true for tiles associated with symmetric SRS. Kalle and Steiner~\cite{Kalle-Steiner} found a parameter $\mathbf{r}$ (related to a Pisot unit) where the associated symmetric SRS tiles form a double tiling (in this paper this is studied in the world of symmetric beta-expansions; these are known to be a special case of symmetric SRS, see \cite{Akiyama-Scheicher:07}). Further generalizations of SRS are studied by Surer~\cite{Surer:09}. Analogs for finite fields have been introduced by Scheicher~\cite{Scheicher:07}. Brunotte {\it et al.}~\cite{Brunotte-Kirschenhofer-Thuswaldner:12a} define SRS for Gaussian integers. In this case the characterization problem for the finiteness property is non-trivial already in dimension one.

\bigskip

{\bf Acknowledgements.} The authors wish to express their thanks to the scientific committees of the international conferences ``Num\'eration 2011''  in Li\`ege, Belgium and ``Numeration and substitution: 2012'' in Kyoto, Japan for inviting them to give expository lectures on shift radix systems at these conferences. Moreover they are grateful to Shigeki Akiyama for inviting them to write this survey for the present RIMS conference series volume. They also thank Wolfgang Steiner for his help; he generated Figures~\ref{fig:steiner4},~\ref{fig:betatiling}, and~\ref{fig:CNStiling}.
\bibliographystyle{siam}
\bibliography{SRSKyoto}

\begin{thebibliography}{10}

\bibitem{AdlerKitchensTresser:01}
{\sc R.~Adler, B.~Kitchens, and C.~Tresser}, {\em Dynamics of non-ergodic
  piecewise affine maps of the torus}, Ergodic Theory Dynam. Systems, 21
  (2001), pp.~959--999.

\bibitem{Akiyama:00}
{\sc S.~Akiyama}, {\em Cubic {P}isot units with finite beta expansions}, in
  Algebraic number theory and Diophantine analysis (Graz, 1998), de Gruyter,
  Berlin, 2000, pp.~11--26.

\bibitem{Akiyama:02}
\leavevmode\vrule height 2pt depth -1.6pt width 23pt, {\em On the boundary of
  self affine tilings generated by {P}isot numbers}, J. Math. Soc. Japan, 54
  (2002), pp.~283--308.

\bibitem{Akiyama-Borbeli-Brunotte-Pethoe-Thuswaldner:05}
{\sc S.~Akiyama, T.~Borb{\'e}ly, H.~Brunotte, A.~Peth{\H{o}}, and J.~M.
  Thuswaldner}, {\em Generalized radix representations and dynamical systems.
  {I}}, Acta Math. Hungar., 108 (2005), pp.~207--238.

\bibitem{Akiyama-Borbely-Brunotte-Pethoe-Thuswaldner:06}
\leavevmode\vrule height 2pt depth -1.6pt width 23pt, {\em Basic properties of
  shift radix systems}, Acta Math. Acad. Paedagog. Nyh\'azi. (N.S.), 22 (2006),
  pp.~19--25 (electronic).

\bibitem{Akiyama-Brunotte-Pethoe-Steiner:06}
{\sc S.~Akiyama, H.~Brunotte, A.~Peth{\H{o}}, and W.~Steiner}, {\em Remarks on
  a conjecture on certain integer sequences}, Period. Math. Hungar., 52 (2006),
  pp.~1--17.

\bibitem{Akiyama-Brunotte-Pethoe-Steiner:07}
\leavevmode\vrule height 2pt depth -1.6pt width 23pt, {\em Periodicity of
  certain piecewise affine planar maps}, Tsukuba J. Math., 32 (2008),
  pp.~197--251.

\bibitem{Akiyama-Brunotte-Pethoe-Thuswaldner:06}
{\sc S.~Akiyama, H.~Brunotte, A.~Peth{\H{o}}, and J.~M. Thuswaldner}, {\em
  Generalized radix representations and dynamical systems. {II}}, Acta Arith.,
  121 (2006), pp.~21--61.

\bibitem{Akiyama-Brunotte-Pethoe-Thuswaldner:07}
\leavevmode\vrule height 2pt depth -1.6pt width 23pt, {\em Generalized radix
  representations and dynamical systems. {III}}, Osaka J. Math., 45 (2008),
  pp.~347--374.

\bibitem{Akiyama-Frougny-Sakarovitch:07}
{\sc S.~Akiyama, C.~Frougny, and J.~Sakarovitch}, {\em Powers of rationals
  modulo 1 and rational base number systems}, Israel J. Math., 168 (2008),
  pp.~53--91.

\bibitem{AH:13}
{\sc S.~Akiyama and E.~Harriss}, {\em Pentagonal domain exchange}, submitted.

\bibitem{Akiyama-Pethoe:02}
{\sc S.~Akiyama and A.~Peth\H{o}}, {\em On canonical number systems}, Theor.
  Comput. Sci., 270 (2002), pp.~921--933.

\bibitem{Akiyama-Pethoe:13}
{\sc S.~Akiyama and A.~Peth{\H{o}}}, {\em Discretized rotation has infinitely
  many periodic orbits}, Nonlinearity, 26 (2013), pp.~871--880.

\bibitem{Akiyama-Rao-Steiner:04}
{\sc S.~Akiyama, H.~Rao, and W.~Steiner}, {\em A certain finiteness property of
  {P}isot number systems}, J. Number Theory, 107 (2004), pp.~135--160.

\bibitem{Akiyama-Scheicher:07}
{\sc S.~Akiyama and K.~Scheicher}, {\em Symmetric shift radix systems and
  finite expansions}, Math. Pannon., 18 (2007), pp.~101--124.

\bibitem{Akiyama-Thuswaldner:00}
{\sc S.~Akiyama and J.~M. Thuswaldner}, {\em Topological properties of
  two-dimensional number systems}, J. Theor. Nomb. Bordx., 12 (2000),
  pp.~69--79.

\bibitem{Arnoux-Ito:01}
{\sc P.~Arnoux and S.~Ito}, {\em Pisot substitutions and {R}auzy fractals},
  Bull. Belg. Math. Soc. Simon Stevin, 8 (2001), pp.~181--207.
\newblock Journ\'ees Montoises d'Informatique Th\'eorique (Marne-la-Vall\'ee,
  2000).

\bibitem{BBK:06}
{\sc V.~Baker, M.~Barge, and J.~Kwapisz}, {\em Geometric realization and
  coincidence for reducible non-unimodular {P}isot tiling spaces with an
  application to {$\beta$}-shifts}, Ann. Inst. Fourier (Grenoble), 56 (2006),
  pp.~2213--2248.
\newblock Num{\'e}ration, pavages, substitutions.

\bibitem{Barat-Berthe-Liardet-Thuswaldner:06}
{\sc G.~Barat, V.~Berth{\'e}, P.~Liardet, and J.~Thuswaldner}, {\em Dynamical
  directions in numeration}, Ann. Inst. Fourier (Grenoble), 56 (2006),
  pp.~1987--2092.
\newblock Num\'eration, pavages, substitutions.

\bibitem{Berthe-Siegel:05}
{\sc V.~Berth{\'e} and A.~Siegel}, {\em Tilings associated with beta-numeration
  and substitutions}, Integers, 5 (2005), pp.~A2, 46 pp. (electronic).

\bibitem{BSSST2011}
{\sc V.~Berth{\'e}, A.~Siegel, W.~Steiner, P.~Surer, and J.~M. Thuswaldner},
  {\em Fractal tiles associated with shift radix systems}, Adv. Math., 226
  (2011), pp.~139--175.

\bibitem{BST:10}
{\sc V.~Berth{\'e}, A.~Siegel, and J.~Thuswaldner}, {\em Substitutions, {R}auzy
  fractals and tilings}, in Combinatorics, automata and number theory, vol.~135
  of Encyclopedia Math. Appl., Cambridge Univ. Press, Cambridge, 2010,
  pp.~248--323.

\bibitem{Bertrand:77}
{\sc A.~Bertrand}, {\em D\'eveloppements en base de {P}isot et r\'epartition
  modulo {$1$}}, C. R. Acad. Sci. Paris S\'er. A-B, 285 (1977), pp.~A419--A421.

\bibitem{Bosio-Vivaldi:00}
{\sc D.~Bosio and F.~Vivaldi}, {\em Round-off errors and {$p$}-adic numbers},
  Nonlinearity, 13 (2000), pp.~309--322.

\bibitem{Boyd:89}
{\sc D.~W. Boyd}, {\em Salem numbers of degree four have periodic expansions},
  in Th\'eorie des nombres ({Q}uebec, {PQ}, 1987), de Gruyter, Berlin, 1989,
  pp.~57--64.

\bibitem{Boyd:96}
\leavevmode\vrule height 2pt depth -1.6pt width 23pt, {\em On the beta
  expansion for {S}alem numbers of degree {$6$}}, Math. Comp., 65 (1996),
  pp.~861--875, $S$29--$S$31.

\bibitem{Boyd:97}
\leavevmode\vrule height 2pt depth -1.6pt width 23pt, {\em The beta expansion
  for {S}alem numbers}, in Organic mathematics ({B}urnaby, {BC}, 1995), vol.~20
  of CMS Conf. Proc., Amer. Math. Soc., Providence, RI, 1997, pp.~117--131.

\bibitem{Brand:64}
{\sc L.~Brand}, {\em The companion matrix and its properties}, Amer. Math.
  Monthly, 71 (1964), pp.~629--634.

\bibitem{Bruin-Lambert-Poggiaspalla-Vaienti:03}
{\sc H.~Bruin, A.~Lambert, G.~Poggiaspalla, and S.~Vaienti}, {\em Numerical
  analysis for a discontinuous rotation of the torus}, Chaos, 13 (2003),
  pp.~558--571.

\bibitem{Brunotte:01}
{\sc H.~Brunotte}, {\em On trinomial bases of radix representations of
  algebraic integers}, Acta Sci. Math. (Szeged), 67 (2001), pp.~521--527.

\bibitem{Brunotte-Kirschenhofer-Thuswaldner:12a}
{\sc H.~Brunotte, P.~Kirschenhofer, and J.~M. Thuswaldner}, {\em Shift radix
  systems for {G}aussian integers and {P}eth{\H o}'s loudspeaker}, Publ. Math.
  Debrecen, 79 (2011), pp.~341--356.

\bibitem{Brunotte-Kirschenhofer-Thuswaldner:12}
\leavevmode\vrule height 2pt depth -1.6pt width 23pt, {\em Contractivity of
  three dimensional shift radix systems with finiteness property}, J. Differ.
  Equations Appl., 18 (2012), pp.~1077--1099.

\bibitem{Burcsi-Kovacs:08}
{\sc P.~Burcsi and A.~Kov{\'a}cs}, {\em Exhaustive search methods for {CNS}
  polynomials}, Monatsh. Math., 155 (2008), pp.~421--430.

\bibitem{Fam-Meditch:78}
{\sc A.~T. Fam and J.~S. Meditch}, {\em A canonical parameter space for linear
  systems design}, IEEE Trans. Autom. Control, 23 (1978), pp.~454--458.

\bibitem{Frougny-Solomyak:92}
{\sc C.~Frougny and B.~Solomyak}, {\em Finite beta-expansions}, Ergodic Theory
  Dynam. Systems, 12 (1992), pp.~713--723.

\bibitem{Gilbert:81}
{\sc W.~J. Gilbert}, {\em Radix representations of quadratic fields}, J. Math.
  Anal. Appl., 83 (1981), pp.~264--274.

\bibitem{Hare-Tweedle:08}
{\sc K.~G. Hare and D.~Tweedle}, {\em Beta-expansions for infinite families of
  {P}isot and {S}alem numbers}, J. Number Theory, 128 (2008), pp.~2756--2765.

\bibitem{Hollander:96}
{\sc M.~Hollander}, {\em Linear Numeration Systems, Finite Beta Expansions, and
  Discrete Spectrum of Substitution Dynamical Systems}, {P}h{D}. thesis,
  Washington University, Seattle, 1996.

\bibitem{Huszti-Scheicher-Surer-Thuswaldner:06}
{\sc A.~Huszti, K.~Scheicher, P.~Surer, and J.~M. Thuswaldner}, {\em
  Three-dimensional symmetric shift radix systems}, Acta Arith., 129 (2007),
  pp.~147--166.

\bibitem{Ito-Rao:06}
{\sc S.~Ito and H.~Rao}, {\em Atomic surfaces, tilings and coincidence. {I}.
  {I}rreducible case}, Israel J. Math., 153 (2006), pp.~129--155.

\bibitem{Johnson:75}
{\sc D.~B. Johnson}, {\em Finding all the elementary circuits of a directed
  graph}, SIAM J. Comput., 4 (1975), pp.~77--84.

\bibitem{Kalle-Steiner}
{\sc C.~Kalle and W.~Steiner}, {\em Beta-expansions, natural extensions and
  multiple tilings associated with {P}isot units}, Trans. Amer. Math. Soc., 364
  (2012), pp.~2281--2318.

\bibitem{Katai-Koernyei:92}
{\sc I.~K{\'a}tai and I.~K\H{o}rnyei}, {\em On number systems in algebraic
  number fields}, Publ. Math. Debrecen, 41 (1992), pp.~289--294.

\bibitem{Katai-Kovacs:80}
{\sc I.~K{\'a}tai and B.~Kov{\'a}cs}, {\em Kanonische {Z}ahlensysteme in der
  {T}heorie der {Q}uadratischen {Z}ahlen}, Acta Sci. Math. (Szeged), 42 (1980),
  pp.~99--107.

\bibitem{Katai-Kovacs:81}
\leavevmode\vrule height 2pt depth -1.6pt width 23pt, {\em Canonical number
  systems in imaginary quadratic fields}, Acta Math. Hungar., 37 (1981),
  pp.~159--164.

\bibitem{Katai-Szabo:76}
{\sc I.~K{\'a}tai and J.~Szab{\'o}}, {\em Canonical number systems for complex
  integers}, Acta Sci. Math. (Szeged), 37 (1975), pp.~255--260.

\bibitem{Kirschenhofer-Pethoe-Surer-Thuswaldner:10}
{\sc P.~Kirschenhofer, A.~Peth{\H{o}}, P.~Surer, and J.~M. Thuswaldner}, {\em
  Finite and periodic orbits of shift radix systems}, J. Th\'eor. Nombres
  Bordeaux, 22 (2010), pp.~421--448.

\bibitem{Kirschenhofer-Pethoe-Thuswaldner:08}
{\sc P.~Kirschenhofer, A.~Peth{\H{o}}, and J.~M. Thuswaldner}, {\em On a family
  of three term nonlinear integer recurrences}, Int. J. Number Theory, 4
  (2008), pp.~135--146.

\bibitem{Knuth:60}
{\sc D.~E. Knuth}, {\em An imaginary number system}, Comm. ACM, 3 (1960),
  pp.~245--247.

\bibitem{Kouptsov-Lowenstein-Vivaldi:02}
{\sc K.~L. Kouptsov, J.~H. Lowenstein, and F.~Vivaldi}, {\em Quadratic rational
  rotations of the torus and dual lattice maps}, Nonlinearity, 15 (2002),
  pp.~1795--1842.

\bibitem{Kovacs:81a}
{\sc B.~Kov{\'a}cs}, {\em Canonical number systems in algebraic number fields},
  Acta Math. Hungar., 37 (1981), pp.~405--407.

\bibitem{Kovacs-Pethoe:91}
{\sc B.~Kov{\'a}cs and A.~Peth\H{o}}, {\em Number systems in integral domains,
  especially in orders of algebraic number fields}, Acta Sci. Math. (Szeged),
  55 (1991), pp.~286--299.

\bibitem{Lagarias-Wang:96a}
{\sc J.~Lagarias and Y.~Wang}, {\em Self-affine tiles in $\mathbb{R}^n$}, Adv.
  Math., 121 (1996), pp.~21--49.

\bibitem{Lagarias-Wang:97}
\leavevmode\vrule height 2pt depth -1.6pt width 23pt, {\em Integral self-affine
  tiles in $\mathbb{R}^n$ {II}. lattice tilings}, J. Fourier Anal. Appl., 3
  (1997), pp.~83--102.

\bibitem{Lowensteinetal:97}
{\sc J.~Lowenstein, S.~Hatjispyros, and F.~Vivaldi}, {\em Quasi-periodicity,
  global stability and scaling in a model of {H}amiltonian round-off}, Chaos, 7
  (1997), pp.~49--66.

\bibitem{Lowenstein-Vivaldi:98}
{\sc J.~H. Lowenstein and F.~Vivaldi}, {\em Anomalous transport in a model of
  {H}amiltonian round-off}, Nonlinearity, 11 (1998), pp.~1321--1350.

\bibitem{Mahler:68}
{\sc K.~Mahler}, {\em An unsolved problem on the powers of {$3/2$}}, J.
  Austral. Math. Soc., 8 (1968), pp.~313--321.

\bibitem{OW:91}
{\sc A.~M. Odlyzko and H.~S. Wilf}, {\em Functional iteration and the
  {J}osephus problem}, Glasgow Math. J., 33 (1991), pp.~235--240.

\bibitem{Parry:60}
{\sc W.~Parry}, {\em On the {$\beta $}-expansions of real numbers}, Acta Math.
  Acad. Sci. Hungar., 11 (1960), pp.~401--416.

\bibitem{Penney:65}
{\sc W.~Penney}, {\em A ``binary'' system for complex numbers}, J. ACM, 12
  (1965), pp.~247--248.

\bibitem{Pethoe:91}
{\sc A.~Peth{\H{o}}}, {\em On a polynomial transformation and its application
  to the construction of a public key cryptosystem}, in Computational number
  theory ({D}ebrecen, 1989), de Gruyter, Berlin, 1991, pp.~31--43.

\bibitem{Pethoe:09}
\leavevmode\vrule height 2pt depth -1.6pt width 23pt, {\em On the boundary of
  the closure of the set of contractive polynomials}, Integers, 9 (2009),
  pp.~311 -- 325.

\bibitem{Rauzy:82}
{\sc G.~Rauzy}, {\em Nombres alg\'ebriques et substitutions}, Bull. Soc. Math.
  France, 110 (1982), pp.~147--178.

\bibitem{Reeve-Black-Vivaldi:13}
{\sc H.~Reeve-Black and F.~Vivaldi}, {\em Near-integrable behaviour in a family
  of discretized rotations}, Nonlinearity, 26 (2013), pp.~1227--1270.

\bibitem{Renyi:57}
{\sc A.~R{\'e}nyi}, {\em Representations for real numbers and their ergodic
  properties}, Acta Math. Acad. Sci. Hungar, 8 (1957), pp.~477--493.

\bibitem{Rockett-Szusz:92}
{\sc A.~M. Rockett and P.~Sz{\"u}sz}, {\em Continued fractions}, World
  Scientific Publishing Co. Inc., River Edge, NJ, 1992.

\bibitem{Scheicher:07}
{\sc K.~Scheicher}, {\em {$\beta$}-expansions in algebraic function fields over
  finite fields}, Finite Fields Appl., 13 (2007), pp.~394--410.

\bibitem{Scheicher-Surer-Thuswaldner-vdWoestijne:08}
{\sc K.~Scheicher, P.~Surer, J.~M. Thuswaldner, and C.~van~de Woestijne}, {\em
  Generalised canonical number systems and digit systems}, preprint.

\bibitem{Scheicher-Thuswaldner:01}
{\sc K.~Scheicher and J.~M. Thuswaldner}, {\em Canonical number systems,
  counting automata and fractals}, Math. Proc. Cambridge Philos. Soc., 133
  (2002), pp.~163--182.

\bibitem{Scheicher-Thuswaldner:03}
{\sc K.~Scheicher and J.~M. Thuswaldner}, {\em On the characterization of
  canonical number systems}, Osaka J. Math., 41 (2004), pp.~327--351.

\bibitem{Schmidt:80}
{\sc K.~Schmidt}, {\em On periodic expansions of {P}isot numbers and {S}alem
  numbers}, Bull. London Math. Soc., 12 (1980), pp.~269--278.

\bibitem{Schur:18}
{\sc I.~Schur}, {\em {\"U}ber {P}otenzreihen, die im {I}nnern des
  {E}inheitskreises beschr\"ankt sind. {II}}, J. reine und angew. Math, 148
  (1918), pp.~122--145.

\bibitem{Sirvent-Wang:00a}
{\sc V.~F. Sirvent and Y.~Wang}, {\em Self-affine tiling via substitution
  dynamical systems and {R}auzy fractals}, Pacific J. Math., 206 (2002),
  pp.~465--485.

\bibitem{Soto-Moya:11}
{\sc F.~Soto-Eguibar and H.~Moya-Cessa}, {\em Inverse of the {V}andermonde and
  {V}andermonde confluent matrices}, Appl. Math. Inf. Sci., 5 (2011),
  pp.~361--366.

\bibitem{ST:11}
{\sc W.~Steiner and J.~M. Thuswaldner}, {\em Rational self-affine tiles}.
\newblock Preprint, arXiv:1203.0758v1.

\bibitem{Surer:07}
{\sc P.~Surer}, {\em Characterisation results for shift radix systems}, Math.
  Pannon., 18 (2007), pp.~265--297.

\bibitem{Surer:09}
\leavevmode\vrule height 2pt depth -1.6pt width 23pt, {\em $\varepsilon$-shift
  radix systems and radix representations with shifted digit sets}, Publ. Math.
  (Debrecen), 74 (2009), pp.~19--43.

\bibitem{Thurston:89}
{\sc W.~Thurston}, {\em Groups, tilings and finite state automata}.
\newblock AMS Colloquium Lecture Notes, 1989.

\bibitem{Vivaldi:94}
{\sc F.~Vivaldi}, {\em Periodicity and transport from round-off errors},
  Experiment. Math., 3 (1994), pp.~303--315.

\bibitem{Vivaldi-Vladimirov:03}
{\sc F.~Vivaldi and I.~Vladimirov}, {\em Pseudo-randomness of round-off errors
  in discretized linear maps on the plane}, Internat. J. Bifur. Chaos Appl.
  Sci. Engrg., 13 (2003), pp.~3373--3393.

\bibitem{Wang:00a}
{\sc Y.~Wang}, {\em Self-affine tiles}, in Advances in Wavelet, K.~S. Lau, ed.,
  Springer, 1998, pp.~261--285.

\bibitem{Weitzer:13}
{\sc M.~Weitzer}, {\em Characterization algorithms for shift radix systems with
  finiteness property}, submitted.

\end{thebibliography}

\end{document}